\documentclass[10pt]{amsart}
\usepackage[margin=1.2in]{geometry}
\usepackage{color}
\usepackage{amssymb}

\usepackage{graphicx}
\graphicspath{{./}}
\usepackage{subcaption}
\usepackage{multirow}
\usepackage{bbm}
\usepackage{lipsum}
\usepackage{amsfonts}
\usepackage{graphicx}
\usepackage{epstopdf}
\usepackage{subcaption}
\usepackage{amssymb}
\usepackage{mathtools}
\usepackage{bbm}
\usepackage{booktabs}
\usepackage{algorithm}
\usepackage{algpseudocode}
\usepackage{tikz}
\usetikzlibrary{arrows.meta, positioning}

\definecolor{colorW1Fill}{rgb}{0.8, 0.9, 1.0} 
\definecolor{colorW1Draw}{rgb}{0.2, 0.4, 0.6} 
\definecolor{colorW2Fill}{rgb}{0.8, 1.0, 0.8} 
\definecolor{colorW2Draw}{rgb}{0.2, 0.6, 0.2} 

\newtheorem{theorem}{Theorem}[section]

\newtheorem{proposition}{Proposition}[section]

\theoremstyle{definition}

\theoremstyle{remark}
\newtheorem{remark}[theorem]{Remark}

\numberwithin{equation}{section}

\usepackage{graphicx}

\DeclareGraphicsExtensions{.pdf,.png,.jpg}

\newcommand{\p}{\partial}

\graphicspath{{./figures/}}

\title{An Adaptive Fourier Spectral Method for the Vlasov-Poisson system}

\author[Seung Yeon Cho]{Seung Yeon Cho}
\address[Seung Yeon Cho]{\newline Department of Mathematics, 
		Gyeongsang National University, 52828 Jinju, Republic of Korea}
\email{chosy89@gnu.ac.kr}

\author{Giovanni Russo}
\address[Giovanni Russo]{\newline Department of Mathematics and Computer Science, University of Catania, 95125 Catania, Italy}
\email{giovanni.russo1@unict.it}
\date{\today}

\begin{document}
	\maketitle
	\begin{abstract}
		Numerical simulation of the Vlasov–Poisson system faces fundamental challenges due to phase-space filamentation. Standard spectral methods rely on artificial filtering to suppress errors, which inadvertently degrades physical structures and accuracy over time. This paper proposes a dynamically adaptive Fourier spectral method combined with high-order time splitting to overcome these limitations. To prevent filamentation-induced aliasing, we dynamically expand the wavenumber domain via adaptive zero-padding whenever spectral tail monitoring detects emerging fine-scale structures. Acting as an exact trigonometric interpolation, it inherently avoids the artificial smoothing of traditional filters, preserving mass and effectively maintaining the $L^2$-norm within the dynamically resolved scales. Numerical experiments on Landau damping and various instability problems validate the robustness of our scheme.
	\end{abstract}
	
	\section{Introduction}
The motion of electrons in a self-consistent electric field is governed by the Vlasov–Poisson system, which is written in nondimensional form as
		\begin{align}\label{VP}
			\begin{split}
				&\frac{\partial{f}}{\partial{t}} + v \cdot \nabla_x{f} + \mathbb{E}(x,t) \cdot \nabla_v{f}=0\cr
				&\mathbb{E}(x,t)= -\nabla_x\phi(x,t), \qquad  -\Delta_x \phi(x,t)= \rho(x,t) - m_{\text{ion}},\cr
				&\rho(x,t)=\int_{\mathbb{R}^{d_v}} f(x,v,t)\,dv, \quad m_{\text{ion}}=\frac{1}{L_x^{d_x}}\int_{\mathbb{T}^{d_x}}\int_{\mathbb{R}^{d_v}} f(x,v,t) \,dv\,dx.
			\end{split}
		\end{align}
Here, $f(x,v,t)$ denotes the velocity distribution function in phase space, with $(x,v) \in \Omega := \mathbb{T}^{d_x} \times \mathbb{R}^{d_v}$ for $t > 0$. The spatial domain $\mathbb{T}^{d_x}$ represents the $d_x$-dimensional torus $\mathbb{R}^{d_x}/(L_x\mathbb{Z}^{d_x})$ with period $L_x>0$, while $\mathbb{R}^{d_v}$ denotes the $d_v$-dimensional velocity space. The electric field is denoted by $\mathbb{E}(x,t)$, and $\phi(x,t)$ is the corresponding self-consistent electrostatic potential. The electron charge density is given by $\rho(x,t)$, while $m_{\text{ion}}$ represents the uniformly distributed background ion charge density. Under spatially periodic boundary conditions, global charge neutrality is naturally assumed, meaning that the integral of the right-hand side of the Poisson equation \eqref{VP} over the spatial domain vanishes.

The distribution function $f$ satisfies several fundamental properties. 
First, it obeys a maximum principle: if the initial data satisfies $0 \leq f_0(x,v) \leq M$ for some constant $M > 0$ for all $(x,v)\in \Omega$, then the solution remains bounded as
\[
0 \leq f(x,v,t) \leq M, \quad \forall t>0, \ (x,v)\in \Omega.
\]
Moreover, the total mass of $f$ is conserved over time, i.e.,
\begin{align*}
	\frac{d}{dt}\int_{\mathbb{R}^{d_v}}\int_{\mathbb{T}^{d_x}} f\,dx\,dv = 0.
\end{align*}
In addition, the $L^p$ norm of $f$ is preserved for any $1 \leq p < \infty$, namely,
\[
\|f\|_p = \left(\int_{\mathbb{R}^{d_v}}\int_{\mathbb{T}^{d_x}} |f|^p\,dx\,dv \right)^{1/p}.
\]
The system also conserves the total energy given by
\begin{equation*}
	\text{Energy} = \int_{\mathbb{R}^{d_v}}\int_{\mathbb{T}^{d_x}} \frac{|v|^2}{2} f\,dx\,dv 
	+ \int_{\mathbb{T}^{d_x}} \frac{|\mathbb{E}|^2}{2}\,dx,
\end{equation*}
as well as the entropy
\[
\text{Entropy} = \int_{\mathbb{R}^{d_v}}\int_{\mathbb{T}^{d_x}} f \log \frac{1}{f}\,dx\,dv.
\]

		From a numerical perspective, computing the numerical solution of the Vlasov–Poisson system remains a fundamental challenge in computational plasma physics due to its high dimensionality and the development of fine-scale structures in phase space. In particular, the severe filamentation of the distribution function imposes strict constraints on both computational efficiency and numerical stability, demanding highly accurate resolution in both spatial and velocity variables.
		
		Over the past decades, a wide range of numerical methods have been developed to handle these challenges. Based on how the phase space is discretized, these can be broadly classified into particle-based and grid-based methods. 
		
		Particle approaches, most notably the Particle-In-Cell (PIC) method, evolve macro-particles along characteristic trajectories \cite{BL,DDSV,HE}. Relevant mathematical convergence study can be found in \cite{GV}. One of the most significant advantages of particle methods is their ability to efficiently simulate high-dimensional systems without suffering from the exponential growth of computational cost associated with grid-based methods  \cite{BL}. Furthermore, they are naturally suited for parallel computing and relatively straightforward to implement. However, particle methods are inherently constrained by statistical noise, which fundamentally limits their accuracy in long-time simulations.
		
		In contrast, grid-based methods \cite{AV,FS} discretize the distribution function directly on fixed phase-space grids. Such approaches circumvent statistical noise and can achieve high accuracy. Accordingly, various techniques—finite difference \cite{QS}, finite volume \cite{BH}, and spectral methods \cite{FX}—have been developed.
		Within such a grid-based framework, a popular treatment of the temporal variable is to employ the semi-Lagrangian (SL) approach \cite{CRS, SRBG}. Often coupled with operator splitting techniques \cite{CCFM}, the SL method allows one to bypass restrictive Courant–Friedrichs–Lewy (CFL) conditions by separately treating the convective and drift terms. For applications where splitting errors must be strictly avoided, fully unsplit methods provide a robust alternative \cite{QR}. Additionally, a growing body of literature focuses on conservative SL methods designed to guarantee the exact preservation of physical invariants \cite{CGQ,CBRY}. 
		Alternatively, purely grid-based methods have proven highly effective for handling asymptotic limits of Vlasov system; notably, asymptotic preserving schemes based on micro-macro decompositions have been successfully developed for multiscale scaling limits \cite{CL}. We also refer to a convergence study in \cite{BM}.

		Among the grid-based methods, in this paper, our interest lies on the Fourier spectral methods \cite{CK} where Fourier expansions are applied in both spatial and velocity coordinates together with Strang splitting. The spectral method yields superior convergence rates for smooth solutions and benefits greatly from the computational efficiency of the Fast Fourier Transform (FFT), however, they suffer from two primary sources of numerical degradation, as fundamentally identified by Klimas and Farrell \cite{KF}. The first issue stems from the periodicity assumption inherent to the Fourier transform. If particles accelerate beyond the truncated velocity boundary, they unphysically re-enter from the opposite side of the domain, creating spurious wrap-around effects that severely contaminate the physical solution. The second issue arises from the severe filamentation of the distribution function, which continuously cascades energy into high-wavenumber modes. Once these unresolved modes exceed the Nyquist limit of the grid, they induce severe aliasing errors that rapidly destabilize the simulation. Although a spectral filtering method was utilized in \cite{KF} to secure numerical stability, an over-reliance on such filtering during long-time simulations can inadvertently smear out naturally emerging fine-scale structures, ultimately leading to a substantial loss of accuracy and physical fidelity.

		This highlights a fundamental trade-off between memory constraints and physical accuracy on fixed grids. Unlike filtering techniques that attempt to stabilize the solution within a strict memory budget—often at the cost of physical fidelity—an alternative philosophy is to dynamically supply additional computational resources (i.e., grid points) precisely where they are most needed. Guided by this perspective, various adaptive techniques have been explored in the broader context of the VP system to fully resolve fine-scale structures without artificial smearing. These include adaptive semi-Lagrangian methods \cite{A, SKT}, adaptive rank-based methods \cite{EL, GQ, K,  ZHCQ}, and dynamically scaled Hermite spectral methods \cite{SWW}.
		
		
		Building upon the original work \cite{CK}, we introduce a highly efficient, dynamically adaptive Fourier spectral method combined with high-order time splitting. This approach overcomes both fundamental limitations discussed in \cite{KF}—spurious wrap-around and filamentation-induced aliasing—without relying on artificial filtering.

		To completely prevent spurious velocity wrap-around, we prescribe a sufficiently large, fixed velocity domain $[-v_{\max}, v_{\max}]$ from the onset, ensuring that the distribution function decays to numerically negligible values well before reaching the boundaries. While it is theoretically possible to dynamically widen the velocity boundaries (e.g., via zero-padding in the velocity space) as fast-moving particles accelerate, we deliberately avoid this approach. Enlarging the velocity domain merely adds grid points to empty far-field regions without improving the internal grid resolution (i.e., $\Delta v = 2v_{\max}/N_v$). Since the root cause of aliasing is severe phase-space filamentation—which necessitates a progressively finer grid—expanding the domain boundaries is fundamentally ineffective against this degradation. 
		
		Because pre-allocating a large $v_{\max}$ inherently neutralizes the wrap-around effect, we opt to keep the phase-space boundaries fixed and restrict our adaptive refinement strictly to the spectral domain. Instead of artificially filtering out the high-mode spectral energy caused by filamentation, we introduce an adaptive zero-padding technique in Fourier space. By continuously monitoring the magnitudes of the high-frequency Fourier modes in both the spatial and velocity directions, we dynamically increase the spectral resolution precisely when the threshold indicates the onset of severe filamentation.

%
%
%
		
		This strategy ensures that every dynamically allocated grid point is invested exclusively in enhancing physical resolution by decreasing $\Delta v$. Because this adaptive zero-padding acts as an exact trigonometric interpolation, it actively accommodates the fine-scale filaments with spectral accuracy while avoiding the artificial smoothing inherent to filters. Consequently, it strictly preserves mass and effectively maintains the $L^2$-norm within dynamically resolved scales, without requiring computationally expensive a posteriori constraints. By automatically adjusting the resolution only when and where necessary, the proposed approach completely circumvents the need to prescribe excessively fine grids a priori, achieving exceptional accuracy in capturing complex phase-space structures while remaining highly competitive in computational efficiency and memory usage.

	The outline of this paper is organized as follows: In section 2, we review high-order time splitting methods. Then, in section 3, we describe the spectral method, and propose adaptive refinement techniques based on Fourier modes. Next, in section 4, we present several numerical examples which demonstrate the efficiency of our methods. Finally, in the last section we draw conclusions. 

	\section{High-Order Time Splitting Methods}
	To solve the Vlasov--Poisson (VP) system efficiently, we adopt the operator splitting framework. The fundamental advantage of this approach is that it decouples the VP equation into a sequence of one-dimensional, constant-coefficient linear advection equations. This decoupling allows us to apply highly accurate phase-space discretizations, such as the Fourier spectral method, to each sub-step independently.
	
	In the remainder of this paper, we define the time evolution operators as follows:
	\begin{itemize}
		\item $\mathcal{D}^{\Delta t}$: The drift operator in velocity space over a time step $\Delta t$, which solves
		\begin{align}\label{drift}
			\frac{\p f}{\p t} + E(x,t) \frac{\p f}{\p v} = 0.
		\end{align}
		During this step, the electric field $E$ is completely determined by the spatial density and is treated as constant in time.
		
		\item $\mathcal{T}^{\Delta t}$: The advection operator in physical space over a time step $\Delta t$, which solves
		\begin{align}\label{transport}
			\frac{\p f}{\p t} + v \frac{\p f}{\p x} = 0.
		\end{align}
		Here, the velocity $v$ acts as a constant advection parameter.
	\end{itemize}
	By alternating these two operators, the electric field $E$ is naturally updated via the Poisson equation before each velocity drift step $\mathcal{D}$. Below, we review three time-splitting methods of increasing formal accuracy.
	
	\subsection{2nd-Order Strang Splitting}
	The foundational time integration is performed using the well-known second-order Strang splitting scheme. The operator $\mathcal{S}^{\Delta t}$, which advances the solution from $t_n$ to $t_{n+1} = t_n + \Delta t$, is symmetrically decomposed as:
	\begin{equation*}
		\mathcal{S}^{\Delta t} := \mathcal{D}^{\Delta t/2} \circ \mathcal{T}^{\Delta t} \circ \mathcal{D}^{\Delta t/2}.
	\end{equation*}
	Due to its symmetric formulation, this method achieves $\mathcal{O}(\Delta t^2)$ accuracy and preserves the time-reversibility of the underlying Hamiltonian system.
	
	\subsection{4th-Order Yoshida Splitting}
	To achieve higher temporal accuracy, we employ the fourth-order Yoshida splitting method. This scheme is constructed by composing three consecutive Strang splitting steps with specifically chosen fractional time increments. The operator $\mathcal{Y}^{\Delta t}$ is defined as:
	\begin{equation*}
		\mathcal{Y}^{\Delta t} := \mathcal{S}^{x_1\Delta t} \circ \mathcal{S}^{x_0\Delta t} \circ \mathcal{S}^{x_1\Delta t},
	\end{equation*}	
	where the fractional weights are given by $x_0 = -\frac{2^{1/3}}{2-2^{1/3}}$ and $x_1 = \frac{1}{2-2^{1/3}}$. Although $x_0$ is negative, which implies a backward-in-time integration step, it does not induce numerical instability in our framework because the spectral advection solves \eqref{drift} and \eqref{transport} exactly and reversibly.
	
	\subsection{6th-Order Splitting Method}
	For highly accurate, long-term simulations, standard splitting methods require an impractically large number of stages. To circumvent this, we implement a highly optimized 6th-order, 11-stage splitting method tailored specifically for the Vlasov--Poisson system \cite{CCFM}. By exploiting the specific commutator structures of the VP equations, this method significantly reduces the splitting error. The update operator $\mathcal{M}^{\Delta t}$ is given by:
	\begin{align}\label{meren}
		\mathcal{M}^{\Delta t} := \mathcal{D}^{\Delta t g_1} \mathcal{T}^{a_1\Delta t} \mathcal{D}^{\Delta t g_2} \mathcal{T}^{a_2\Delta t} \mathcal{D}^{\Delta t g_3} \mathcal{T}^{a_3\Delta t} \mathcal{D}^{\Delta t g_3} \mathcal{T}^{a_2\Delta t} \mathcal{D}^{\Delta t g_2} \mathcal{T}^{a_1\Delta t} \mathcal{D}^{\Delta t g_1},
	\end{align}
	where the drift step coefficients $g_i$ incorporate modified forcing terms to achieve 6th-order accuracy with fewer stages:
	\begin{equation*}
		g_i = b_i + 2c_im_{\text{ion}} (\Delta t)^2 + 4d_im_{\text{ion}}^2 (\Delta t)^4 - 8e_im_{\text{ion}}^3 (\Delta t)^6, \quad i=1,2,3.
	\end{equation*}
	The constant $m_{\text{ion}}$ is determined by the ion charge density \eqref{VP}, and the precise values of the constants $a_i, b_i, c_i, d_i,$ and $e_i$ are provided in Table \ref{coeff meren}. This specific combination guarantees strictly time-reversible integration up to machine precision, making it an ideal choice for testing our adaptive spectral methods without interference from temporal truncation errors.
	
%
%
%

	\begin{table}[ht]
		\centering
		{\begin{tabular}{|c|c|c|c|}
				\hline
				\multicolumn{1}{ |c }{}&
				\multicolumn{1}{ |c }{$i=1$}& \multicolumn{1}{ |c  }{$i=2$} & \multicolumn{1}{ |c|  }{$i=3$}  \\
				\hline	
				\multicolumn{1}{ |c|  }{$a_i$}&{\small $0.168735950563437422448196$}        &{\small$0.377851589220928303880766$} 
				&{\small$-0.093175079568731452657924$}
				\\
				\multicolumn{1}{ |c|  }{$b_i$}&{\small$0.049086460976116245491441$} &{\small$0.264177609888976700200146$}
				&{\small$0.186735929134907054308413$}
				\\
				\multicolumn{1}{ |c|  }{$c_i$}&{\small$-0.000069728715055305084099$}
				&{\small$-0.000625704827430047189169$}
				&{\small$-0.002213085124045325561636$}
				\\
				\multicolumn{1}{ |c|  }{$d_i$}&
				{\small $0$} &{\small$-2.916600457689847816445691 \cdot 10^{-6}$} &{\small$3.048480261700038788680723 \cdot 10^{-5}$}
				\\
				\multicolumn{1}{ |c|  }{$e_i$}&
				{\small $0$}   & {\small $0$}  &{\small $4.985549387875068121593988 \cdot 10^{-7}$}
				\\
				\hline
				\hline
		\end{tabular}}
		\caption{Coefficients in \eqref{meren}.}\label{coeff meren}
	\end{table}

\section{Fourier Spectral Method with Adaptive Refinement}
In this section, we detail the proposed Fourier spectral method, augmented with an adaptive refinement strategy, for solving the Vlasov--Poisson system.

Since the Fourier spectral approach fundamentally requires periodic boundary conditions, we truncate the inherently unbounded velocity domain to a finite interval $[-v_{\max}, v_{\max}]$. The boundary $v_{\max}$ is chosen to be sufficiently large so that the distribution function $f$ decays to numerically negligible values at the boundaries. As long as the effective support of $f$ remains strictly confined within this truncated domain, the periodic extension accurately preserves the physical solution. Consequently, the Fourier representation can be applied to the velocity variable without introducing significant truncation errors or spurious wrap-around effects.

It is worth noting that the assumption of periodicity in the velocity variable is a standard and widely adopted technique in the development of spectral methods for broader kinetic equations, including the Boltzmann, Landau--Fokker--Planck, and Vlasov equations. This approach was originally introduced in \cite{PP} and subsequently analyzed for the homogeneous Boltzmann equation with various collision kernels \cite{PR}, before being extended to inhomogeneous cases \cite{FR}. Since then, numerous fast spectral methods based on this periodic framework have been extensively studied \cite{FMP, MP}. Similar velocity-space truncations and periodic assumptions have also been successfully applied to the Landau--Fokker--Planck equation \cite{PRT}.

\subsection{Fourier Spectral Method}Now, we describe the Fourier spectral method in one dimension. Let the phase space $(x,v) \in [0,L_x] \times [-v_{\max},v_{\max}]$ be discretized by uniform grids:
\begin{align*}
    x_l &= l \Delta x, \quad l = 0,\dots,N_x-1,\\
    v_m &= -v_{\max} + m \Delta v,\quad m = 0,\dots,N_v-1
\end{align*} 
where $\Delta x = L_x / N_x$, $\Delta v = L_v / N_v$ with $L_v = 2v_{\max}$. We denote by $f_{l,m}(t)$ the numerical approximation of the exact distribution function at the grid point $(x_l, v_m)$, i.e., $f_{l,m}(t) \approx f(x_l, v_m, t)$. 

Employing the discrete Fourier transform (DFT), the discrete distribution function $f_{l,m}$ can be mapped to its spectral representations in either the spatial or velocity domain. We denote the Fourier coefficients uniformly by $\hat{f}$, where the transformed variable is explicitly indicated by its corresponding wavenumber index: $\hat{f}_{k_x, m}$ (with respect to $x$) and $\hat{f}_{l, k_v}$ (with respect to $v$).
On the grid the truncated Fourier series in the spatial and velocity directions are written, respectively, as follows:
\begin{align*}
	f_{l,m}(t) &= \sum_{k_x=-N_x/2}^{N_x/2-1} \hat{f}_{k_x, m}(t) e^{i \alpha_{k_x} x_l},\\
	f_{l,m}(t) &= \sum_{k_v=-N_v/2}^{N_v/2-1} \hat{f}_{l, k_v}(t) e^{i \beta_{k_v} v_m},\end{align*} 
where the discrete wavenumbers are defined as $\alpha_{k_x} = 2\pi\frac{k_x}{L_x}$ and $\beta_{k_v} = 2\pi\frac{k_v}{L_v}$. The corresponding discrete Fourier coefficients $\hat{f}_{k_x, m}$ and $\hat{f}_{l, k_v}$ are computed via the forward DFTs over the respective grid points:
\begin{align*}
	\hat{f}_{k_x, m}(t) &= \frac{1}{N_x} \sum_{l=0}^{N_x-1} f_{l,m}(t) e^{-i \alpha_{k_x} x_l},\\
	\hat{f}_{l, k_v}(t) &= \frac{1}{N_v} \sum_{m=0}^{N_v-1} f_{l,m}(t) e^{-i \beta_{k_v} v_m}.
\end{align*}
The Fourier spectral method is based on operator splitting, where the Vlasov--Poisson system is decoupled into a transport step (spatial advection) and a drift step (velocity advection), solved alternately. At each substep, Fast Fourier Transforms (FFT) and inverse FFTs (IFFT) are applied in the corresponding variable to exactly evaluate spatial or velocity derivatives in the spectral domain.

In the \textbf{transport step} $\mathcal{T}^{\Delta t}$, governed by $\partial_t f + v \partial_x f = 0$, the distribution function is advanced by applying an exact phase shift in the $x$-directional Fourier space. For each discrete velocity level $v_m$, the advection acts simply as a multiplier to the Fourier coefficients:$$\hat{f}_{k_x, m}(t+\Delta t) = \hat{f}_{k_x, m}(t)\, e^{-i \alpha_{k_x} v_m \Delta t}.$$Subsequently, an inverse DFT in $x$ is applied to return the updated solution to the physical grid.

For the \textbf{drift step} $\mathcal{D}^{\Delta t}$, governed by $\partial_t f + E(x,t) \partial_v f = 0$, we first compute the macroscopic density on the grid, $\rho_l = \sum_{q=0}^{N_v-1} f_{l,m} \Delta v$, to update the electric field via the Poisson equation. Since the electric field relates to the electrostatic potential by $E = -\partial_x \phi$, we solve $\partial_x E = \rho - 1=:\eta$. Under the assumption of periodic boundary conditions, we solve this directly in the spatial Fourier space. The DFT of the Poisson equation yields an explicit expression for each Fourier mode:$$i \alpha_{k_x} \widehat{E}_{k_x} = \widehat{\eta}_{k_x}, \quad \text{for } k_x \neq 0,$$
where the zero mode $\widehat{E}_{0}$ is fixed to zero due to the neutrality condition. The discrete electric field $E_l$ in physical space is then retrieved via the IFFT. Finally, in the drift step, a forward DFT is performed with respect to $v$. For each spatial grid point $x_l$, the solution is advanced through a phase shift driven by the newly computed local electric field $E_l$:$$\hat{f}_{l, k_v}(t+\Delta t) = \hat{f}_{l, k_v}(t)\, e^{-i \beta_{k_v} E_l \Delta t},$$which is followed by a final IFFT in $v$ back to the phase space, thereby completing one full time step. In summary, each splitting cycle involves repeated FFT/IFFT operations across the $x$ and $v$ dimensions. This structure enables highly efficient computation while maintaining spectral accuracy. 

\subsection{Adaptive refinement in Fourier space}	
When solving VP system with Fourier spectral methods, the distribution function in phase space starts to be distorted as the aliasing errors from high-wavenumber affect the lower wave region. The main source of such errors come from insufficient number of grid points in wavenumber space. To overcome this, we employ an adaptive mesh refinement (AMR) algorithm that operates by monitoring the magnitude of Fourier modes. Our strategy is to dynamically detect the emergence of high-wave components, which are closely related to aliasing errors, and accordingly enlarge the computational domain in Fourier space, which leads to the refinement of the mesh for physical space after IFFT operations. At the end of each time step, we performe the refinement procedures independently in both the spatial ($x$) and velocity ($v$) directions. 

In the following description, we only explain the refinement procedure for $v$ direction. By symmetry, the same approach can be applied to the $x$ direction. 


\subsubsection{Method 1: Two-Window Spectral Comparison}	
To check the emergence of high-wavenumber components in the chosen direction, our first strategy is to consider two disjoint windows $W_{\text{inner}}$ and $W_{\text{outer}}$ in the Fourier space, both located near the high-wavenumber region:
\[
W_{\text{inner}} = [k_{\text{inner},L}, k_{\text{inner},R}], \quad W_{\text{outer}} = [k_{\text{outer},L}, k_{\text{outer},R}],
\]
with $W_{\text{outer}}$ positioned physically closer to the maximum resolved wavenumber than $W_{\text{inner}}$ (see Figure \ref{fig:spectral_windows}).
	\begin{figure}[htbp]
		\centering
		\begin{tikzpicture}[
			window/.style={rectangle, thick, rounded corners=2pt, minimum height=1cm, anchor=south},
			arrow/.style={-{Stealth[scale=0.8]}, thin, dashed},
			axis label/.style={below=3pt, font=\small},
			boundary label/.style={below=6pt, font=\footnotesize},
			]
			
			\draw [->, >=Stealth, thick] (-1,0) -- (11,0) node[right, font=\large] {$k$};

			\draw [thick, gray!60, smooth] (0, 3.5) .. controls (1, 0.5) and (5, 0.2) .. (10, 0.05) 
			node[above left, font=\small, text=gray!80] {$|\hat{f}_{l,k_v}|$};

			\draw (0,-0.1) -- (0,0.1); 
			\draw (10,-0.1) -- (10,0.1); 
			
			\node (originLabel) at (0,-0.2) [axis label] {$0$};
			\node (kmaxLabel) at (10,-0.2) [axis label] {$k_{max}$};

			\draw [dashed, thin, gray] (3, 0.1) -- (3, 1.2); 
			\draw [dashed, thin, gray] (5, 0.1) -- (5, 1.2);
			\draw [dashed, thin, gray] (7, 0.1) -- (7, 1.2); 
			\draw [dashed, thin, gray] (9, 0.1) -- (9, 1.2);
			
			\filldraw [fill=colorW1Fill, draw=colorW1Draw, window] (3, 0.8) rectangle (5, 1.2)
			node [midway, font=\large\bfseries] {$W_{\text{inner}}$};
			
			\filldraw [fill=colorW2Fill, draw=colorW2Draw, window] (7, 0.8) rectangle (9, 1.2)
			node [midway, font=\large\bfseries] {$W_{\text{outer}}$};

			\draw (3, -0.1) -- (3, 0.1); 
			\node (k1L) at (2.8, -0.6) [boundary label] {$k_{\text{inner},L}$};
			\draw[arrow] (k1L.north) -- (3, -0.05);
			
			\draw (5, -0.1) -- (5, 0.1); 
			\node (k1R) at (5.2, -0.6) [boundary label] {$k_{\text{inner},R}$};
			\draw[arrow] (k1R.north) -- (5, -0.05);
			
			\draw (7, -0.1) -- (7, 0.1); 
			\node (k2L) at (6.8, -0.6) [boundary label] {$k_{\text{outer},L}$};
			\draw[arrow] (k2L.north) -- (7, -0.05);
			
			\draw (9, -0.1) -- (9, 0.1); 
			\node (k2R) at (9.2, -0.6) [boundary label] {$k_{\text{outer},R}$};
			\draw[arrow] (k2R.north) -- (9, -0.05);

			\draw [thin, colorW2Draw, dashed] (10, 0.1) -- (10, 1.8);
			
		\end{tikzpicture}
		\caption{Two sampling windows $W_{\text{inner}}$ and $W_{\text{outer}}$ in the high-wavenumber region.
		}
		\label{fig:spectral_windows}
	\end{figure}
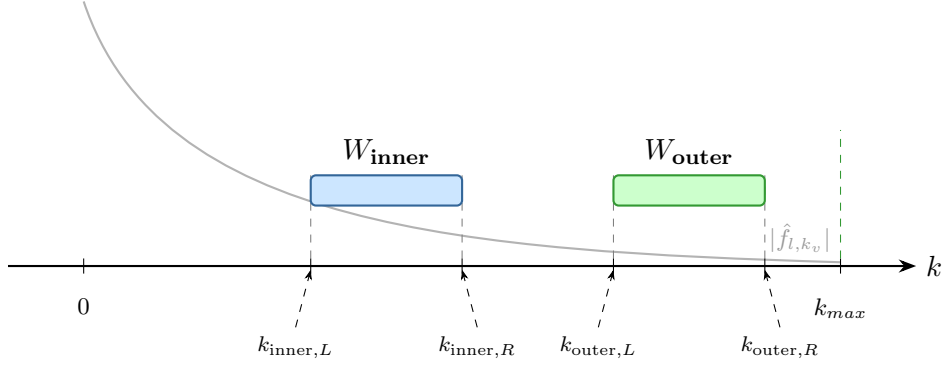
Next, we compute the root-mean-square (RMS) magnitude of the Fourier coefficients within each window for $l=0,...,N_x-1$:
\[
M_{\text{inner}}^l = \sqrt{\frac{1}{|W_{\text{inner}}|} \sum_{k_v \in W_{\text{inner}}} |\hat{f}_{l,k_v}|^2}, 
\quad
M_{\text{outer}}^l = \sqrt{\frac{1}{|W_{\text{outer}}|} \sum_{k_v \in W_{\text{outer}}} |\hat{f}_{l,k_v}|^2}.
\]
where $|A|$ is the number of elements in the set $A$. For each $l$, the quantity $M_{\text{inner}}^l$ represents the baseline magnitude in a moderately high-wavenumber region, while $M_{\text{outer}}^l$ captures the energy near the highest resolved modes. If significant energy reaches the inner window before fully polluting the outer window, i.e.,
\[
M_{\text{inner}}^l > C \, M_{\text{outer}}^l,
\]
for a prescribed constant $C > 1$, this condition serves as an early indicator that the current grid resolution will soon become insufficient. 
We evaluate this criterion across all spatial indices $l$. If the condition is triggered for any index $l$, we globally enlarge the computational wavenumber domain by increasing the resolution by a factor of $R > 1$, and apply zero-padding to the extended modes. This procedure effectively broadens the range of resolvable Fourier modes, accommodating the emerging fine-scale structures and robustly delaying the onset of aliasing.

The performance of the proposed adaptive refinement strategy is governed by three hyperparameters: (1) the size and location of the monitoring windows $W_{\text{inner}}$ and $W_{\text{outer}}$, (2) the threshold constant $C$ controlling sensitivity to high-wavenumber growth, (3) the refinement factor $R$ determines the extent of resolution refinement. Proper tuning of these parameters is essential to balance computational efficiency and accuracy, especially in problems exhibiting strong filamentation where high-wavenumber modes develop rapidly over time.

In this study, we adopt a specific set of hyperparameters based on numerical heuristic to ensure robust performance. The monitoring windows are positioned near the Nyquist limit $k_{max}$ to detect imminent aliasing, defined as $W_{\text{inner}} = [0.8k_{max}, 0.9k_{max}]$ and $W_{\text{outer}} = [0.9k_{max}, k_{max}]$. Placing $W_{\text{inner}}$ slightly away from the boundary provides a buffer zone that allows for early detection of high-wavenumber growth before it contaminates the entire spectrum. Furthermore, by defining $W_{\text{outer}}$ at the very edge of the resolved spectrum, we can sensitively monitor the flatness of the spectral tail, which is a key indicator of resolution loss. The threshold constant is set to $C = 10$, requiring the high-frequency tail to become relatively flat (i.e., within one order of magnitude) before triggering refinement, thus preventing premature domain expansion due to transient noise. Finally, the refinement factor is chosen as $R = 1.2$. 

\begin{remark}
	To ensure that the computational efficiency of the FFT algorithm is not compromised by non-power-of-two grid sizes, the newly expanded grid number is aggressively snapped to the nearest multiple of 128. This dynamic snapping preserves the mixed-radix optimization of the FFTW library of MATLAB while significantly extending the simulation's lifespan under memory constraints.
\end{remark}

\subsubsection{Method 2: Relative RMS Tail Indicator}
To provide a more global assessment of spectral degradation, an alternative approach evaluates the relative noise floor of the high-wavenumber region against the dominant macroscopic wave. Instead of comparing two local windows as in Method 1, we define a single boundary window, $W_{boundary}$, positioned at the outermost edge of the resolved spectrum. 

In the second approach, we compute the root-mean-square (RMS) magnitude of $\hat{f}_{l,k_v}$ over all indices corresponding to the boundary region. Letting $N_{boundry}$ be the number of wavenumbers in $W_{boundary}$, the RMS tail magnitude $\bar{E}_{tail}$ and the global peak magnitude $E_{peak}$ are defined as:$$\bar{E}_{tail} = \sqrt{ \frac{1}{N_{boundry} N_x}\sum_{l=0}^{N_x-1} \sum_{k_v \in W_{boundary}}  |\hat{f}_{l,k_v}|^2 },$$
$$E_{peak} = \max_{l, k_v} |\hat{f}_{l,k_v}|.$$
The ratio of these two quantities effectively represents the noise-to-signal ratio across the mixed phase space. If the relative magnitude of the tail exceeds a prescribed tolerance threshold, i.e.,$$\frac{\bar{E}_{tail}}{E_{peak}} > \tau_{tol},$$the computational grid is adaptively expanded. Unlike Method 1, which relies solely on the local shape and flatness of the tail, this indicator evaluates the absolute emergence of high-frequency white noise relative to the physical signal. This double summation ensures that the tail fluctuation is averaged over both the high-wavenumber spectrum and the entire untransformed physical domain, making the trigger highly robust against transient, low-amplitude floating-point artifacts. For our numerical implementation, the boundary window is set to $W_{boundary} = [k_{\text{boundary}}, k_{max}]$ with $k_{\text{boundary}}=0.8k_{max}$, and the tolerance is chosen as $\tau_{tol} = 10^{-6},10^{-9},10^{-12}$. Furthermore, to mitigate explosive memory consumption, adopt a fractional refinement factor of $R = 1.2$ as in the previous approach.

\begin{figure}[tbp]
	\centering
	\begin{tikzpicture}[
		window/.style={rectangle, thick, rounded corners=2pt, minimum height=1cm, anchor=south},
		arrow/.style={-{Stealth[scale=0.8]}, thin, dashed},
		axis label/.style={below=3pt, font=\small},
		boundary label/.style={below=6pt, font=\footnotesize},
		]
		
		\definecolor{colorBndryFill}{RGB}{230, 240, 255}
		\definecolor{colorBndryDraw}{RGB}{50, 100, 200}
		\definecolor{colorPeak}{RGB}{220, 20, 60}
		
		\draw [->, >=Stealth, thick] (-1,0) -- (11,0) node[right, font=\large] {$k_v$};
		
		\draw (0,-0.1) -- (0,0.1); 
		\draw (10,-0.1) -- (10,0.1); 
		
		\node (originLabel) at (0,-0.2) [axis label] {$0$};
		\node (kmaxLabel) at (10,-0.2) [axis label] {$k_{max}$};

		\draw [thick, gray!60, smooth] (0, 3.5) .. controls (1, 0.5) and (5, 0.2) .. (10, 0.05) 
		node[above left, font=\small, text=gray!80] {$|\hat{f}_{l,k_v}|$};
		
		\draw [dashed, colorPeak, thick] (0, 0) -- (0, 3.5);
		\filldraw [colorPeak] (0, 3.5) circle (2pt);
		\node [right, colorPeak, font=\small\bfseries] at (0.2, 3.3) {$\displaystyle E_{peak} = \max_{\forall l, \forall k_v} |\hat{f}_{l,k_v}|$};

		\draw [dashed, thin, gray] (7, 0.1) -- (7, 1.2); 
		\draw [dashed, thin, gray] (10, 0.1) -- (10, 1.2);
		
		\filldraw [fill=colorBndryFill, draw=colorBndryDraw, window] (7, 0.8) rectangle (10, 1.5)
		node [midway, font=\large\bfseries] {};
		
		\node [above=0.1cm, font=\small\bfseries, color=colorBndryDraw] at (8.5, 1.5) {RMS: $\bar{E}_{tail}$};
		
		\node [above=0.1cm, font=\small\bfseries, color=colorBndryDraw] at (8.5, 0.75) {$W_{boundary}$};

		\draw (7, -0.1) -- (7, 0.1); 
		\node (kL) at (6.8, -0.6) [boundary label] {$k_{\text{boundary}}$};
		\draw[arrow] (kL.north) -- (7, -0.05);
		
		\draw[arrow] (9.8, -0.45) -- (10, -0.05);
		
	\end{tikzpicture}
\caption{Spectral domain configuration for Method 2. The root-mean-square (RMS) magnitude $\bar{E}_{tail}$ is calculated within the high-wavenumber boundary window $W_{boundary}$ and compared against the global peak $E_{peak}$.}
	\label{fig:method2_window}
\end{figure}
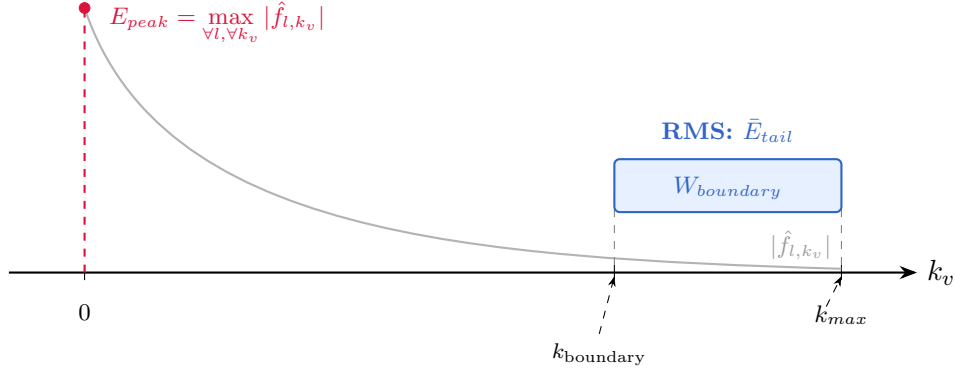

	In the following proposition, we show that the adaptive refinement operation based on zero-padding in frequency space maintains the mass and $L^2$-norm of $f$. In the proposition, we denote the numerical solution at time $t^n$ by $f_{l,m}^n \approx f(x_l, v_m, t^n)$, and define the discrete total mass of $f$ at time $n$ by
	\begin{align}\label{discrete total mass}
		\text{Mass}(f^n) = \Delta x \Delta v \sum_{l,m} f_{l,m}^n.
	\end{align}
	Similarly, the discrete $L^p$ norm of $f$ at time $n$ for $=1,2$ is given by
	\begin{align}\label{discrete Lp}
		\|f^n\|_p = \left( \Delta x \Delta v \sum_{l,m} |f_{l,m}^n|^p \right)^{1/p}.
	\end{align}
	Note that $\|f^n\|_1\geq \text{Mass}(f^n)$, where the equality holds if and only if $f_{l,m}^n\geq 0$ for all $l,\,m$. 
\begin{proposition}\label{Prop 1}
	Let $f^n_{l,m}$ be the discrete distribution function represented on a uniform phase-space grid of size $N_x \times N_v$ at a fixed time step $n$, where $l$ and $m$ denote the spatial and velocity indices, respectively. Let the adaptively refined distribution function $f^{ext,n}_{l, m_{ext}}$ on a finer extended grid of size $N_x \times M_v$ (with $M_v > N_v$, where $m_{ext}$ denotes the extended velocity index) be obtained by zero-padding the high-wavenumber Fourier modes of $f^n_{l,m}$ along the velocity direction for each fixed spatial index $l$. Then, this one-dimensional adaptive refinement operation exactly preserves the total discrete mass and the global $L^2$-norm of the distribution function over the entire phase space. By symmetry, the exact preservation also holds for refinement along the spatial direction.
\end{proposition}

\begin{proof}
	Without loss of generality, we consider the zero-padding operation applied strictly along the velocity variable $v$. For any fixed spatial grid point $x_l$, let $\hat{f}^n_{l, k_v}$ denote the 1D discrete Fourier transform of $f^n_{l,m}$ with respect to $v$. The total discrete mass over the entire phase space is given by the double summation:
	$$ \text{Mass}(f^n) = \sum_{l=0}^{N_x-1} \left( \sum_{m=0}^{N_v-1} f^n_{l,m} \Delta v \right) \Delta x =  \sum_{l=0}^{N_x-1} \left(L_v \hat{f}^n_{l, 0}\right) \Delta x, $$
	where $L_v=N_v \Delta v$ is the length of the velocity domain. The adaptive zero-padding operation maps the original 1D Fourier coefficients $\hat{f}^n_{l, k_v}$ to a new extended array $\hat{f}^{ext,n}_{l, k_v}$ of size $M_v$ such that:
	$$ \hat{f}^{ext,n}_{l, k_v} = \begin{cases} \hat{f}^n_{l, k_v}, & \text{for } -N_v/2 \leq k_v \leq N_v/2-1, \\ 0, & \text{otherwise.} \end{cases} $$
	Since the zero-th frequency mode remains strictly identical ($\hat{f}^{ext,n}_{l, 0} = \hat{f}^n_{l, 0}$ for all $l$) and the physical domain size $L_v$ is unchanged, the global total mass is exactly conserved: 
	$$\text{Mass}(f^{ext,n}) = \sum_{l=0}^{N_x-1} \left(L_v \hat{f}^{ext,n}_{l, 0}\right) \Delta x = \sum_{l=0}^{N_x-1} \left(L_v \hat{f}^{n}_{l, 0}\right) \Delta x = \text{Mass}(f^n).$$
	
	Furthermore, by applying Parseval's theorem to the 1D transform along the $v$-direction, the global discrete $L^2$-norm squared is evaluated as:
	$$ ||f^n||_{2}^2 = \sum_{l=0}^{N_x-1} \left( \sum_{m=0}^{N_v-1} |f^n_{l,m}|^2 \Delta v \right) \Delta x =  \sum_{l=0}^{N_x-1} \left(L_v\sum_{k_v=-N_v/2}^{N_v/2-1} |\hat{f}^n_{l, k_v}|^2\right) \Delta x. $$
	Again, since the zero-padding operation appends only zeros to the high-frequency spectrum, the inner sum over $k_v$ is trivially identical for every spatial index $l$:
	$$ ||f^{ext,n}||_{2}^2 =  \sum_{l=0}^{N_x-1} \left(L_v \sum_{k_v=-M_v/2}^{M_v/2-1} |\hat{f}^{ext,n}_{l, k_v}|^2\right) \Delta x = \sum_{l=0}^{N_x-1} \left(L_v\sum_{k_v=-N_v/2}^{N_v/2-1} |\hat{f}^n_{l, k_v}|^2\right) \Delta x = ||f^n||_{2}^2. $$
	This completes the proof.
\end{proof}

Next, we state that the Fourier spectral method preserves the discrete total mass and $L^2$-norm.

\begin{proposition}\label{Prop 2}
	For any arbitrary discrete distribution function $f$, the semi-Lagrangian advection steps computed via the FFT-based exact phase shifts exactly conserve both the discrete total mass and the discrete $L^2$-norm up to machine precision.
\end{proposition}
\begin{proof}
	We demonstrate the strict conservation by detailing the transport step; the proof for the drift step follows by exact symmetry. Let $f^n_{l,m}$ denote the distribution function at the $n$-th time step. Recall that the transport equation $\partial_t f + v \partial_x f = 0$ is solved in the $x$-Fourier space with the exact phase shift:
	$$\hat{f}^*_{k_x, m} = \hat{f}^n_{k_x, m} e^{-i \alpha_{k_x} v_m \Delta t},$$
	where $\alpha_{k_x} = 2\pi\frac{k_x}{L_x}$. Noting that the discrete total mass is determined solely by the zero-th mode ($k_x = 0$), it is straightforward to have
	\begin{align*}
		\text{Mass}(f^*) &= \sum_{m=0}^{N_v-1} \left( L_x \hat{f}^*_{0,m} \right) \Delta v \\
		&= \sum_{m=0}^{N_v-1} \left( L_x \hat{f}^n_{0,m} \right) \Delta v \\
		&= \sum_{m=0}^{N_v-1} \left( \sum_{l=0}^{N_x-1} f_{l,m}^n\Delta x \right) \Delta v = \text{Mass}(f^n),
	\end{align*}
	where $f^*$ is discrete solutions to the transport equation obtained by IDFT of $\hat{f}^*_{k_x, m}$.
	Applying Parseval's identity along the $x$-direction, we have
	\begin{align*}
		||f^*||_2^2 &= \sum_{m=0}^{N_v-1} \left(L_x \sum_{k_x=-N_x/2}^{N_x/2-1} |\hat{f}^*_{k_x, m}|^2\right) \Delta v\\ 
		&= \sum_{m=0}^{N_v-1} \left(L_x \sum_{k_x=-N_x/2}^{N_x/2-1} |\hat{f}^n_{k_x, m}|^2\right) \Delta v\\ 
		&= \sum_{m=0}^{N_v-1} \left(\sum_{l=0}^{N_x-1} |f_{l, m}^n|^2 \Delta x \right) \Delta v\\
		&=||f^n||_2^2.
	\end{align*} 
This confirms the exact algebraic conservation, completing the proof.
\end{proof}

\begin{remark}
	Proposition \ref{Prop 1} establishes the exact conservation of the discrete invariants during the adaptive grid expansion (zero-padding), while Proposition \ref{Prop 2} guarantees their strict preservation during the semi-Lagrangian advection steps. By combining these two results, it is proven that the proposed fully adaptive spectral method strictly preserves the discrete total mass and $L^2$-norm throughout the entire time evolution prior to reaching the resolution limit.	
\end{remark}
\begin{remark}
	In the following sections, we present several numerical tests where the simulations with adaptive techniques reach the hardware memory limit at a certain time $t_m>0$. Beyond this critical point, the depletion of available computational resources restricts any further grid expansion. Consequently, the severe phase-space filamentation pushes energy into unresolved high-frequency modes that exceed the maximum spectral capacity. This inevitable loss of high-frequency information is reflected as a gradual decay in the $L^2$-norm ($||f||_2 < ||f_0||_2$) beyond $t_m$. In contrast, as mathematically guaranteed by the numerical formulation, the discrete total mass remains strictly invariant up to machine precision throughout the entire simulation.
\end{remark}

	\section{Numerical tests}

	In this section, we present several numerical tests to demonstrate the performance and accuracy of our proposed method. The numerical solutions are computed by combining high-order splitting schemes in time with Fourier spectral discretization in phase space. All simulations were performed using a single NVIDIA RTX 3060 GPU with MATLAB. For brevity, throughout the subsequent examples, we introduce the following notation:
	$$f_M(v; v_{th}) = \frac{1}{\sqrt{2\pi v_{th}^2}} e^{-\frac{v^2}{2v_{th}^2}},$$
	where $v, v_{th} \in \mathbb{R}$ and $v_{th}$ denotes the thermal velocity.
	
	To evaluate the conservation properties of the numerical scheme, we define the corresponding discrete macroscopic invariants on the computational grid. Following the notation in Section 3, we let $x_l$ and $v_m$ denote the spatial and velocity grid points with uniform grid cell spacings $\Delta x$ and $\Delta v$, respectively. The numerical solution at $(x_l,v_m)$ at time $t^n$ is denoted by $f_{l,m}^n$, and the associated discrete electric field is denoted by $\mathbb{E}_l^n$. Recalling the definition of total mass \eqref{discrete total mass} and discrete $L^p$ norm ($p=1,2$) in \eqref{discrete Lp}, we define the discrete total energy by
	\begin{equation*}
		\text{Energy}^n = \Delta x \Delta v \sum_{l,m} \frac{|v_m|^2}{2} f_{l,m}^n 
		+ \Delta x \sum_{l} \frac{|\mathbb{E}_l^n|^2}{2}.
	\end{equation*}
	The discrete entropy is given by
	$$
	\text{Entropy}^n = \Delta x \Delta v \sum_{l,m,\, f_{l,m}^n > 0} f_{l,m}^n \log \frac{1}{f_{l,m}^n},
	$$
	where the summation is strictly restricted to indices $(l,m)$ satisfying $f_{l,m}^n > 0$ to avoid numerical singularities.
	
	Throughout this section, we refer to the two methods introduced in Section 3 as AMR1 and AMR2. The Fourier spectral method with fixed grids is referred to as Non-AMR, for which the number of grid points, $N_x$ and $N_v$, are set to match the maximum resolution attained by the AMR methods before reaching the memory limit in each test case.

		\subsection{Weak Landau Damping}
		We start from the weak Landau damping test which provides a standard benchmark where the theoretical decay behavior is well understood (e.g. see \cite{AV,CBRY,FSB,QR,CGQ,ZHCQ}).  The initial distribution function is given by the following perturbed equilibrium:
		\begin{equation}\label{init weak}
			f_0(x, v) =f_M(v;v_{th}) (1 + \alpha \cos(kx)),
		\end{equation}
		with $\alpha = 0.01$ and $k = 0.5$, which yields a moderate damping rate that is suitable for numerical verification. 
		The spatial domain has length $L = 4\pi$, ensuring periodicity consistent with the chosen wavenumber $k$, while the velocity domain is truncated at $v_{max} = 9$ which is sufficiently large to capture the support of the distribution function. The thermal velocity is set to $v_{th} = 1$.
		
        \subsubsection{Accuracy Test}
		To rigorously verify the temporal accuracy of our implemented splitting schemes (2nd, 4th, and 6th order), we conducted an accuracy test focusing on the relative $L^1$ error of the distribution function $f$. To isolate the temporal error from the spatial and velocity discretization errors, we fixed the grid resolution at a sufficiently high level, i.e., $N_x = 32$ and $N_v = 2816$, ensuring that the spatial error is negligible compared to the temporal error. The reference solution, $f_{\text{ref}}$, was computed using the 6th-order splitting method with a highly refined time step $\Delta t = 1/2^6$. The simulation was run up to a final time $T = 50$. We then computed the numerical solutions for the 2nd-, 4th-, and 6th-order methods using time steps $\Delta t \in \{2, 1, 0.5, 0.25, 0.125\}$. The relative $L^1$ error was evaluated at the final time $T=50$ as:
		$$\text{Relative $L^1$ error} = \frac{||f_{\Delta t} - f_{\text{ref}}||_{L^1}}{||f_{\text{ref}}||_{L^1}},$$
		where $f_{\Delta t}$ denotes the numerical solution at the final time $T=50$ obtained with a fixed time step $\Delta t$.	
		\begin{figure}[htbp]
			\centering
			\includegraphics[width=0.5\linewidth]{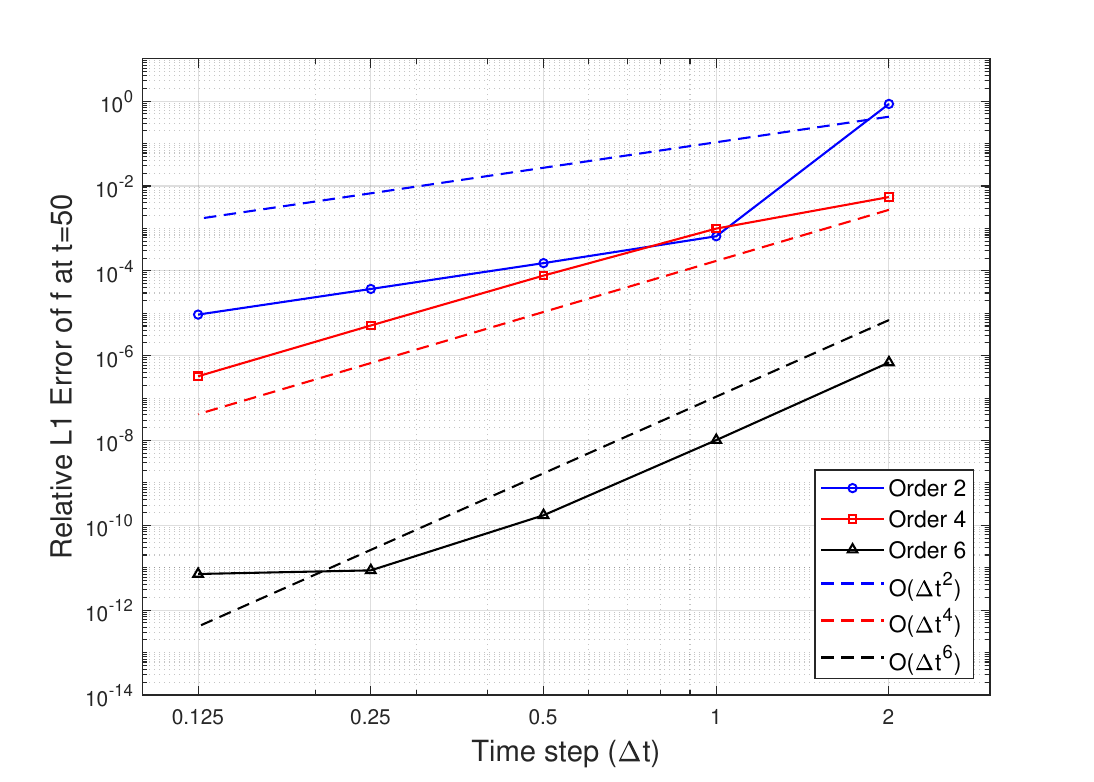}
			\caption{Weak Landau damping.}\label{Weak acc}
		\end{figure}
		In Figure \ref{Weak acc}, we display the resulting relative errors as a function of the time step size. As expected, the 2nd and 4th-order methods exhibit slopes closely aligning with the theoretical curves of $\mathcal{O}(\Delta t^2)$ and $\mathcal{O}(\Delta t^4)$, respectively. We note that, for the 6th-order splitting method, the relative error drops along the curve of $\mathcal{O}(\Delta t^6)$ and reaches a plateau. This flattening occurs because the temporal error has fallen below the machine precision threshold, demonstrating the exceptionally high accuracy and rapid convergence of the 6th-order scheme.
		
		\subsubsection{Reversibility}	
		Given that the Vlasov-Poisson system is inherently time-reversible, we aim to evaluate whether our proposed spectral method can preserve this reversible behavior at the discrete level. To this end, we define and compare two numerical solutions: the forward solution ($f_{\text{fwd}}$) and the backward solution ($f_{\text{bwd}}$). 
		
		The forward solution, $f_{\text{fwd}}(t)$, is obtained by directly integrating the system forward in time from the initial state at $t = 0$ up to $t = 80$ with a time step of $\Delta t = 0.25$. To compute the backward solution, we take the final state of the forward computation, $f_{\text{fwd}}(80)$, as a new initial condition. We then integrate the system backward in time from $t = 80$ down to $t = 0$ using a negative time step, $\Delta t = -0.25$, yielding $f_{\text{bwd}}(t)$. For example, when $t=50$, $f_{\text{fwd}}(50)$ represents the forward solution obtained by integrating the system from $t=0$ to $t=50$. Meanwhile, $f_{\text{bwd}}(50)$ represents the backward solution obtained by using $f_{\text{fwd}}(80)$ as the initial condition and integrating backward in time from $t=80$ to $t=50$ with a time step of $\Delta t = -0.25$.
		
		Using these two solutions, at each intermediate time step, we first compare the $L^2$ norms of the corresponding electric fields. We then quantitatively measure the discrepancy between $f_{\text{fwd}}$ and $f_{\text{bwd}}$ in terms of the $L^1$ and $L^2$ norms, total energy, and entropy. These numerical solutions are computed using a grid resolution of $N_x=32$ and $N_v=2816$.

			\begin{figure}[htbp]
					\centering
					\begin{subfigure}[b]{0.4\linewidth}
							\includegraphics[width=1\linewidth]{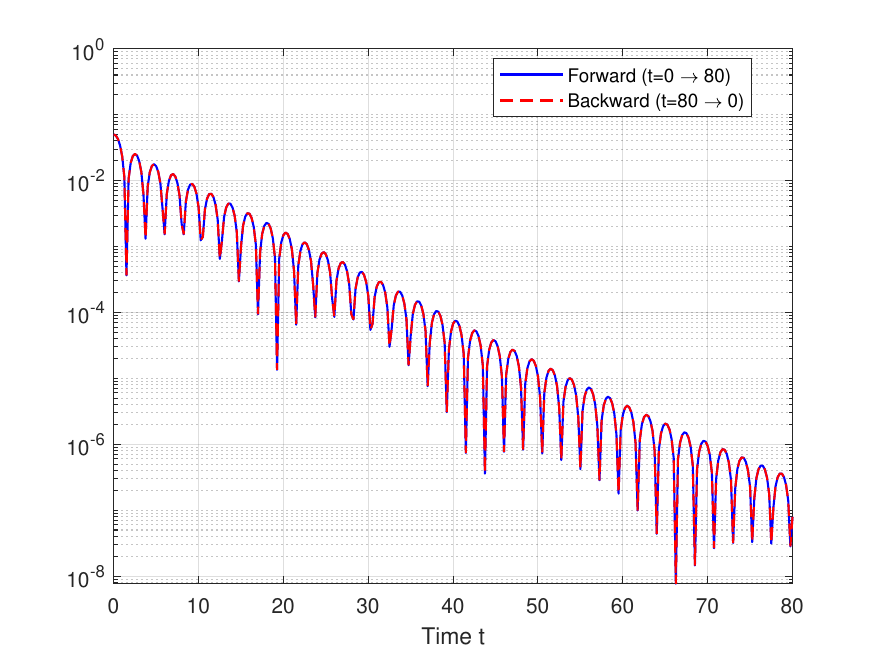}
							\subcaption{$L^2$-norm of $E$\\ \hspace{3mm}} 
						\end{subfigure}	
					\begin{subfigure}[b]{0.4\linewidth}
							\includegraphics[width=1\linewidth]{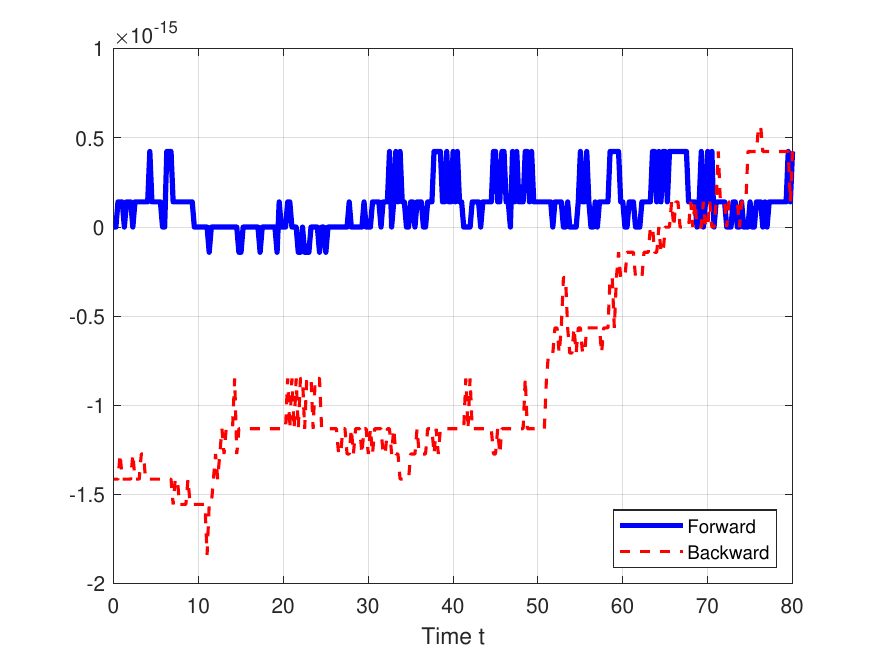}
							\subcaption{$\displaystyle \frac{\int f dvdx - \int f_0 dvdx}{\int f_0 dvdx}$}
						\end{subfigure}					
					\begin{subfigure}[b]{0.4\linewidth}
							\includegraphics[width=1\linewidth]{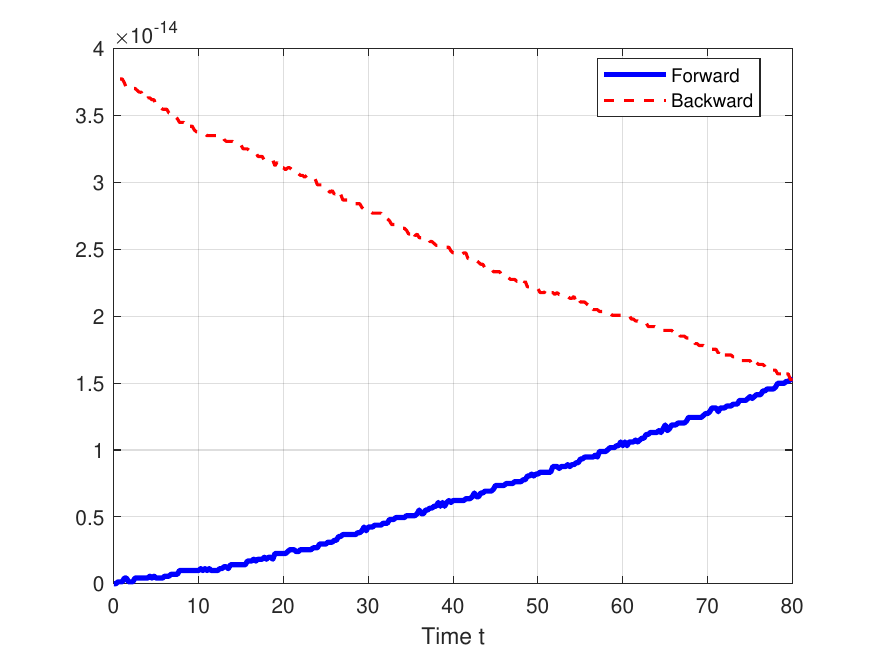}
							\subcaption{ $\displaystyle\frac{\|f^n\|_1-\|f^0\|_1}{\|f^0\|_1}$}
						\end{subfigure}	
					\begin{subfigure}[b]{0.4\linewidth}
							\includegraphics[width=1\linewidth]{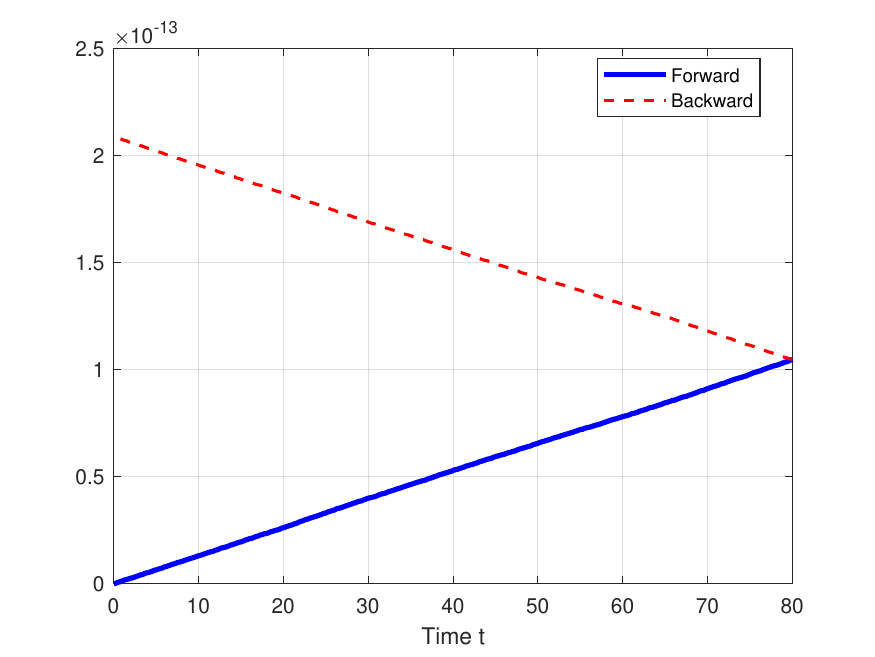}
							\subcaption{$\displaystyle\frac{\|f^n\|_2-\|f^0\|_2}{\|f^0\|_2}$}
						\end{subfigure}	
					\begin{subfigure}[b]{0.4\linewidth}
							\includegraphics[width=1\linewidth]{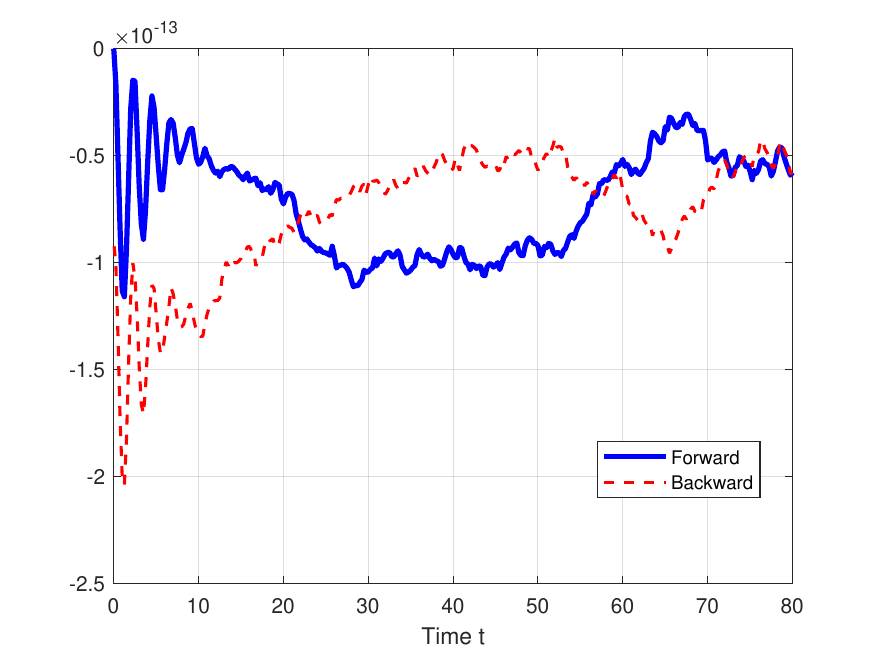}
							\subcaption{$\displaystyle \frac{Energy(t)-Energy(0)}{Energy(0)}$}
						\end{subfigure}	
					\begin{subfigure}[b]{0.4\linewidth}
							\includegraphics[width=1\linewidth]{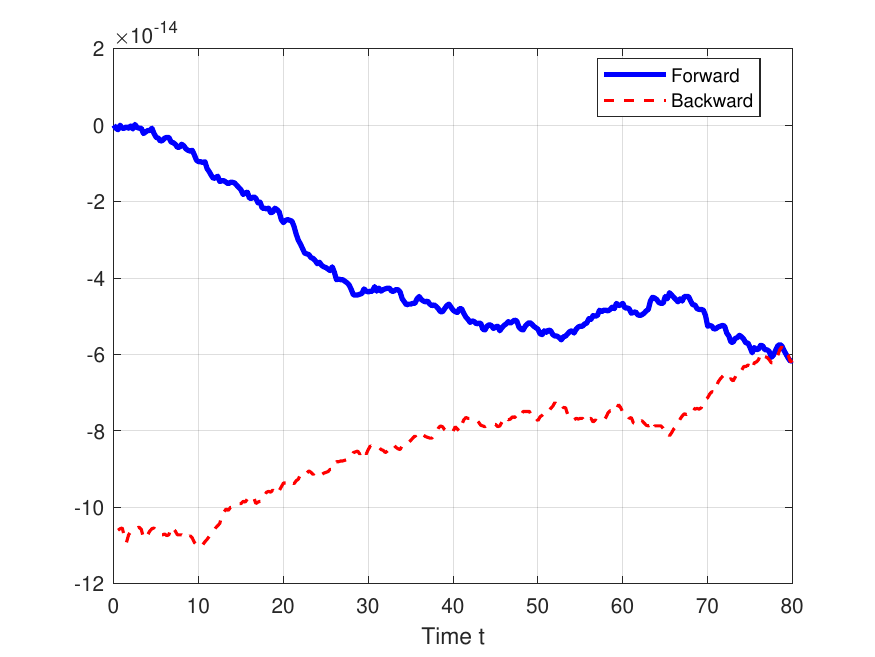}
							\subcaption{$\displaystyle \frac{Entropy(t)-Entropy(0)}{Entropy(0)}$}
						\end{subfigure}	
					\caption{Weak Landau damping. Numerical results for reversibility.}\label{Fig weak 2}
				\end{figure}	
                
			\begin{figure}[htbp]
					\centering
					\begin{subfigure}[b]{0.4\linewidth}
							\includegraphics[width=1\linewidth]{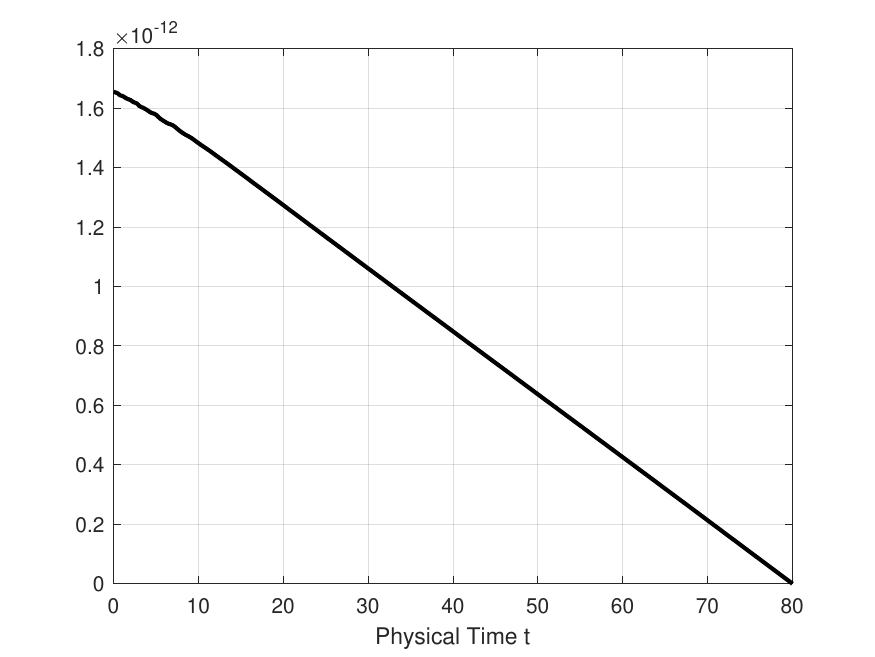}
							\subcaption{$\frac{||f_{\text{fwd}}(t) - f_{\text{bwd}}(t)||_1}{||f_{\text{fwd}}(0)||_1}$}
						\end{subfigure}
					\begin{subfigure}[b]{0.4\linewidth}
							\includegraphics[width=1\linewidth]{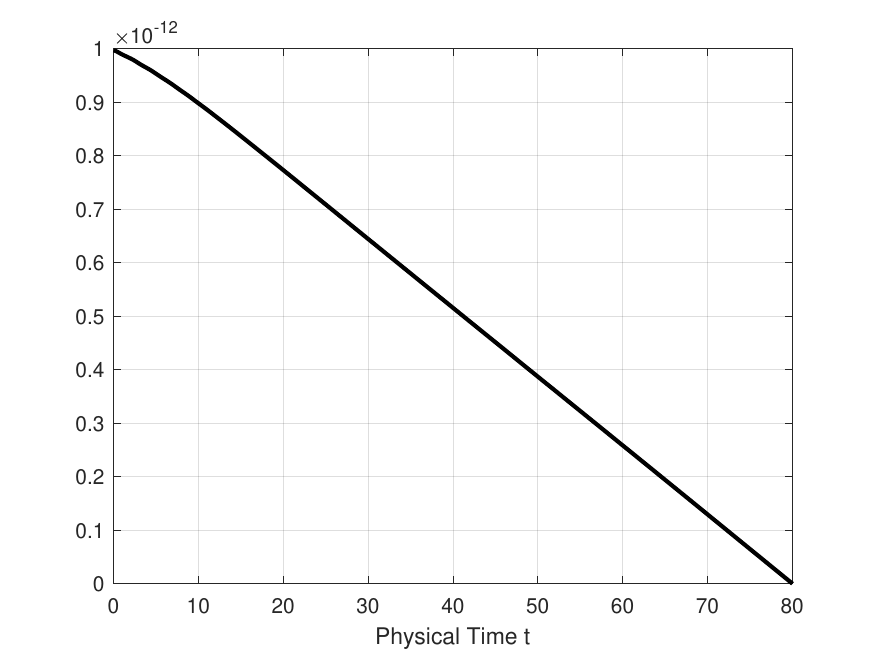}
											\subcaption{$\frac{||f_{\text{fwd}}(t) - f_{\text{bwd}}(t)||_2}{||f_{\text{fwd}}(0)||_2}$}
						\end{subfigure}		
					\caption{Weak Landau damping. Numerical results for reversibility.}\label{Fig weak 3}
				\end{figure}

		As shown in Figure \ref{Fig weak 2}, the $L^2$ norms of the electric fields associated with the forward and backward solutions show very good agreement with each other. Moreover, we observe that macroscopic conservative quantities—such as the $L^1$- and $L^2$-norms of the distribution function, total entropy, and total energy—are recovered with exceptional accuracy during the backward integration. This strict reversibility is strongly supported by the remarkably small relative discrepancies between the forward and backward distributions ($f_{\text{fwd}}$ and $f_{\text{bwd}}$), as illustrated in Figure \ref{Fig weak 3}. 
		
			\begin{figure}[ht]
		\centering
		\includegraphics[width=0.4\linewidth]{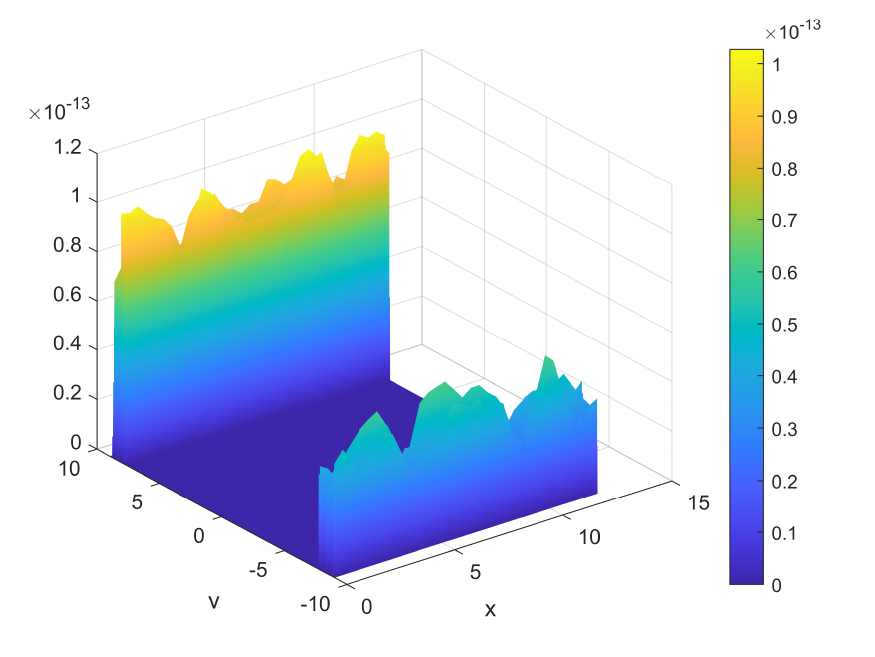}
		\caption{Weak Landau damping. The negative part of $f$, i.e. $\displaystyle \frac{|f|-f}{2}$.}\label{weak neg}
	\end{figure}	
		We note that the minor loss of positivity observed in the simulation is primarily attributable to numerical round-off errors near the truncated velocity boundaries, as detailed in Figure \ref{weak neg}.

					\subsubsection{Application of adaptive refinement}
						Next, we demonstrate the performance of adaptive Fourier spectral methods combined with the proposed AMR approaches.
					For the AMR methods, we initially begin with $N_x \times N_v = 32 \times 128$, while for the non-AMR methods, we use $32 \times 2816$, which allows us to capture fine-scale structures generated by phase mixing during the simulation time $t\in [0,80]$. 
					
					\begin{figure}[h]
						\centering
						\begin{subfigure}[b]{0.4\linewidth}
							\includegraphics[width=1\linewidth]{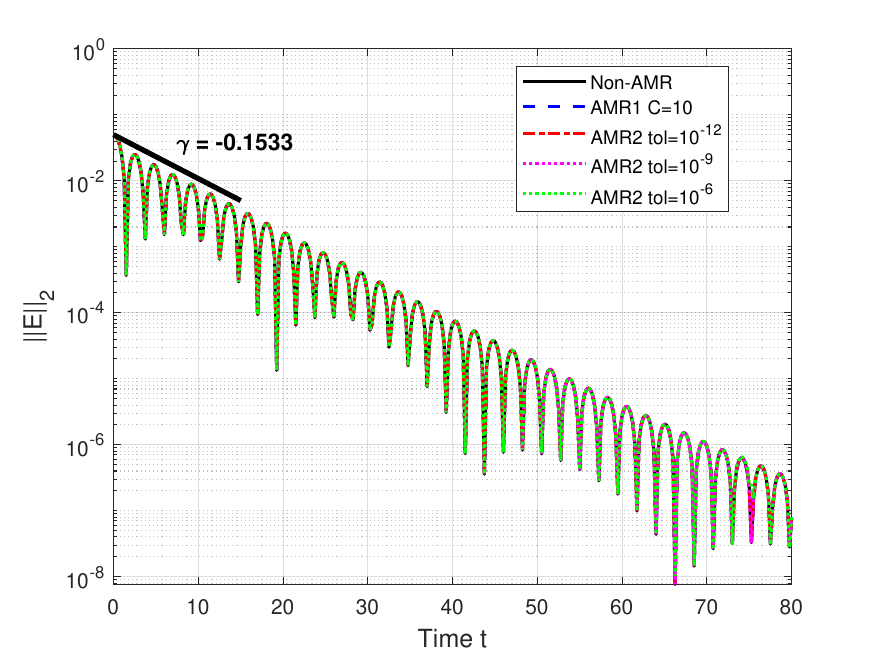}
							\subcaption{$L^2$-norm of $E$\\ \hspace{3mm}}\label{weak a}
						\end{subfigure}	
							\begin{subfigure}[b]{0.4\linewidth}
						\includegraphics[width=1\linewidth]{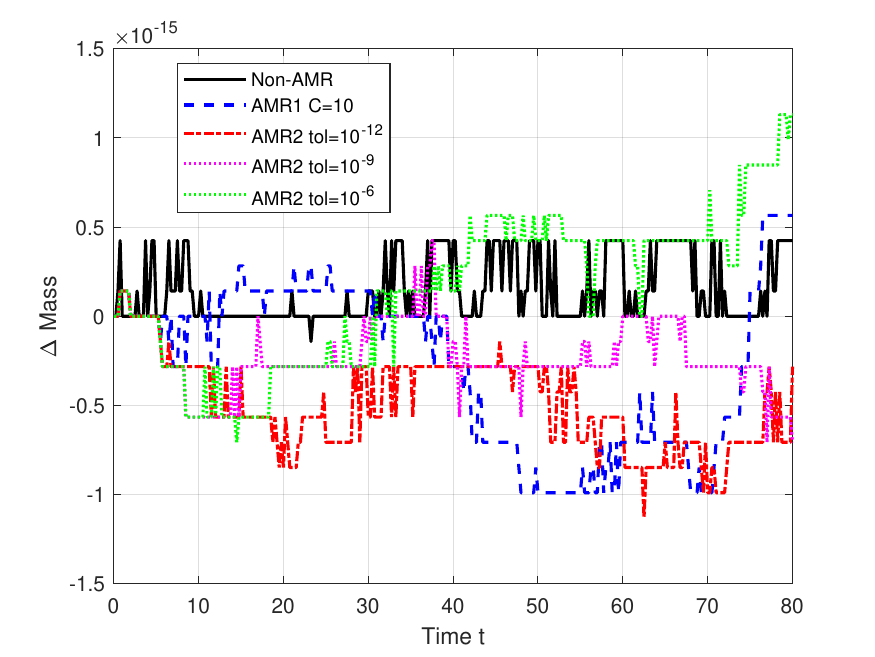}
						\subcaption{$\displaystyle \frac{\int f dvdx - \int f_0 dvdx}{\int f_0 dvdx}$}\label{weak b}
					\end{subfigure}	
						\begin{subfigure}[b]{0.4\linewidth}
						\includegraphics[width=1\linewidth]{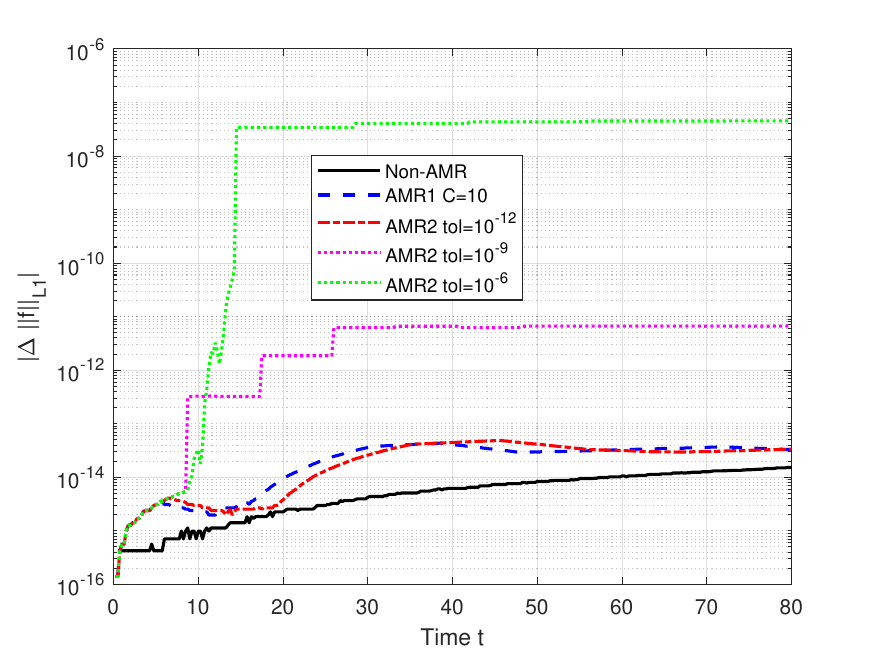}
						\subcaption{ $\displaystyle\frac{\|f^n\|_1-\|f^0\|_1}{\|f^0\|_1}$}\label{weak c}
					\end{subfigure}	
						\begin{subfigure}[b]{0.4\linewidth}
							\includegraphics[width=1\linewidth]{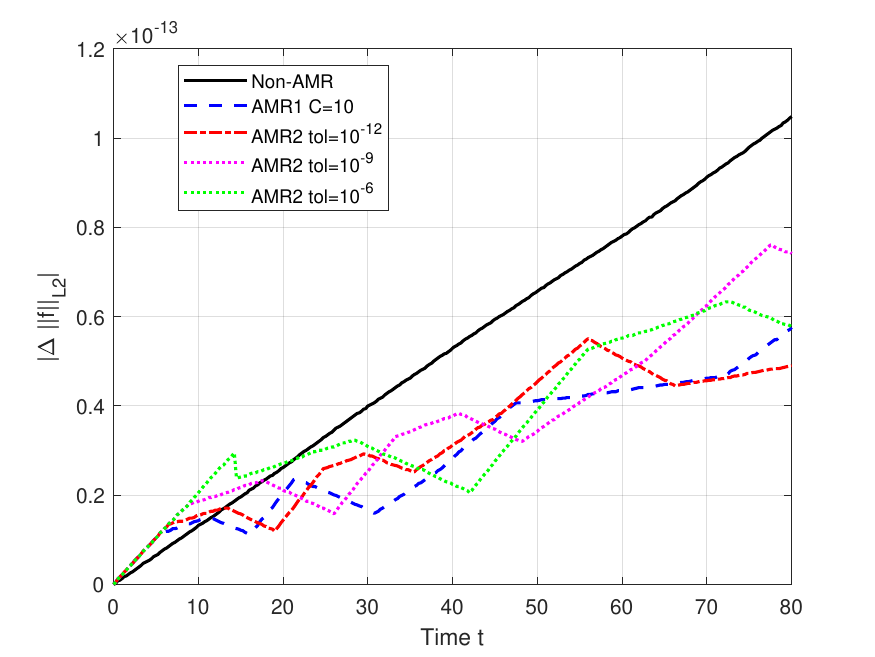}
							\subcaption{$\displaystyle\frac{\|f^n\|_2-\|f^0\|_2}{\|f^0\|_2}$}\label{weak d}
						\end{subfigure}		
									\begin{subfigure}[b]{0.4\linewidth}
						\includegraphics[width=1\linewidth]{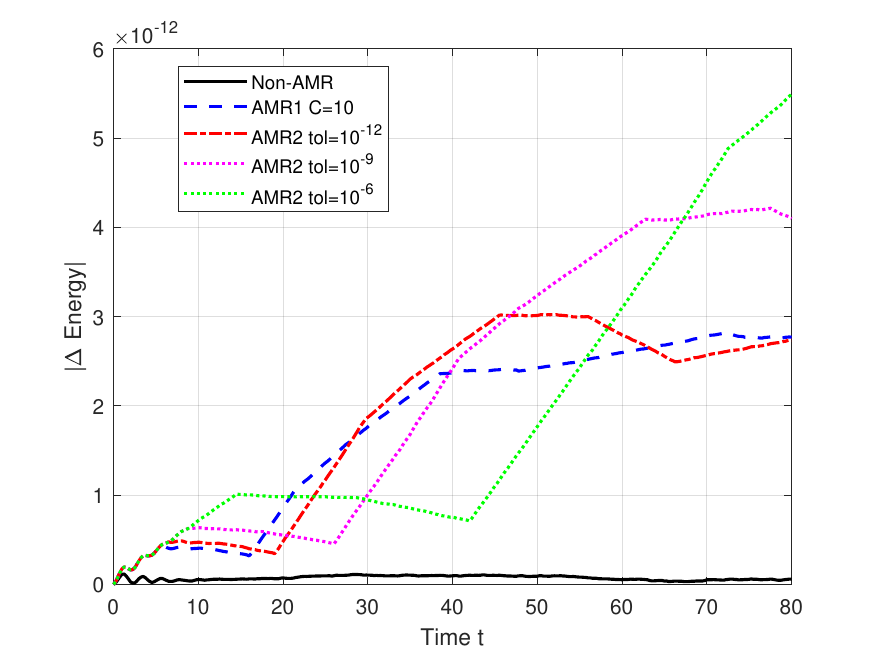}
						\subcaption{$\displaystyle \frac{Energy(t)-Energy(0)}{Energy(0)}$}\label{weak e}
					\end{subfigure}	
						\begin{subfigure}[b]{0.4\linewidth}
							\includegraphics[width=1\linewidth]{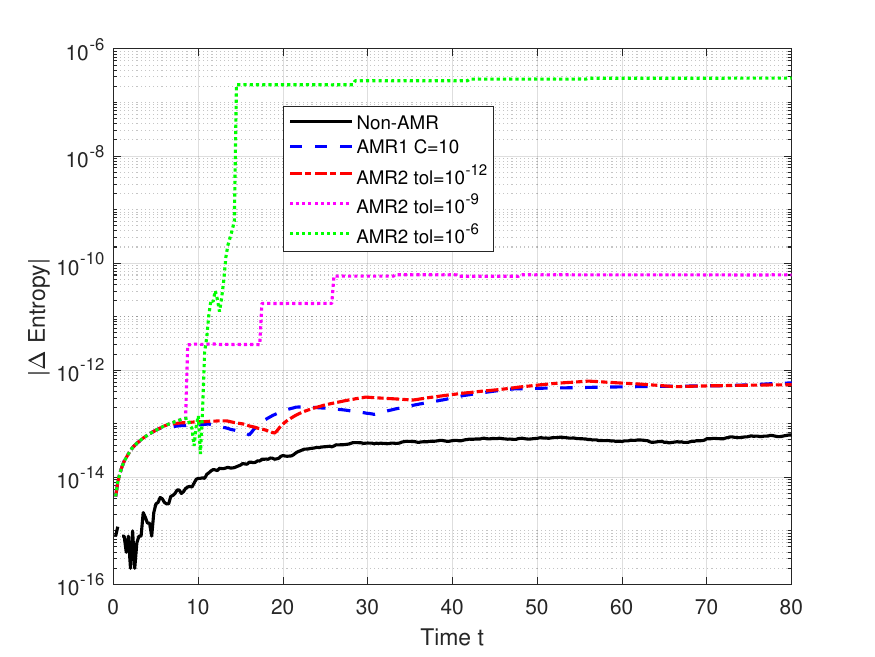}
							\subcaption{$\displaystyle \frac{Entropy(t)-Entropy(0)}{Entropy(0)}$}\label{weak f}
						\end{subfigure}					
						\caption{Weak Landau damping.}\label{Fig weak 1}
					\end{figure}					
				
						In Figure \ref{Fig weak 1}, we compare the performance of non-adaptive and adaptive techniques. 
					We first note that all Fourier spectral methods resolve the theoretical damping rate $\gamma=-0.1533$ \cite{FSB} of the $L^2$-norm of the electric field (see Figure \ref{weak a}), and the total mass is exactly preserved (Figure \ref{weak b}). The other conservative quantities are also well preserved except for the $L^1$-norm in the case of AMR2 with $\tau_{tol}=10^{-6}$ (Figures \ref{weak c}-\ref{weak f}).

                
				As depicted in Figure \ref{Fig weak 1}, the time evolution of the conservative variables reveals a subtle but important numerical feature of the proposed adaptive scheme. Whenever the Fourier domain is extended, errors increase both in the $L^1$-norm and in the entropy. These increases in errors are intrinsically due to the loss of positivity. More precisely, the zero-padding operation in the Fourier domain corresponds to global Fourier interpolation in the physical domain. Thus, near locations where the function values are extremely close to zero, such as the tail of the velocity distribution, this Fourier interpolation inevitably introduces highly localized, small negative values. Because the $L^1$-norm evaluates the absolute value and the entropy penalizes negative values, the sudden emergence of these negative undershoots directly translates into the observed jumps in errors. Nevertheless, it is important to emphasize that this localized loss of positivity at the tails is a well-known artifact of pure Fourier spectral methods. As mathematically proven in Propositions \ref{Prop 1} and \ref{Prop 2}, it does not pollute the fundamental algebraic conservation of the mass and the $L^2$-norm up to $t_m$.
					\begin{figure}[h]
					\centering
					\begin{subfigure}[b]{0.4\linewidth}
						\includegraphics[width=1\linewidth]{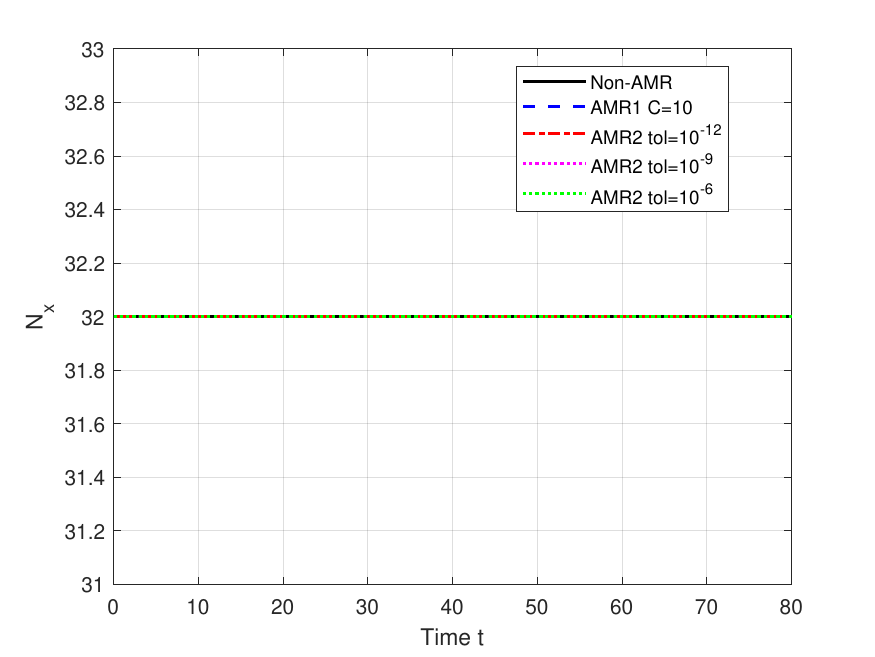}
						\subcaption{$N_x$} 
					\end{subfigure}	
					\begin{subfigure}[b]{0.4\linewidth}
						\includegraphics[width=1\linewidth]{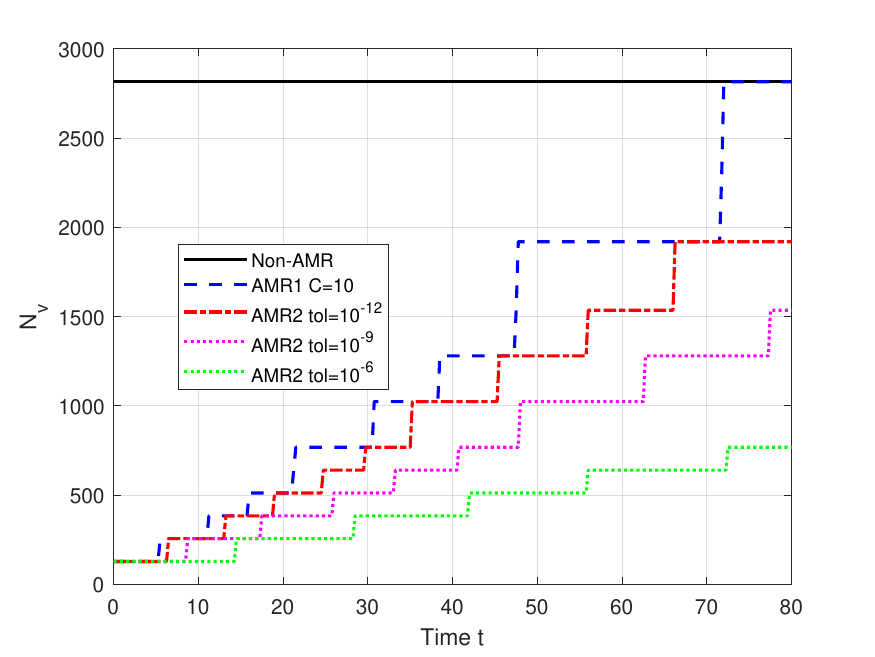}
						\subcaption{$N_v$} 
					\end{subfigure}	
					\begin{subfigure}[b]{0.4\linewidth}
						\includegraphics[width=1\linewidth]{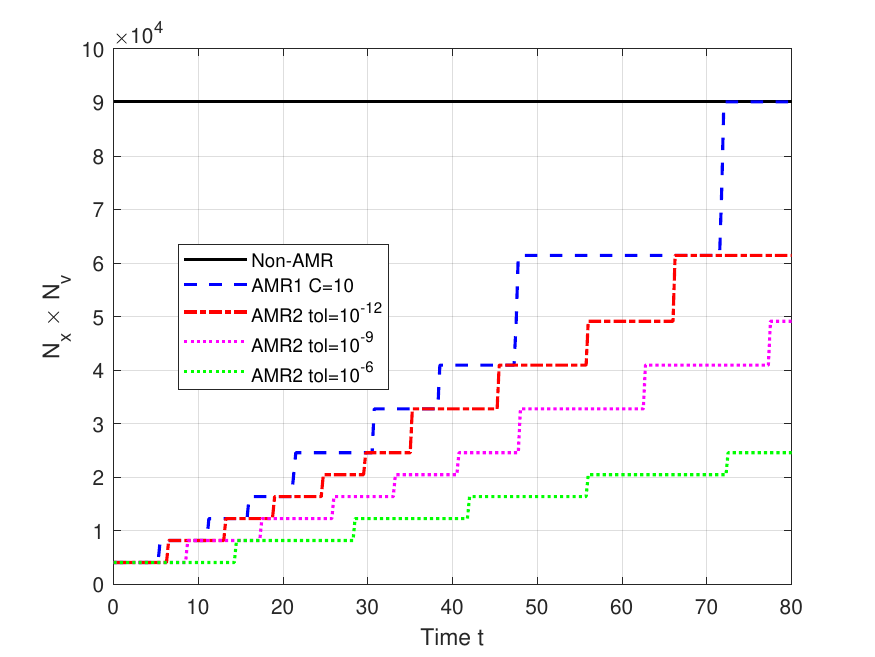}
						\subcaption{$N_x \times N_v$} 
					\end{subfigure}	
                    \begin{subfigure}[b]{0.4\linewidth}
						\includegraphics[width=1\linewidth]{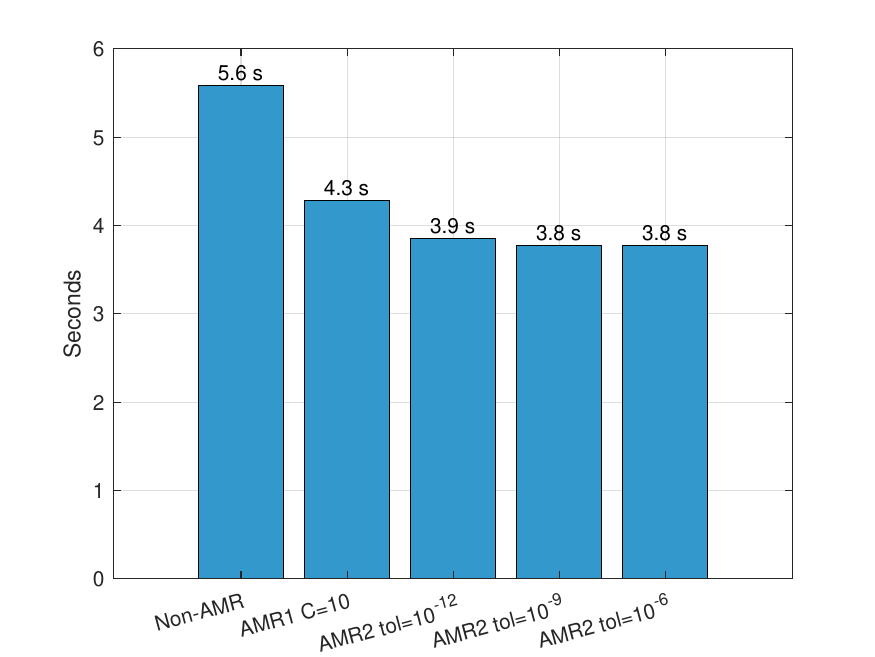}
						\subcaption{GPU time}\label{weak cpu}
					\end{subfigure}		
					\caption{Weak Landau damping. Time evolution of grid numbers and GPU time.}\label{weak N}			\end{figure}
			
				Figure \ref{weak N} shows that the AMR method yields reductions in both computational time and memory usage without losing accuracy. The comparison regarding the time evolution of the number of grid points in Figure \ref{weak N} implies that the refinement criterion for AMR1 is overly strict, leading to a premature expansion of the domain compared to AMR2. Among the AMR2 methods with different thresholds, we verify that decreasing $\tau_{tol}$ accelerates the domain extension. We also note that the number of grid points required to fully resolve the solution increases approximately linearly over time, see panels (B) and (C). The time evolution of the phase space profiles is illustrated in Figure \ref{weak phase}, which is in good agreement with \cite{SKT}.

                To explicitly visualize the mechanism of our dynamic resolution allocation, in Figure \ref{weak Fourier mode} we illustrate the temporal evolution of the spatially averaged Fourier modes along the velocity direction, i.e., $\frac{1}{N_x}\sum_l |\hat{f}(l,k_v)|$. 
                As phase mixing progresses, the macroscopic structures cascade into microscopic scales, generating a high-frequency spectral tail. In the fixed-grid approach, this tail eventually reaches the prescribed limit in Fourier space, risking numerical aliasing. In contrast, the proposed AMR methods effectively accommodate the cascading spectral energy by dynamically expanding the maximum resolvable wavenumber ($k_{\max}$) via increasing $N_v$. Thanks to the relatively mild nonlinear effects in the weak Landau damping regime, the adaptive zero-padding successfully maintains a clean exponential tail decay throughout the entire simulation without exhausting the hardware memory limits.

				\begin{figure}[h]
					\centering
					\begin{subfigure}[b]{0.32\linewidth}
						\includegraphics[width=1\linewidth]{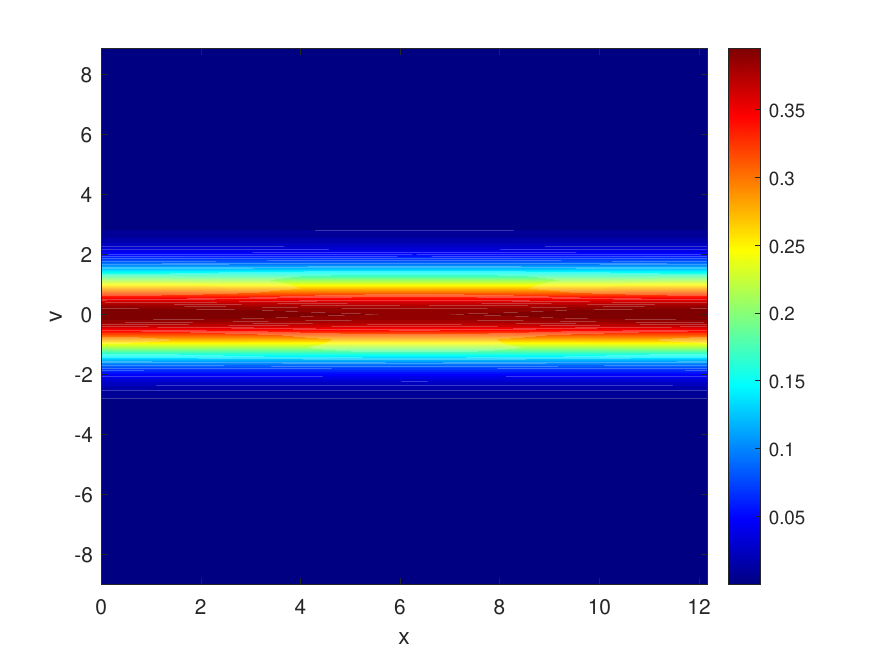}
						\subcaption{$t=0$} 
					\end{subfigure}
				    \begin{subfigure}[b]{0.32\linewidth}
				    	\includegraphics[width=1\linewidth]{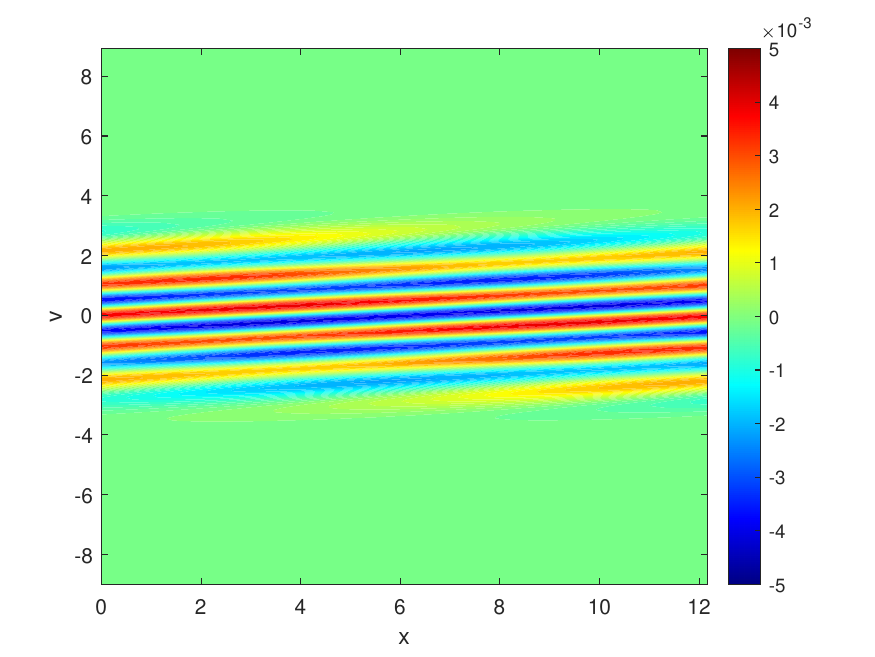}
				    	\subcaption{$t=10$} 
				    \end{subfigure}
					\begin{subfigure}[b]{0.32\linewidth}
						\includegraphics[width=1\linewidth]{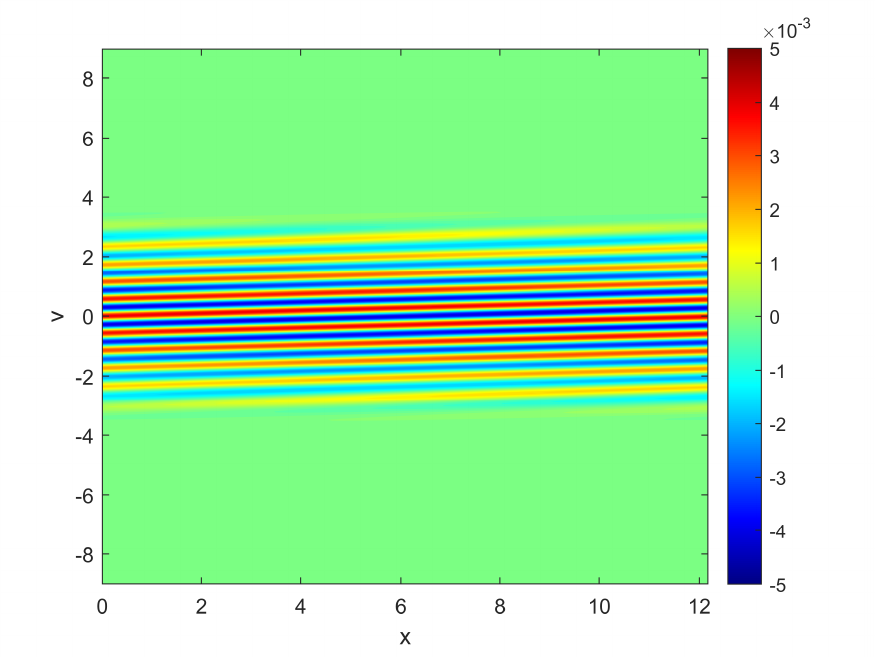}
						\subcaption{$t=20$} 
					\end{subfigure}
					\begin{subfigure}[b]{0.32\linewidth}
						\includegraphics[width=1\linewidth]{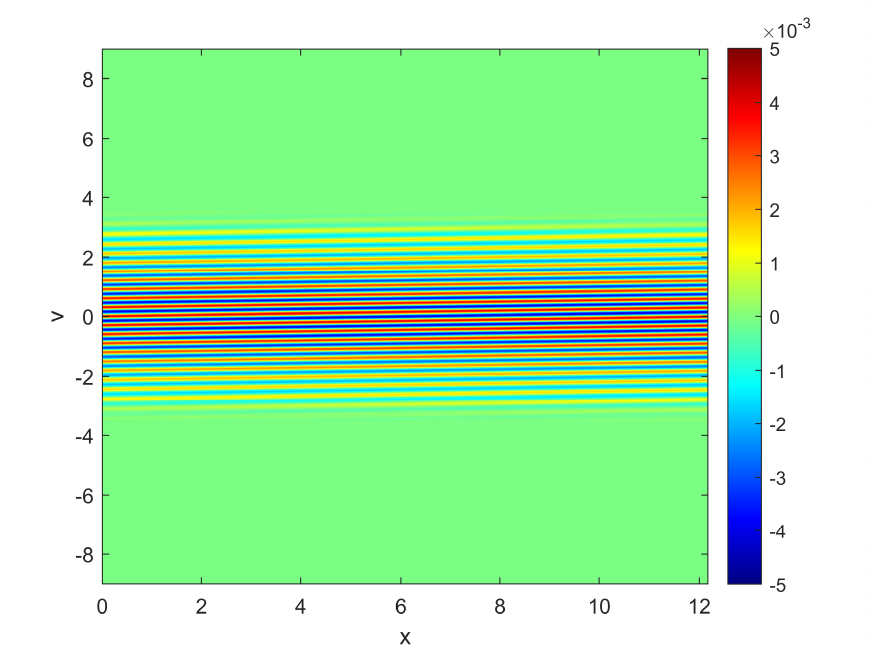}
						\subcaption{$t=40$} 
					\end{subfigure}
					\begin{subfigure}[b]{0.32\linewidth}
						\includegraphics[width=1\linewidth]{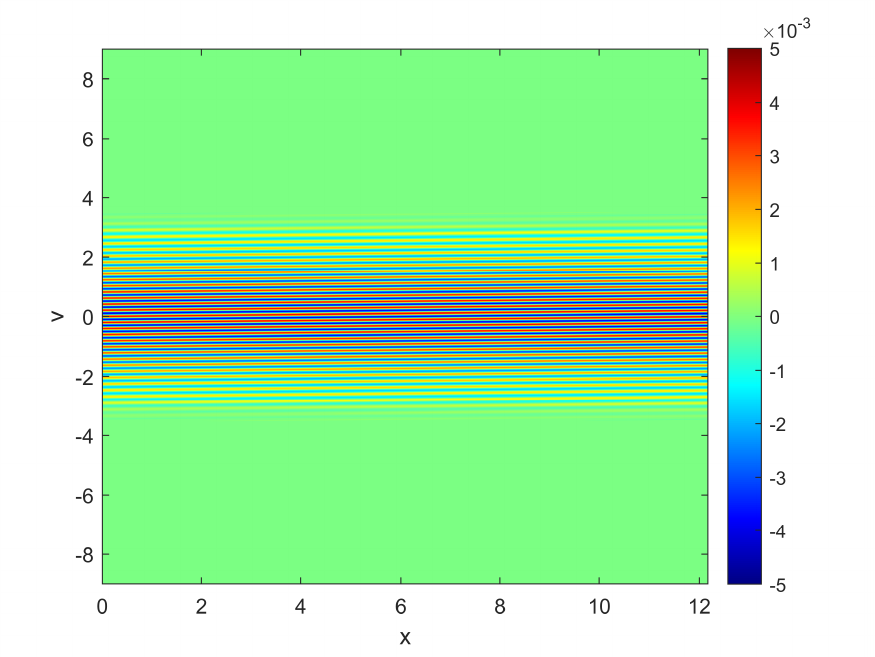}
						\subcaption{$t=60$} 
					\end{subfigure}		
					\begin{subfigure}[b]{0.32\linewidth}
						\includegraphics[width=1\linewidth]{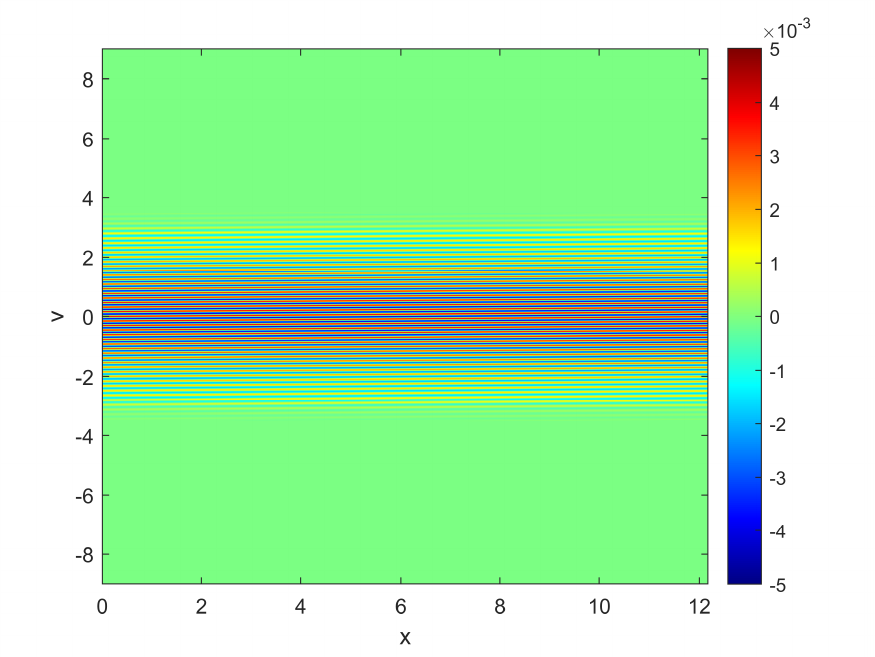}
						\subcaption{$t=80$} 
					\end{subfigure}				
					\caption{Weak Landau damping. Phase space profiles of $M$ at $t=0$, and the perturbation $f-M$ at subsequent times.}\label{weak phase}
				\end{figure}

\begin{figure}[h]
	\centering

    \begin{subfigure}[b]{0.24\linewidth}
		\includegraphics[width=1\linewidth]{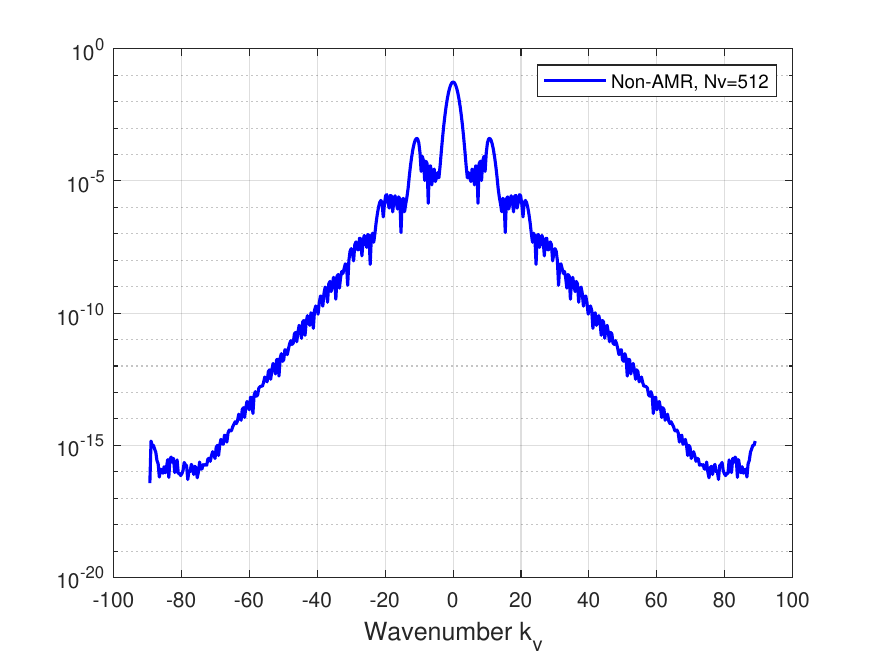}
		\subcaption{$t=20$} 
	\end{subfigure}
	\begin{subfigure}[b]{0.24\linewidth}
		\includegraphics[width=1\linewidth]{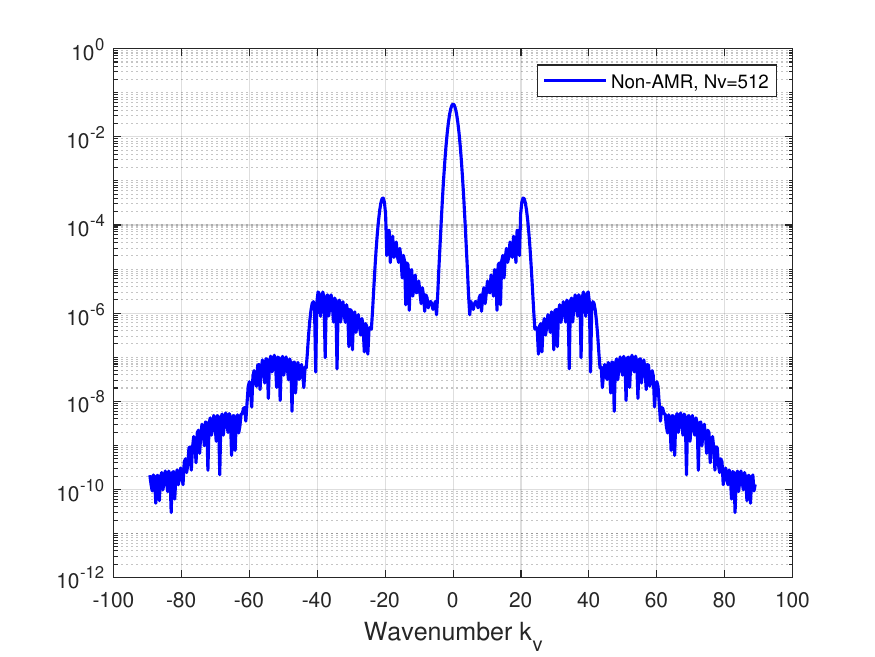}
		\subcaption{$t=40$} 
	\end{subfigure}
	\begin{subfigure}[b]{0.24\linewidth}		\includegraphics[width=1\linewidth]{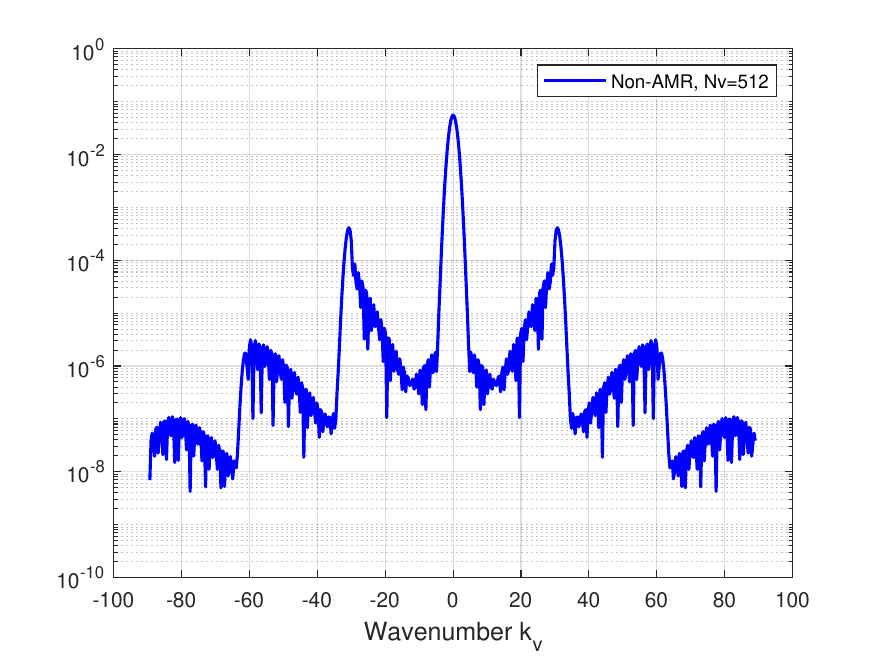}
		\subcaption{$t=60$} 
	\end{subfigure}		
	\begin{subfigure}[b]{0.24\linewidth}
		\includegraphics[width=1\linewidth]{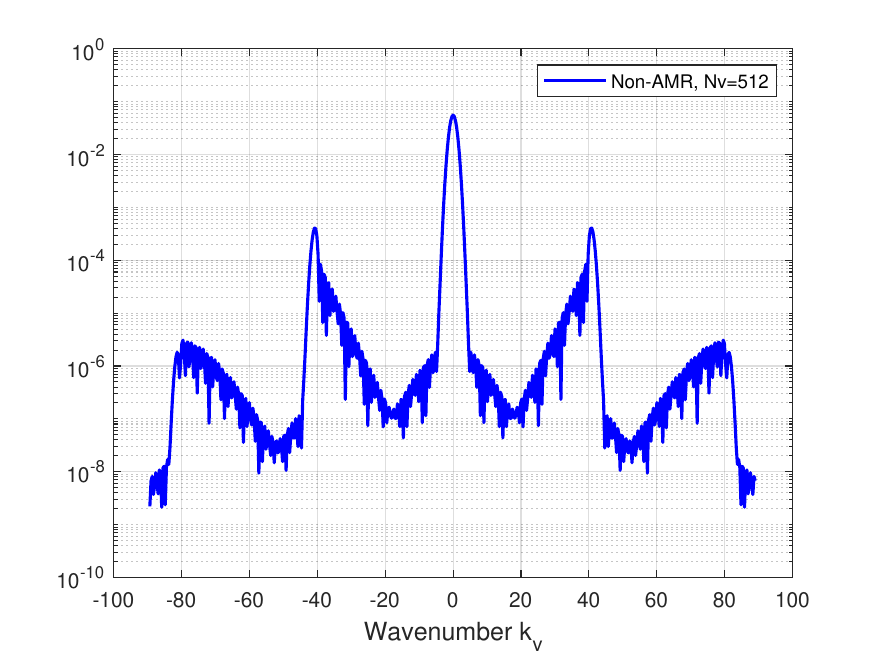}
		\subcaption{$t=80$} 
	\end{subfigure}		

    \begin{subfigure}[b]{0.24\linewidth}
		\includegraphics[width=1\linewidth]{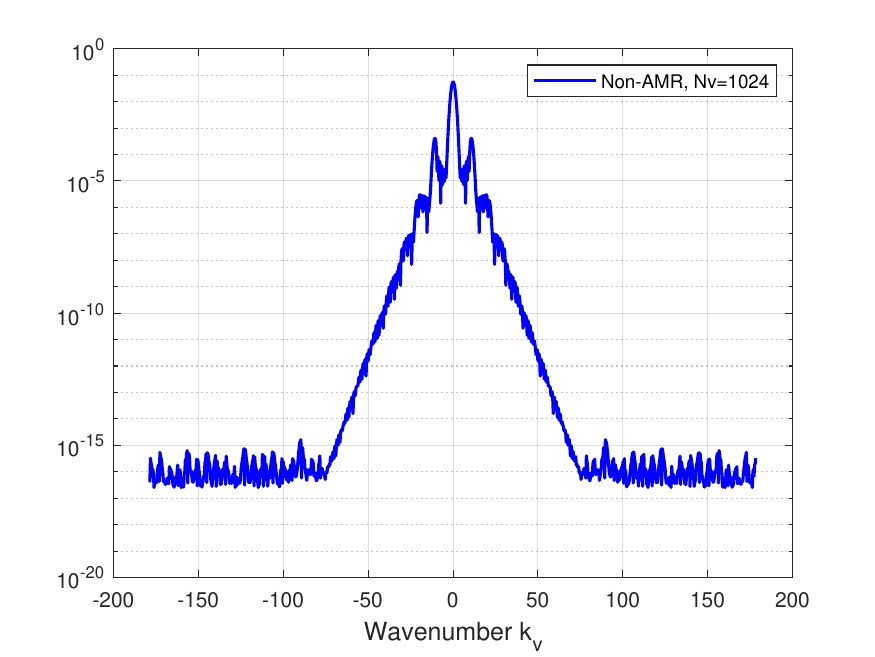}
		\subcaption{$t=20$} 
	\end{subfigure}
	\begin{subfigure}[b]{0.24\linewidth}
		\includegraphics[width=1\linewidth]{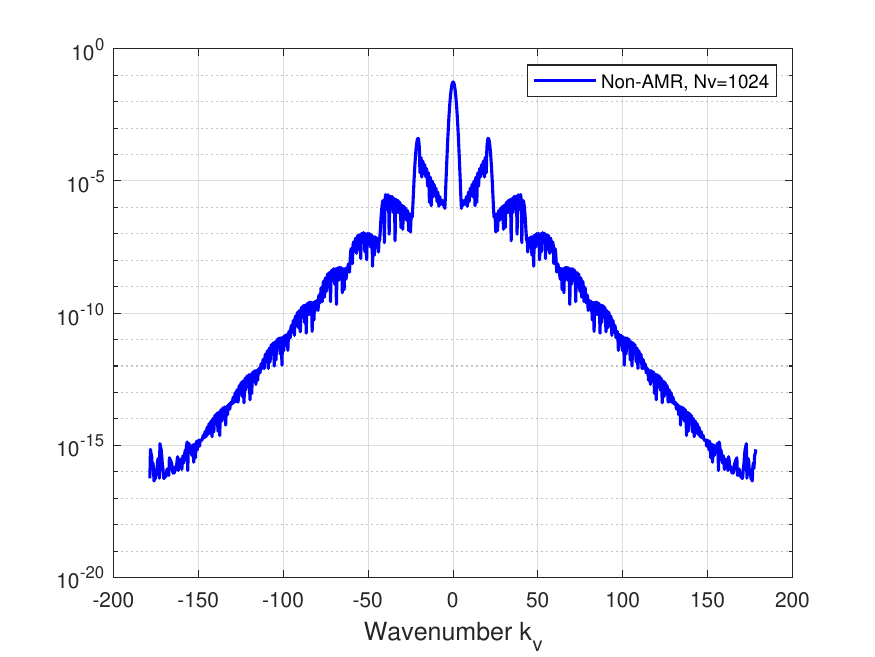}
		\subcaption{$t=40$} 
	\end{subfigure}
	\begin{subfigure}[b]{0.24\linewidth}		\includegraphics[width=1\linewidth]{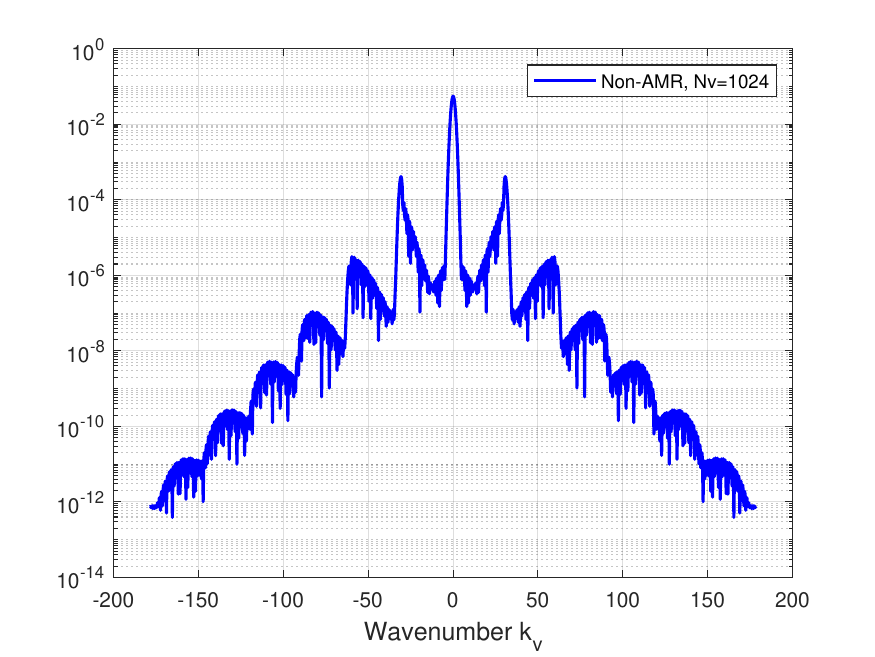}
		\subcaption{$t=60$} 
	\end{subfigure}		
	\begin{subfigure}[b]{0.24\linewidth}
		\includegraphics[width=1\linewidth]{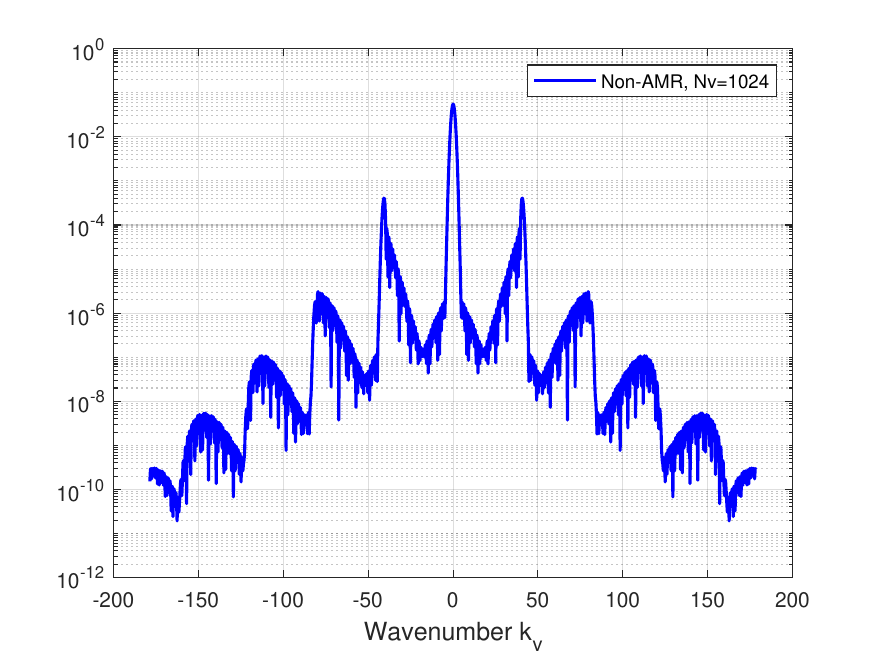}
		\subcaption{$t=80$} 
	\end{subfigure}		

    \begin{subfigure}[b]{0.24\linewidth}
		\includegraphics[width=1\linewidth]{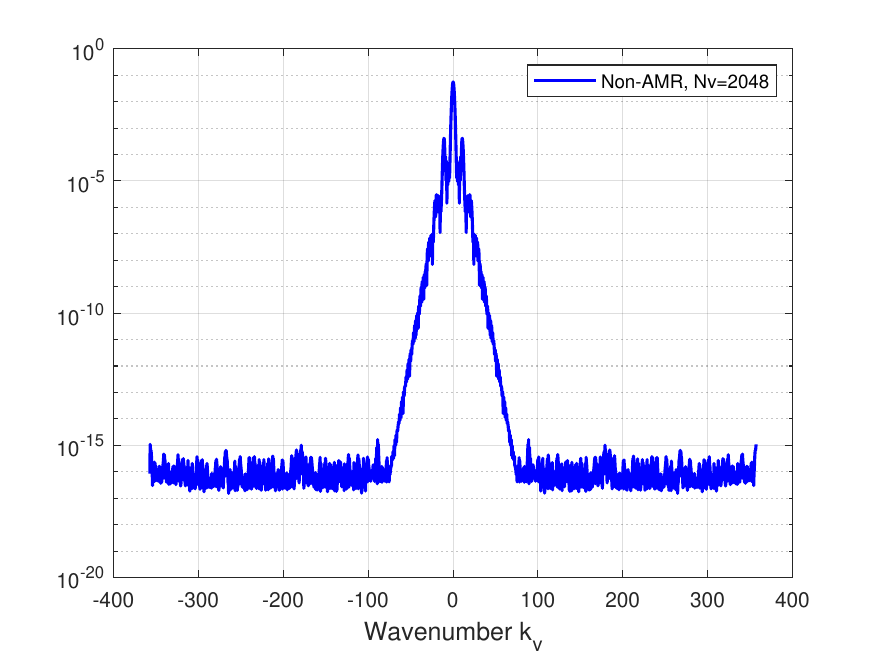}
		\subcaption{$t=20$} 
	\end{subfigure}
	\begin{subfigure}[b]{0.24\linewidth}
		\includegraphics[width=1\linewidth]{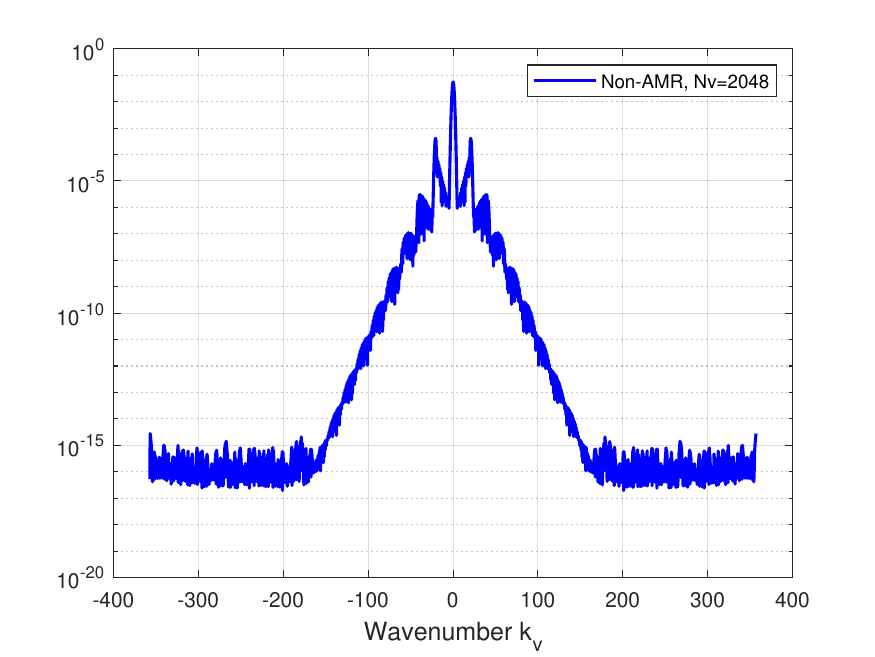}
		\subcaption{$t=40$} 
	\end{subfigure}
	\begin{subfigure}[b]{0.24\linewidth}		\includegraphics[width=1\linewidth]{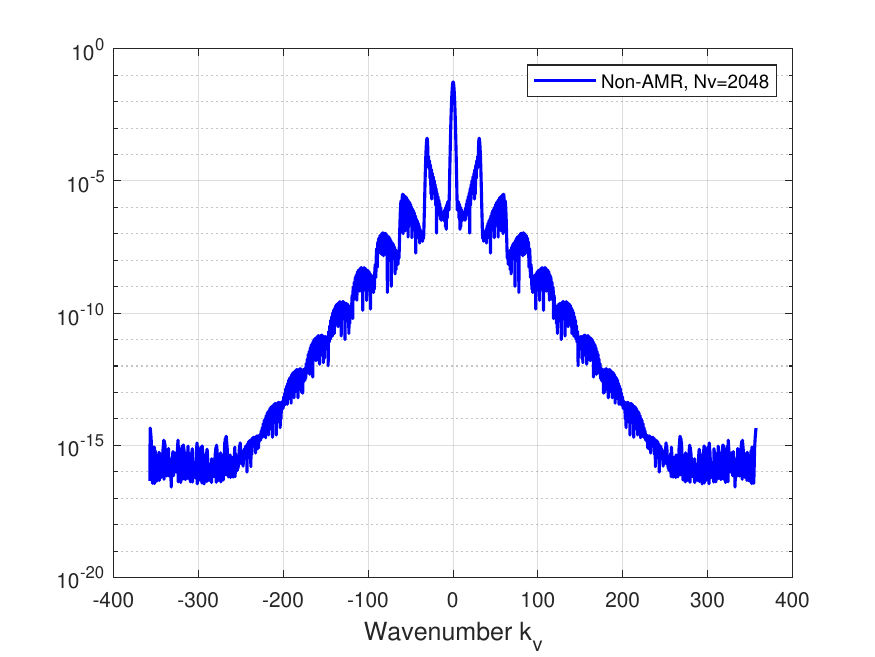}
		\subcaption{$t=60$} 
	\end{subfigure}		
	\begin{subfigure}[b]{0.24\linewidth}
		\includegraphics[width=1\linewidth]{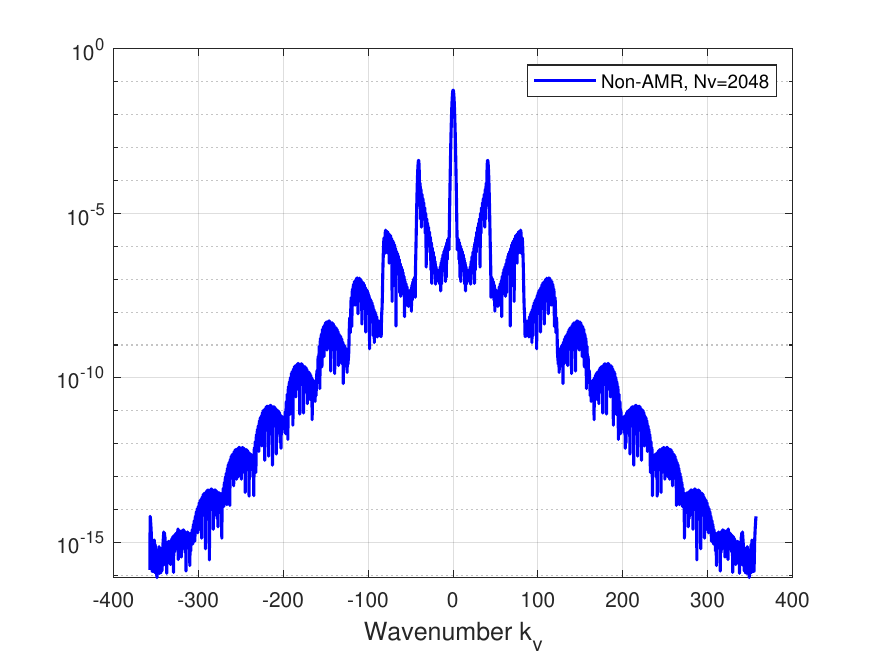}
		\subcaption{$t=80$} 
	\end{subfigure}

\begin{subfigure}[b]{0.24\linewidth}
		\includegraphics[width=1\linewidth]{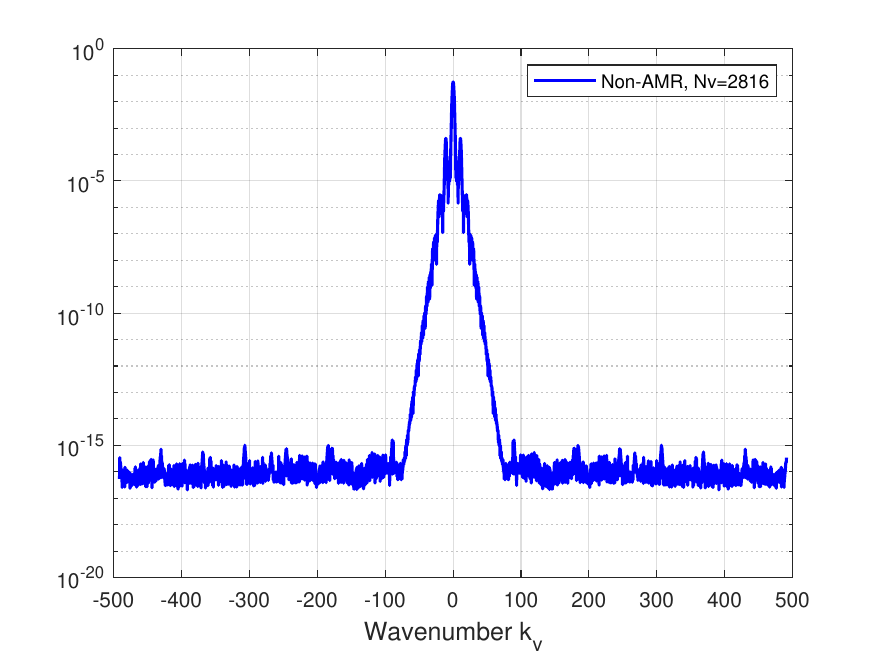}
		\subcaption{$t=20$} 
	\end{subfigure}
	\begin{subfigure}[b]{0.24\linewidth}
		\includegraphics[width=1\linewidth]{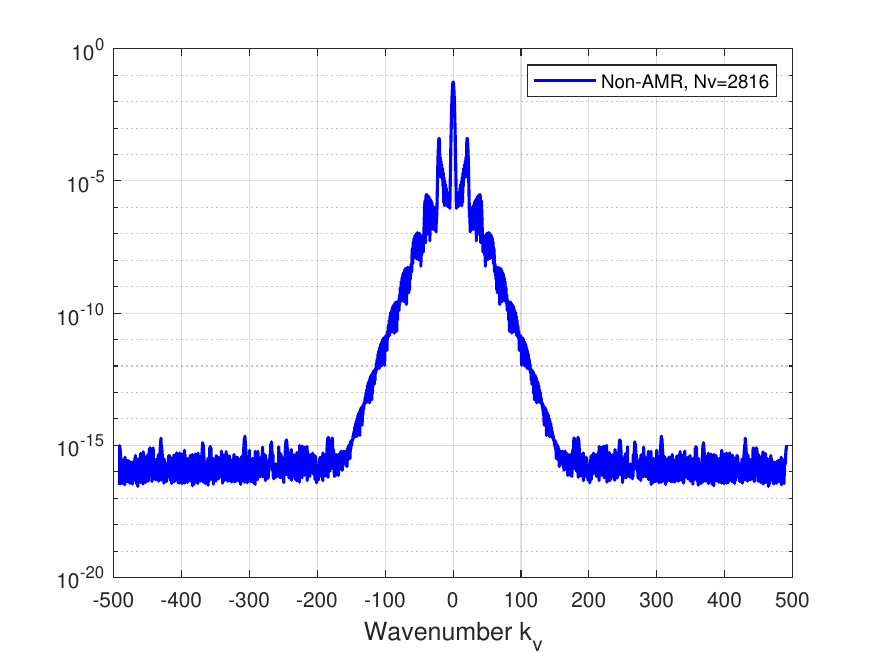}
		\subcaption{$t=40$} 
	\end{subfigure}
	\begin{subfigure}[b]{0.24\linewidth}		\includegraphics[width=1\linewidth]{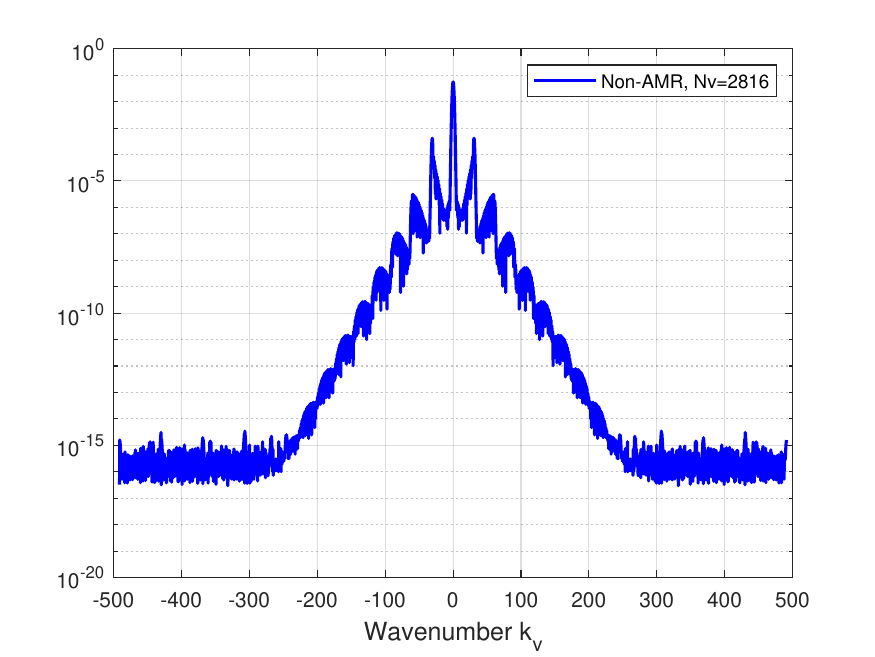}
		\subcaption{$t=60$} 
	\end{subfigure}		
	\begin{subfigure}[b]{0.24\linewidth}
		\includegraphics[width=1\linewidth]{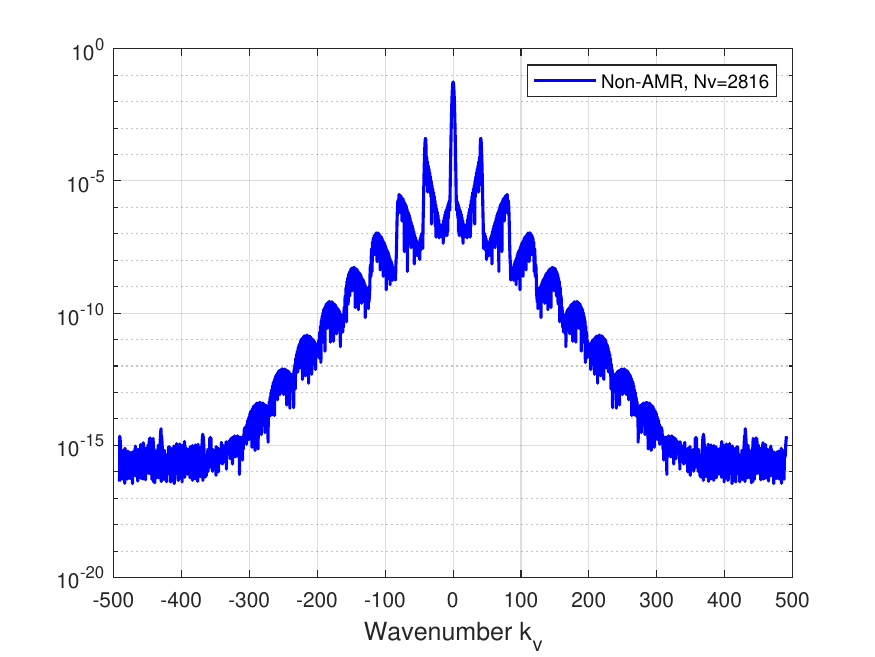}
		\subcaption{$t=80$} 
	\end{subfigure}		
     \begin{subfigure}[b]{0.24\linewidth}
		\includegraphics[width=1\linewidth]{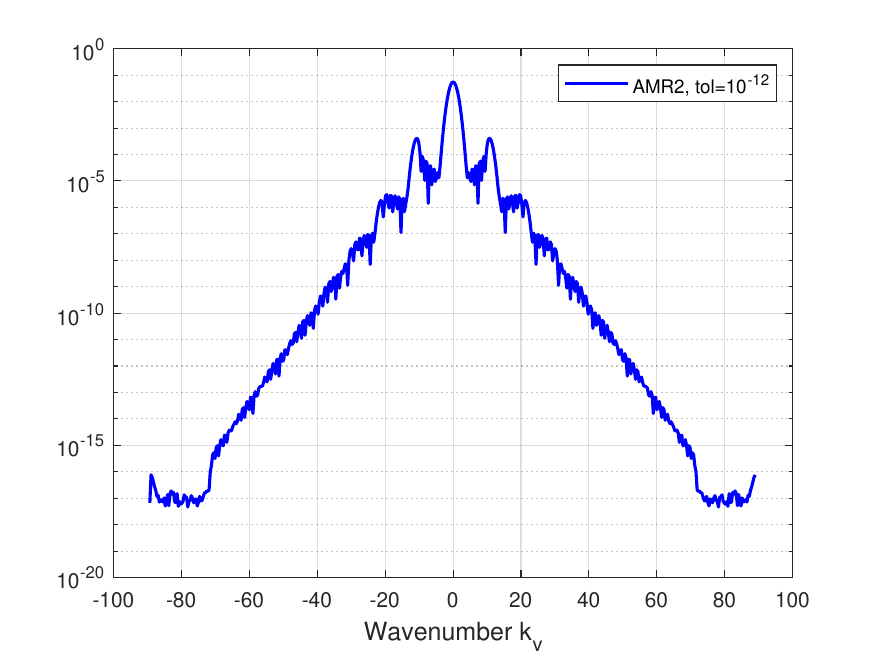}
		\subcaption{$t=20$} 
	\end{subfigure}
	\begin{subfigure}[b]{0.24\linewidth}
		\includegraphics[width=1\linewidth]{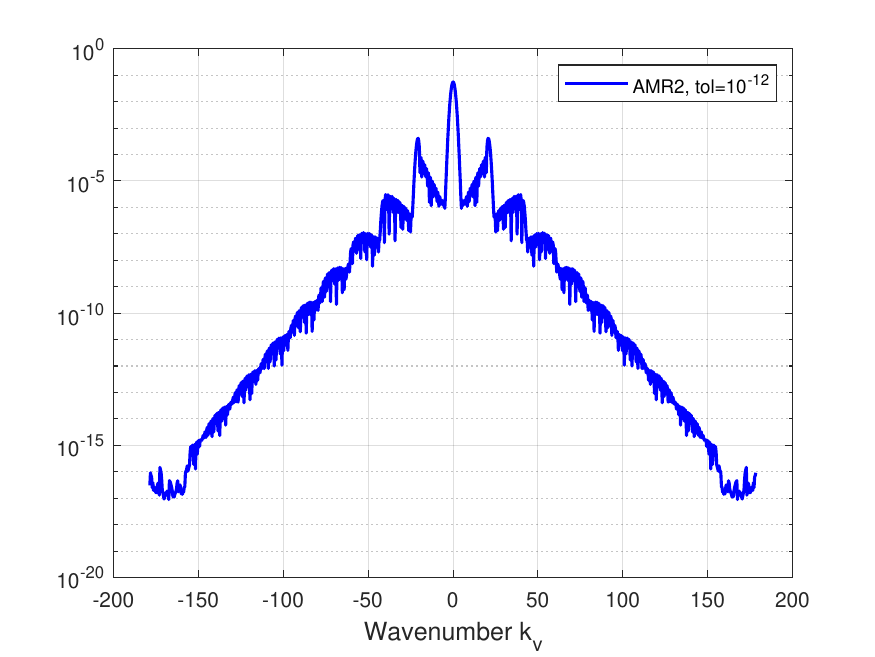}
		\subcaption{$t=40$} 
	\end{subfigure}
	\begin{subfigure}[b]{0.24\linewidth}		\includegraphics[width=1\linewidth]{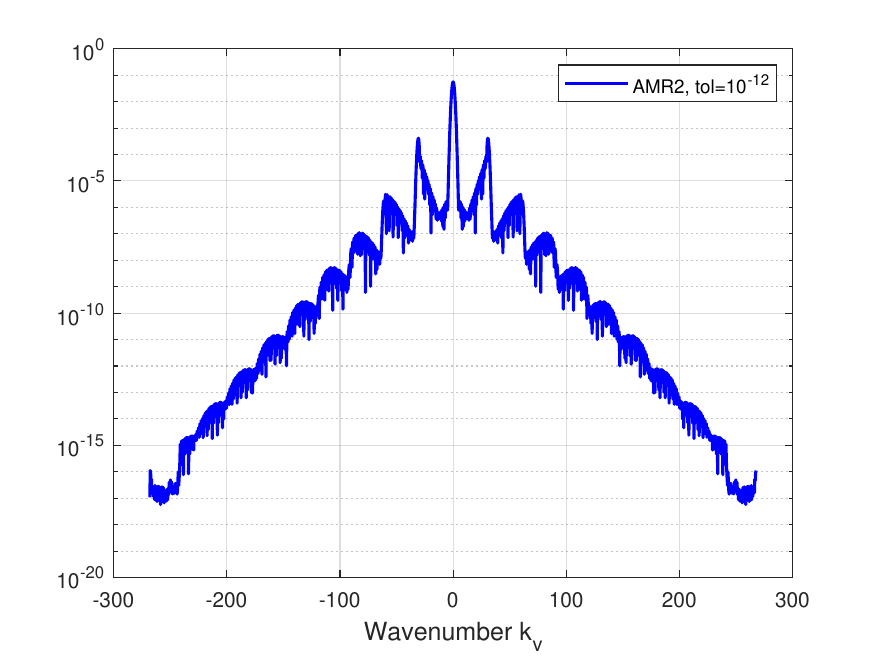}
		\subcaption{$t=60$} 
	\end{subfigure}		
	\begin{subfigure}[b]{0.24\linewidth}
		\includegraphics[width=1\linewidth]{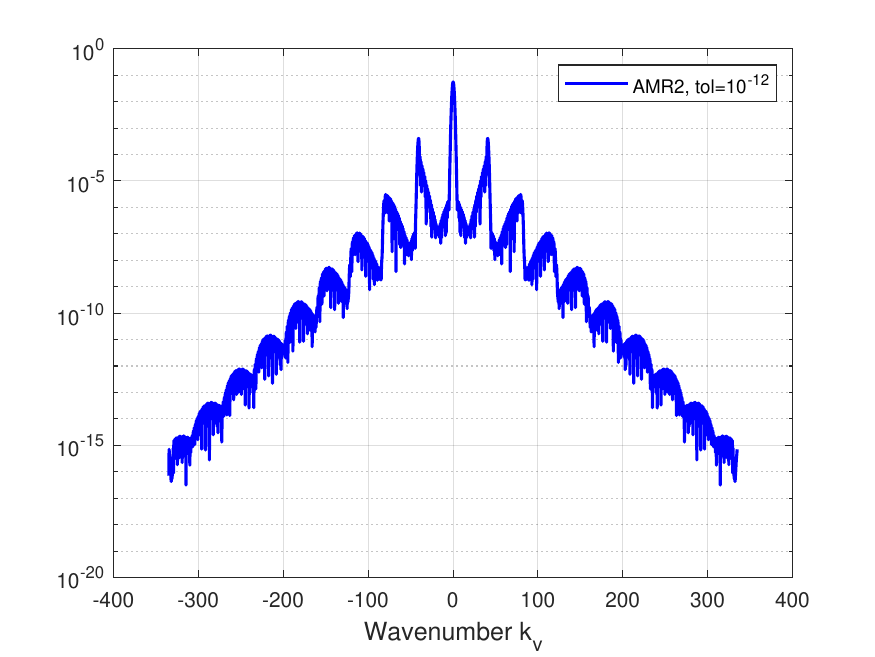}
		\subcaption{$t=80$} 
	\end{subfigure}	
    \caption{Weak Landau damping. Time evolution of the averaged Fourier mode magnitude in velocity $\frac{1}{N_x}\sum_l |\hat{f}(l,k_v)|$.
    Each row corresponds to Non-AMR with $N_x=32$ and various $N_v=512, 1024, 2048, 2816$. The last row reports the result for AMR2 with $\tau_{tol}=10^{-12}$).}\label{weak Fourier mode}
    
\end{figure}

					\subsection{Strong Landau damping}
					The next example is the well-known Strong Landau damping problem. The initial distribution function is the same as \eqref{init weak}, but with a larger perturbation amplitude $\alpha = 0.5$, which leads to the nonlinear behavior of the solutions. We set $k = 0.5$ and $v_{th} = 1$, and the computational domain is $[0,4\pi]\times [-9,9]$.

						\begin{figure}[h]
						\centering
						\begin{subfigure}[b]{0.4\linewidth}
							\includegraphics[width=1\linewidth]{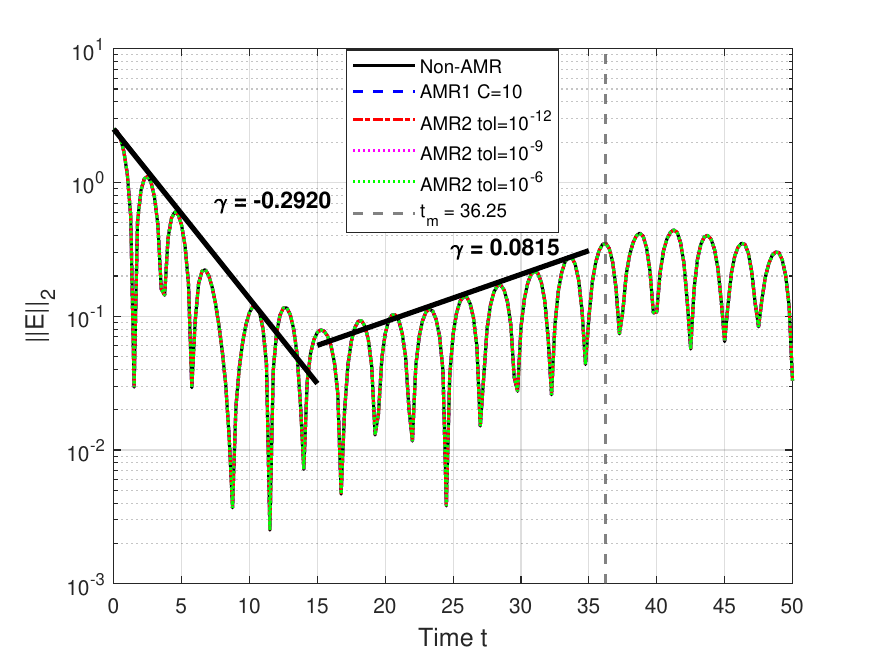}
							\subcaption{$L^2$-norm of $E$\\ \hspace{3mm}}\label{strong a} 
						\end{subfigure}	
					\begin{subfigure}[b]{0.4\linewidth}
						\includegraphics[width=1\linewidth]{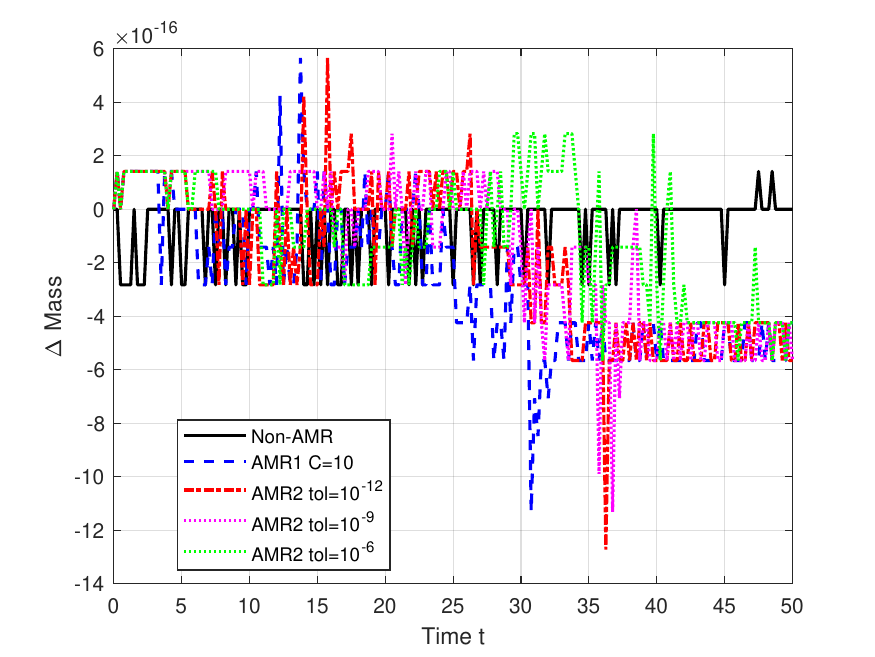}
						\subcaption{$\displaystyle \frac{\int f dvdx - \int f_0 dvdx}{\int f_0 dvdx}$}\label{strong b}
					\end{subfigure}	
					\begin{subfigure}[b]{0.4\linewidth}
					\includegraphics[width=1\linewidth]{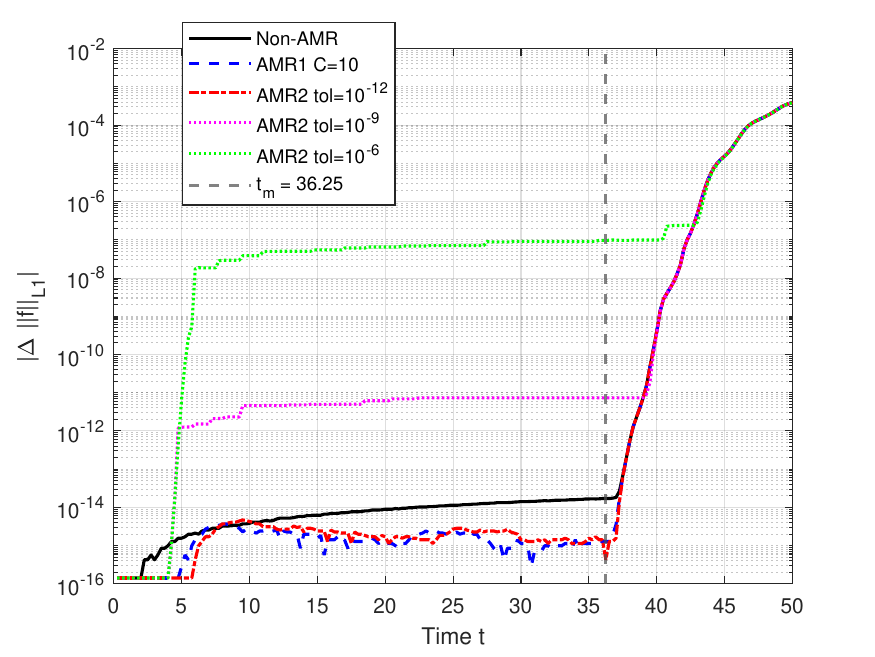}
					\subcaption{ $\displaystyle\frac{\|f^n\|_1-\|f^0\|_1}{\|f^0\|_1}$}\label{strong c}
				\end{subfigure}	
						\begin{subfigure}[b]{0.4\linewidth}
							\includegraphics[width=1\linewidth]{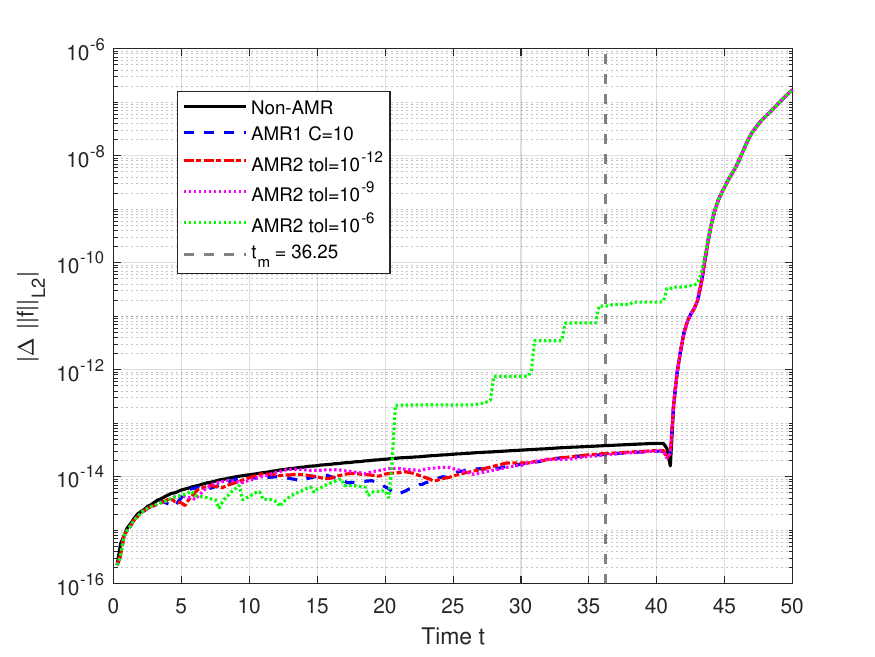}
							\subcaption{$\displaystyle\frac{\|f^n\|_2-\|f^0\|_2}{\|f^0\|_2}$}\label{strong d}
						\end{subfigure}	
					
						\begin{subfigure}[b]{0.4\linewidth}
						\includegraphics[width=1\linewidth]{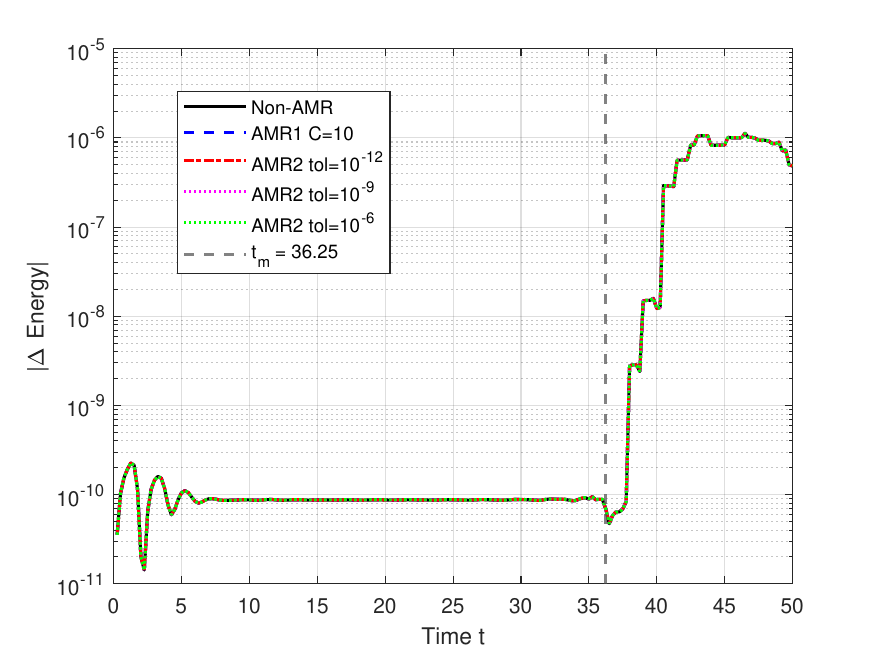}
						\subcaption{$\displaystyle \frac{Energy(t)-Energy(0)}{Energy(0)}$}\label{strong e}
					\end{subfigure}				
						\begin{subfigure}[b]{0.4\linewidth}
							\includegraphics[width=1\linewidth]{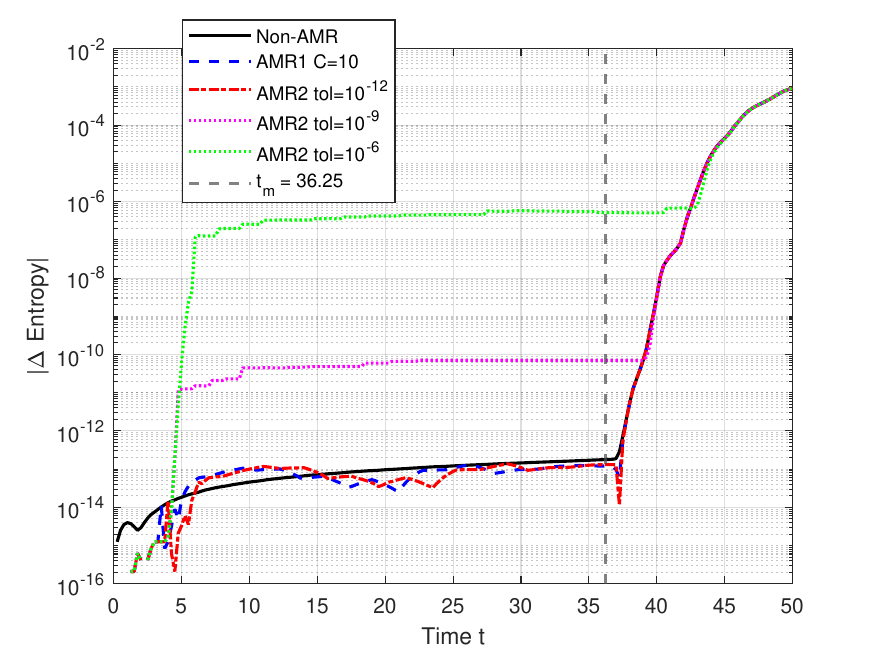}
							\subcaption{$\displaystyle \frac{Entropy(t)-Entropy(0)}{Entropy(0)}$}\label{strong f}
						\end{subfigure}

						\caption{Strong Landau damping.}\label{Fig strong 1}
					\end{figure}
                 
In Figure \ref{strong a}, we present the time evolution of the electric field in the $L^2$-norm, which demonstrates good agreement with the theoretical linear decay rate of $\gamma_1 = -0.2920$ and the nonlinear growth rate of $\gamma_2 = 0.0815$ \cite{RS}. Furthermore, Figures \ref{strong b}-\ref{strong f} display the time evolution of relative errors in the conserved quantities. In these figures, the vertical dashed line at $t_m \approx 36.25$ denotes the critical time at which the adaptive grid resolution for the AMR2 ($\tau_{tol}=10^{-12}$) case first hits the prescribed memory limit (see Figure \ref{strong N}).
        
During the dynamically resolved regime ($t < t_m$), the proposed adaptive method accurately preserves these physical invariants, validating its reliability and high fidelity. Beyond $t_m$, however, the depletion of available computational resources restricts any further grid expansion. Consequently, as indicated in the figures, most physical invariants—such as the $L^1$-norm, total energy, and entropy—begin to exhibit a noticeable loss of accuracy beyond $t_m$.

In contrast, as expected from the numerical formulation, the discrete total mass remains strictly invariant throughout the entire simulation. It shows no degradation even after $t_m$, confirming that the zero-padding operation rigorously preserves the fundamental total mass regardless of the grid resolution constraints. Note that the decrease in the $L^2$-norm implies the loss of high-frequency information due to the insufficient wavenumber domain; as the restricted grid can no longer accommodate the severe phase-space filamentation, unresolved high-wavenumber modes are inevitably truncated, leading to a gradual decay in spectral energy.

This limitation is inherently tied to the extreme physical demands of the strong Landau damping problem; severe phase-space filamentation causes the grid resolution to increase exponentially over time. As illustrated in Figure \ref{strong N}, the simulation reaches the hardware memory capacity near $t_m \approx 36.25$. Nevertheless, the results remain highly robust up to this critical threshold, demonstrating that our scheme extracts the maximum possible resolution and efficiency from the available computational resources before the onset of numerical degradation. 
		
The measured computational times reported in Figure \ref{strong cpu} demonstrate that the adaptive technique significantly reduces unnecessary computational overhead before reaching the hardware memory limit. Finally, the phase space profiles associated with AMR2 ($\tau_{tol}=10^{-12}$) are displayed in Figure \ref{strong phase}, clearly illustrating the complex vortex structures and severe filamentation.

               In contrast to the mild spectral broadening observed in the previous weak Landau damping test, the severe nonlinear vortex formation in this problem pushes the computational resources to their limits. To demonstrate this, in Figure \ref{strong Fourier mode} we illustrate the temporal evolution of the averaged Fourier modes. The first four rows present the results obtained using globally refined fixed grids with $N_x = 2816$ and progressively increasing velocity resolutions ($N_v = 4096, 8192$, $16384$ and $23936$),
                while the fifth row displays the results of the proposed AMR2 method ($\tau_{tol}=10^{-12}$). 
As clearly evidenced by the fixed-grid simulations: a smaller number of Fourier modes inevitably triggers an earlier onset of aliasing, as the limited spectral domain fails to accommodate the explosively expanding high-frequency tail. Prior to the critical time $t_m \approx 36.25$, the well-resolved configurations, including our adaptive scheme, accurately capture the complex spectral dynamics without numerical degradation. Beyond this time, it is evident that the magnitude of Fourier mode is large, which explains the noticeable loss of accuracy.

						\begin{figure}[h]
						\centering
						\begin{subfigure}[b]{0.4\linewidth}
							\includegraphics[width=1\linewidth]{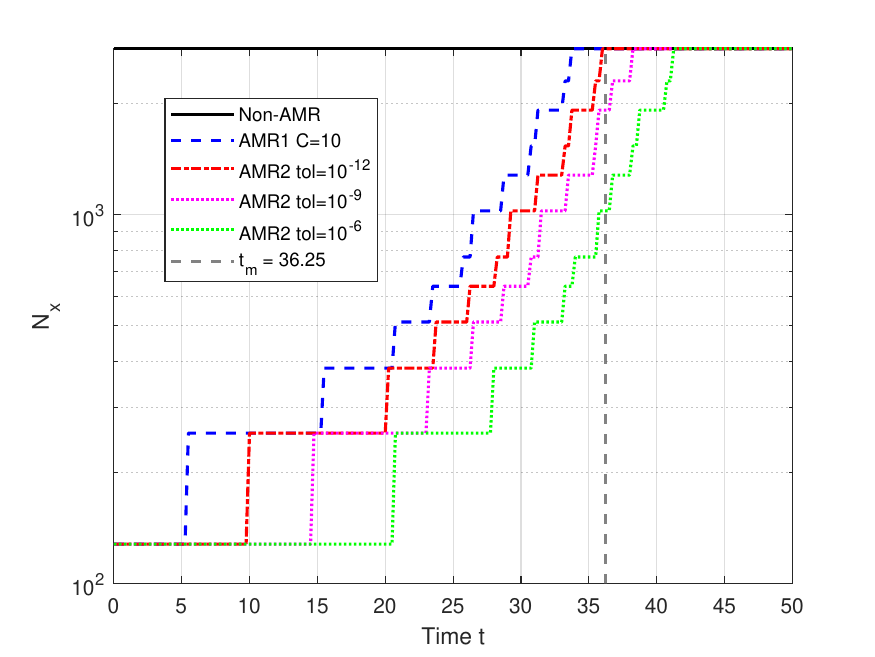}
							\subcaption{$N_x$} 
						\end{subfigure}	
						\begin{subfigure}[b]{0.4\linewidth}
							\includegraphics[width=1\linewidth]{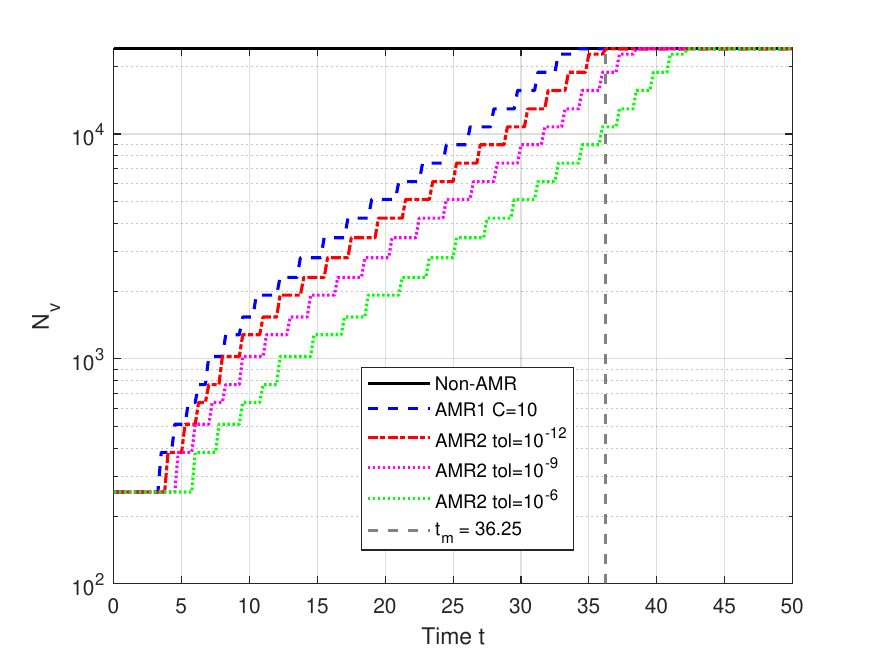}
							\subcaption{$N_v$} 
						\end{subfigure}	
						\begin{subfigure}[b]{0.4\linewidth}
							\includegraphics[width=1\linewidth]{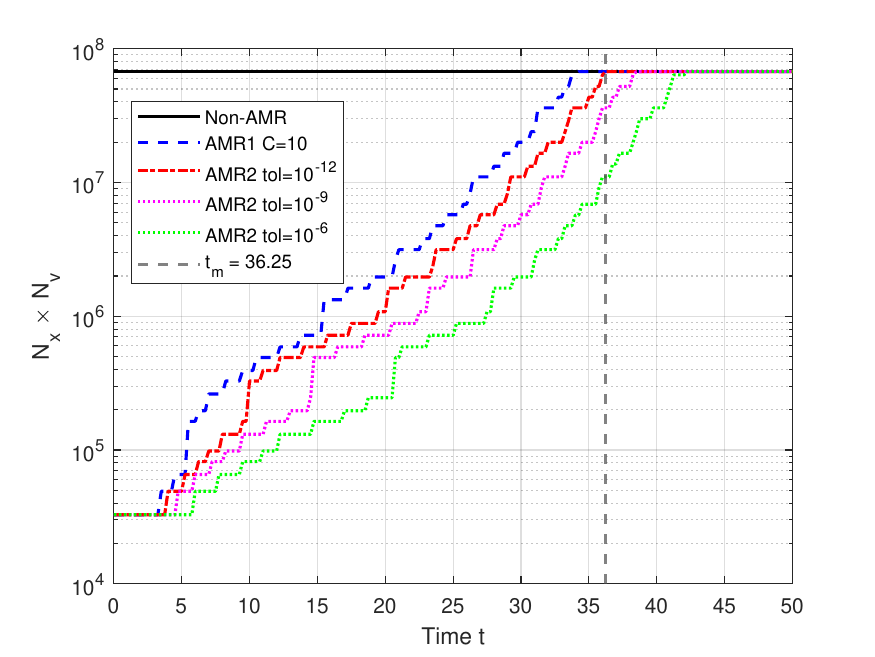}
							\subcaption{$N_x \times N_v$} 
						\end{subfigure}	
                        \begin{subfigure}[b]{0.4\linewidth}
						\includegraphics[width=1\linewidth]{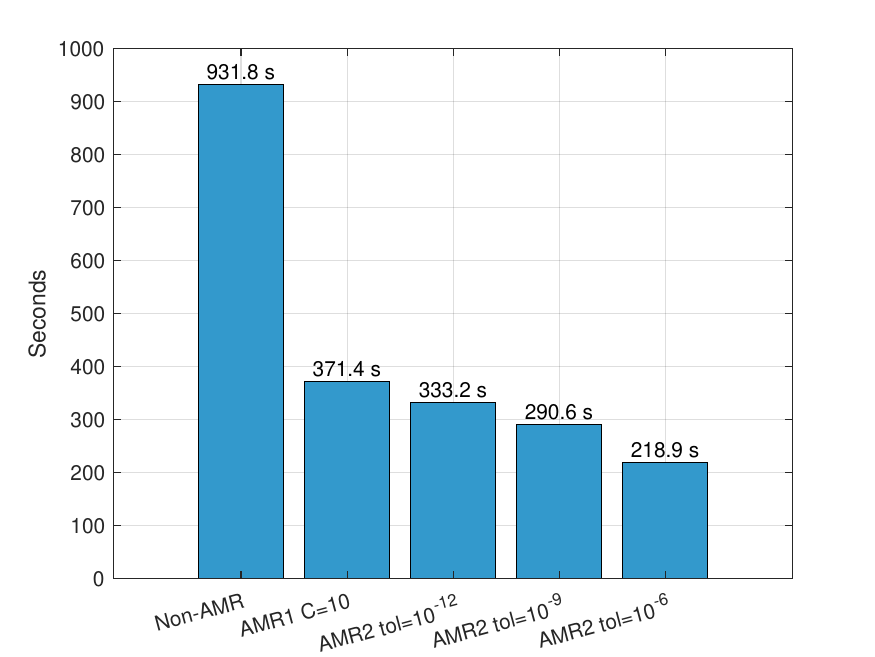}
						\subcaption{GPU time}\label{strong cpu} 
					\end{subfigure}
						\caption{Strong Landau damping. Time evolution of grid numbers and GPU time.}\label{strong N}
					\end{figure}		
\begin{figure}[h]
	\centering
	\begin{subfigure}[b]{0.32\linewidth}
		\includegraphics[width=1\linewidth]{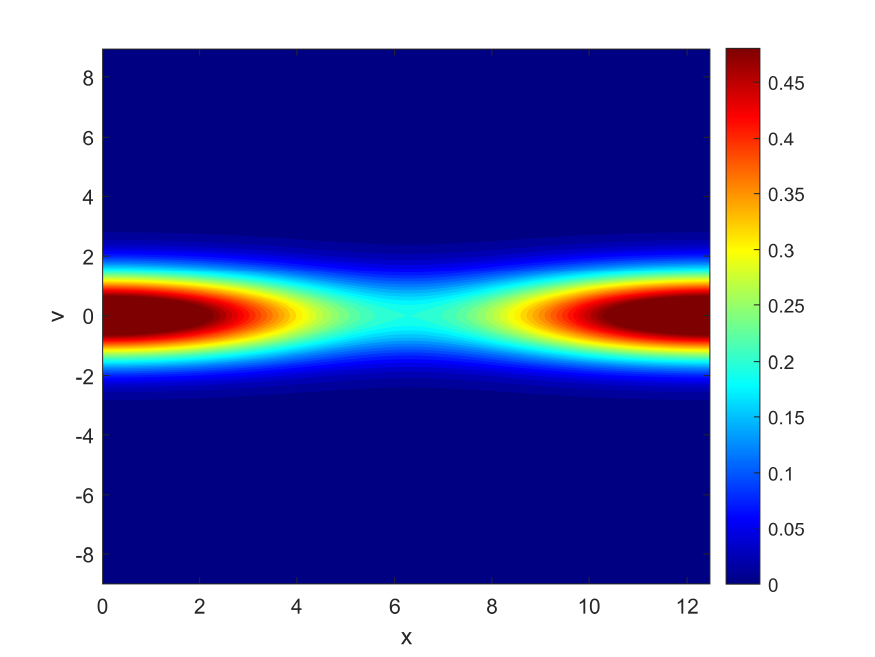}
		\subcaption{$t=0$} 
	\end{subfigure}
	\begin{subfigure}[b]{0.32\linewidth}
		\includegraphics[width=1\linewidth]{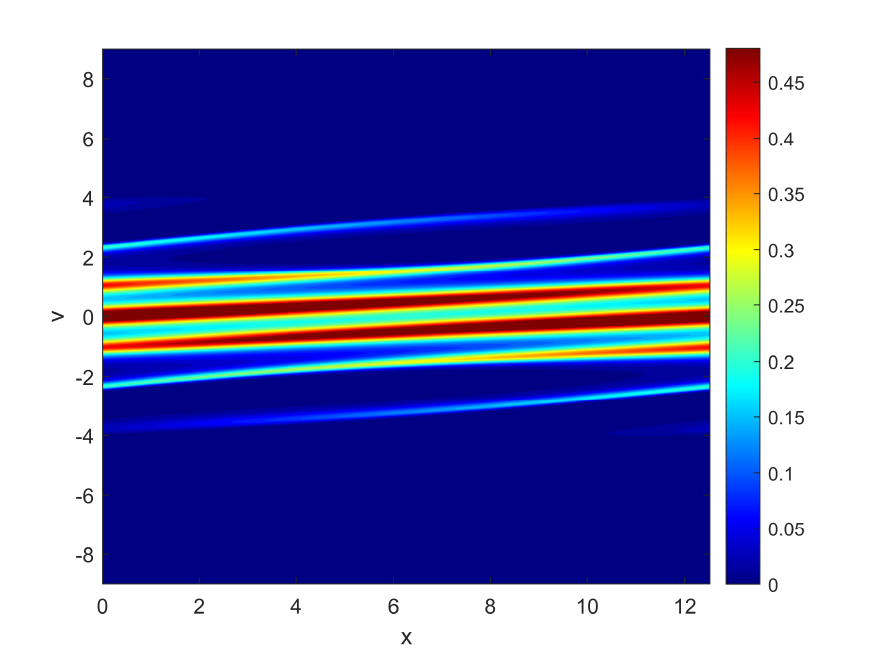}
		\subcaption{$t=10$} 
	\end{subfigure}
	\begin{subfigure}[b]{0.32\linewidth}
		\includegraphics[width=1\linewidth]{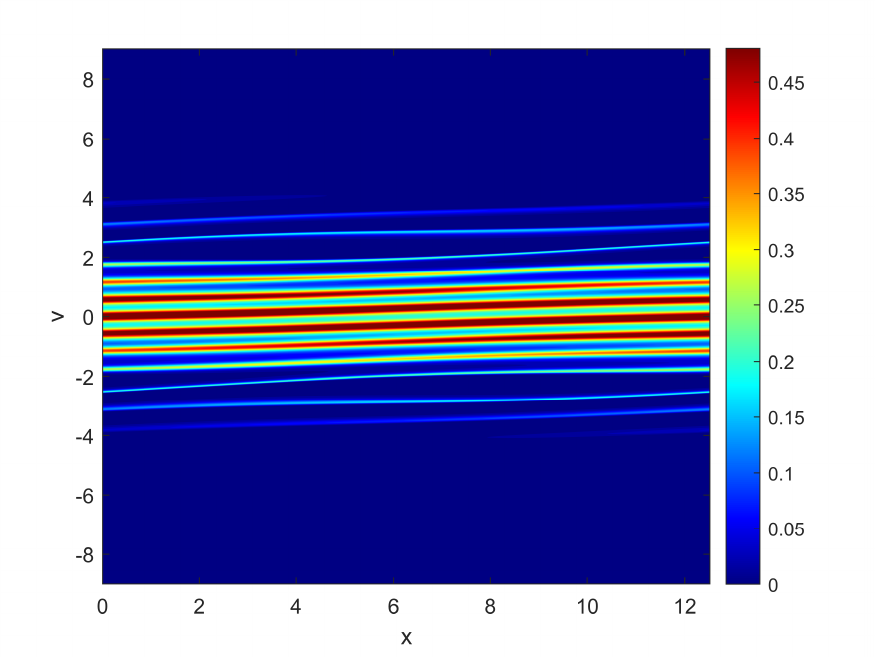}
		\subcaption{$t=20$} 
	\end{subfigure}
	\begin{subfigure}[b]{0.32\linewidth}
		\includegraphics[width=1\linewidth]{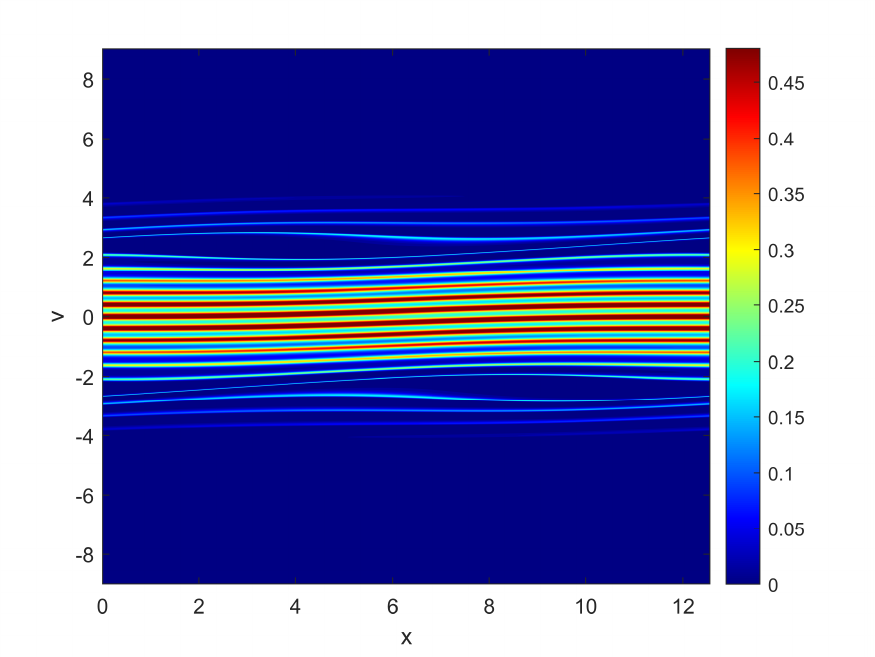}
		\subcaption{$t=30$} 
	\end{subfigure}
	\begin{subfigure}[b]{0.32\linewidth}		\includegraphics[width=1\linewidth]{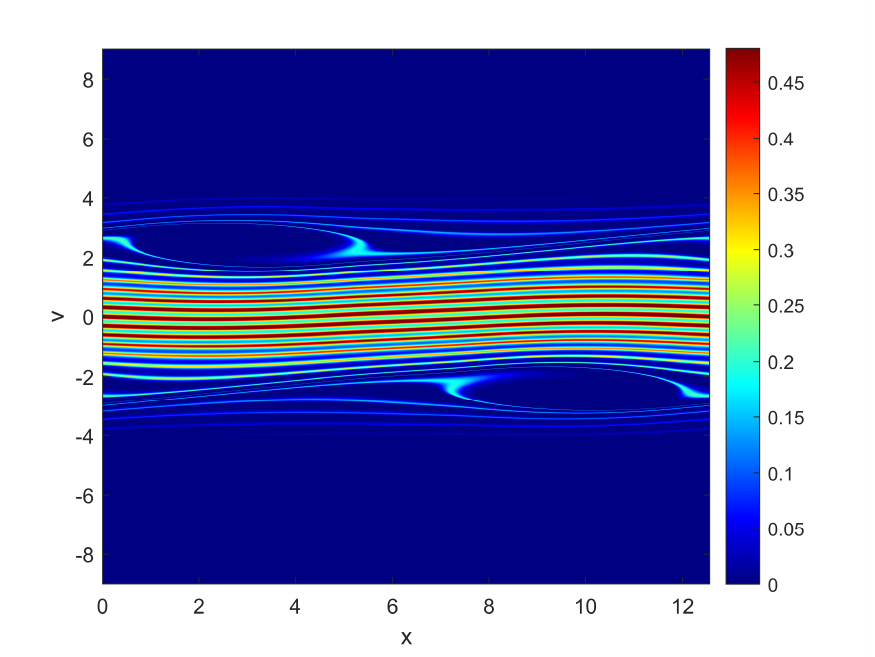}
		\subcaption{$t=40$} 
	\end{subfigure}		
	\begin{subfigure}[b]{0.32\linewidth}
		\includegraphics[width=1\linewidth]{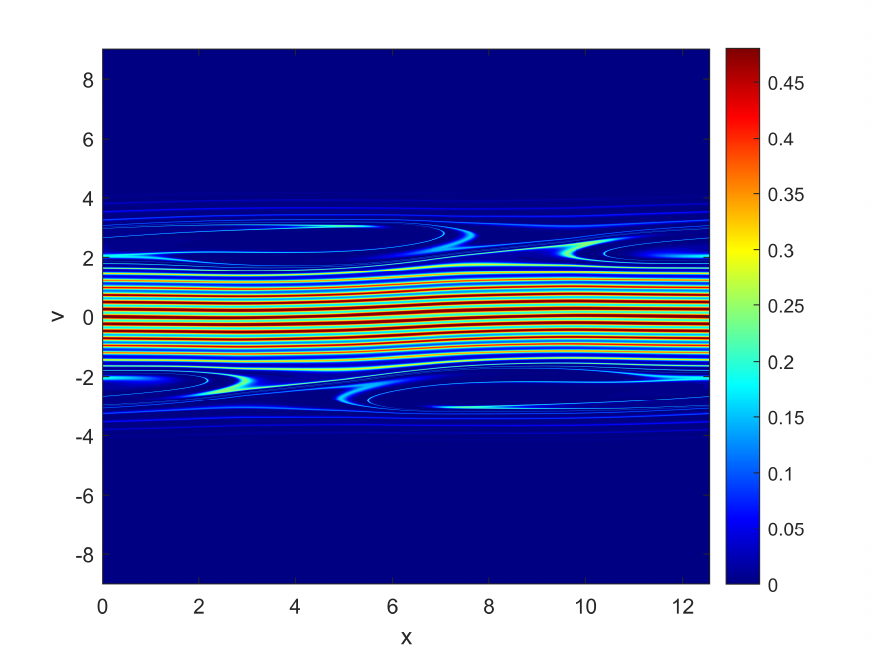}
		\subcaption{$t=50$} 
	\end{subfigure}				
	\caption{Strong Landau damping. Phase space profile of $f$.}\label{strong phase}
\end{figure}

\begin{figure}[h]
	\centering

    \begin{subfigure}[b]{0.24\linewidth}
		\includegraphics[width=1\linewidth]{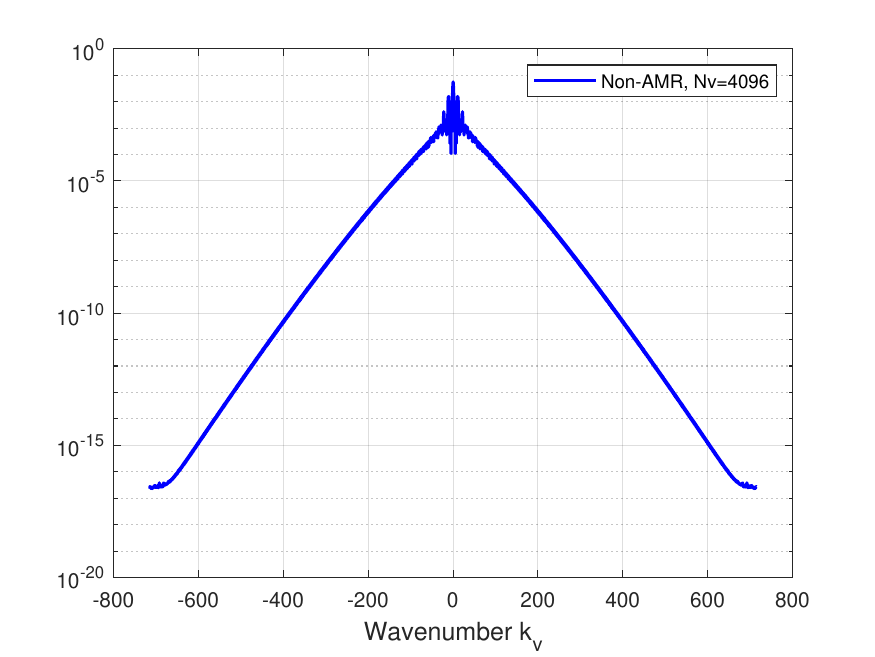}
		\subcaption{$t=20$} 
	\end{subfigure}
	\begin{subfigure}[b]{0.24\linewidth}
		\includegraphics[width=1\linewidth]{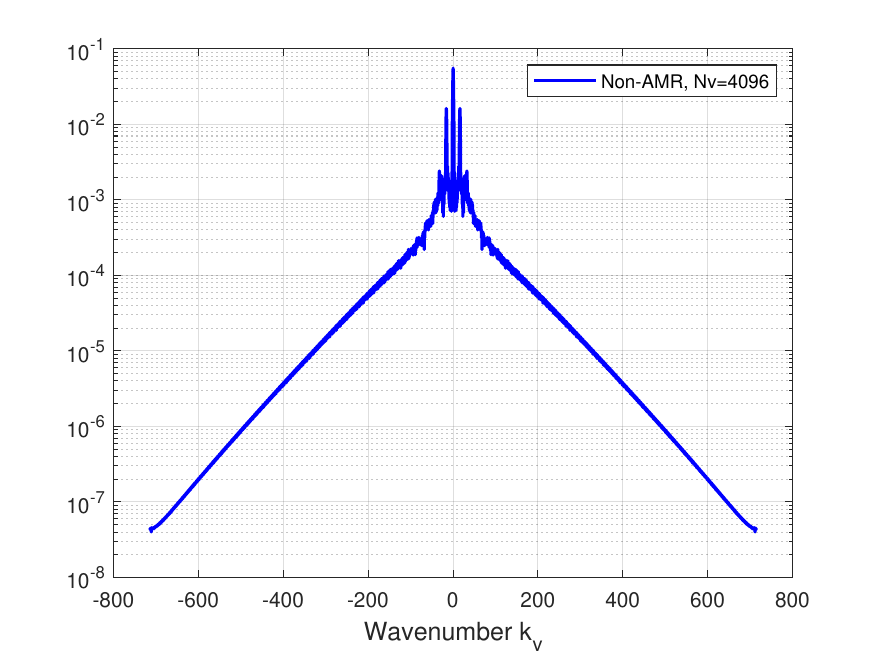}
		\subcaption{$t=30$} 
	\end{subfigure}
	\begin{subfigure}[b]{0.24\linewidth}		\includegraphics[width=1\linewidth]{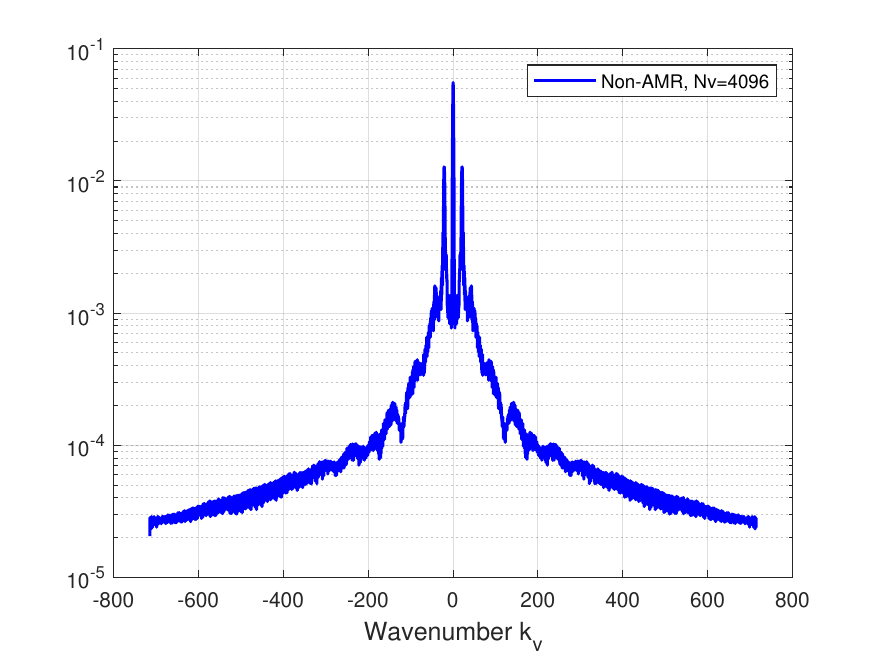}
		\subcaption{$t=40$} 
	\end{subfigure}		
	\begin{subfigure}[b]{0.24\linewidth}
		\includegraphics[width=1\linewidth]{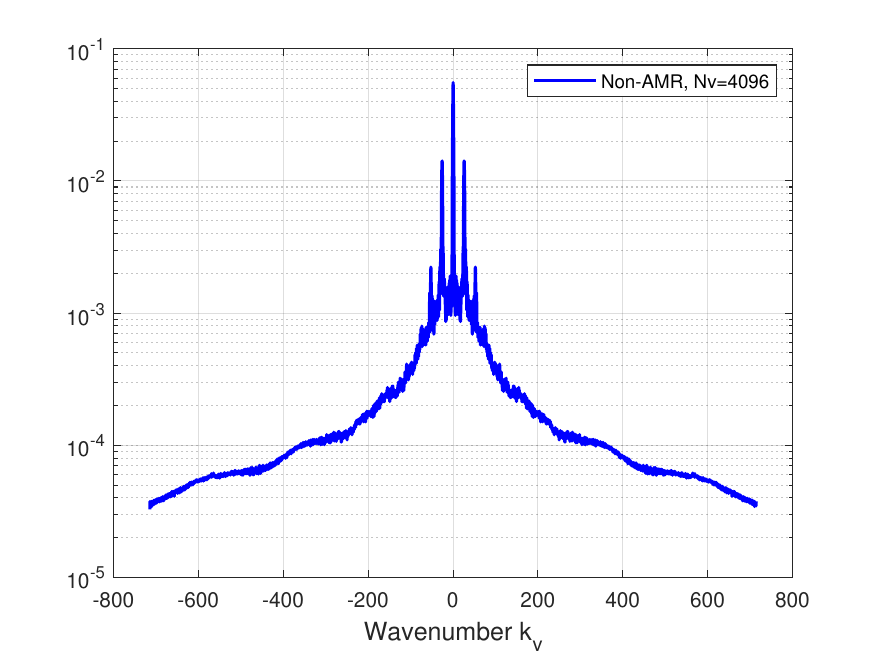}
		\subcaption{$t=50$} 
	\end{subfigure}		

    \begin{subfigure}[b]{0.24\linewidth}
		\includegraphics[width=1\linewidth]{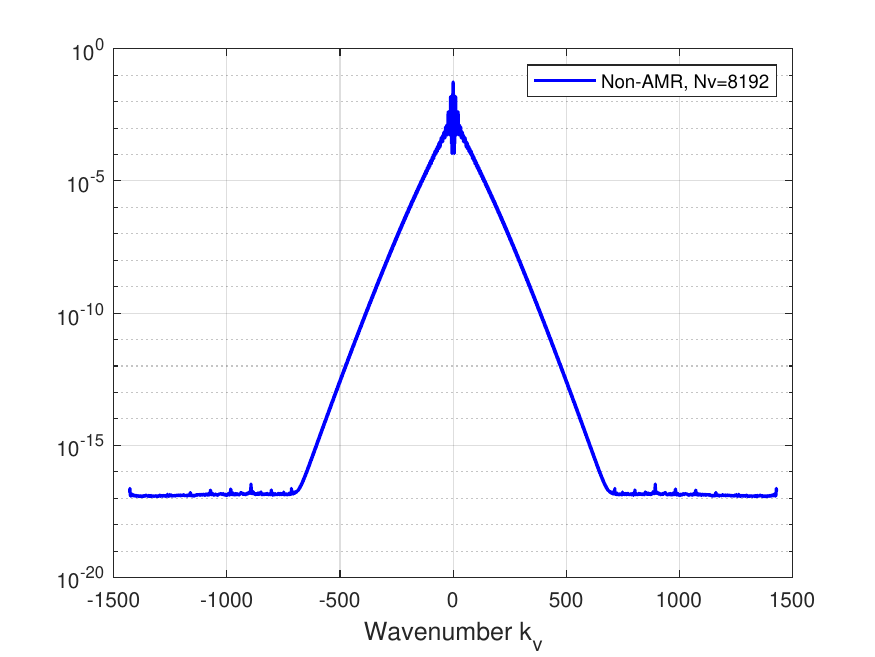}
		\subcaption{$t=20$} 
	\end{subfigure}
	\begin{subfigure}[b]{0.24\linewidth}
		\includegraphics[width=1\linewidth]{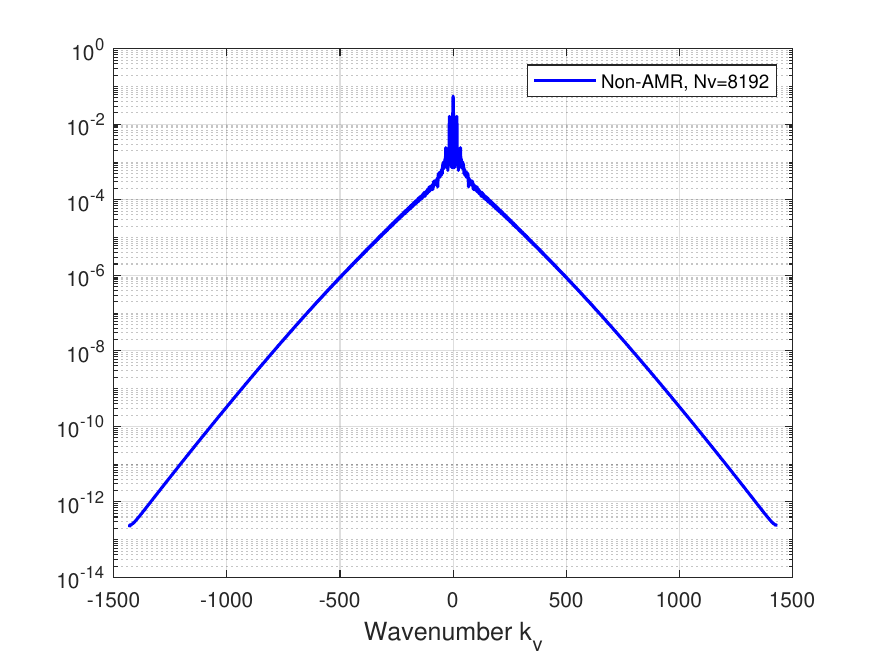}
		\subcaption{$t=30$} 
	\end{subfigure}
	\begin{subfigure}[b]{0.24\linewidth}		\includegraphics[width=1\linewidth]{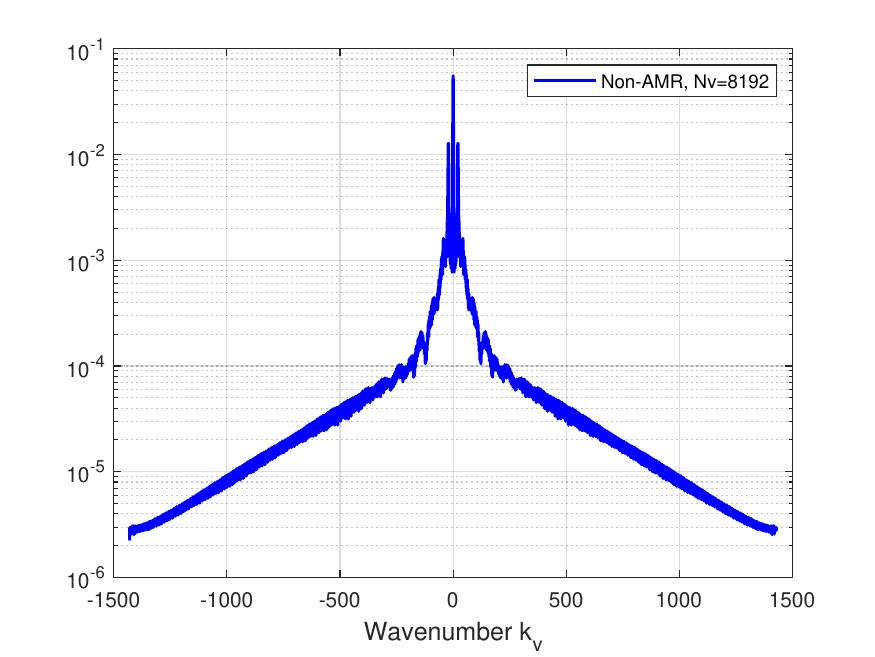}
		\subcaption{$t=40$} 
	\end{subfigure}		
	\begin{subfigure}[b]{0.24\linewidth}
		\includegraphics[width=1\linewidth]{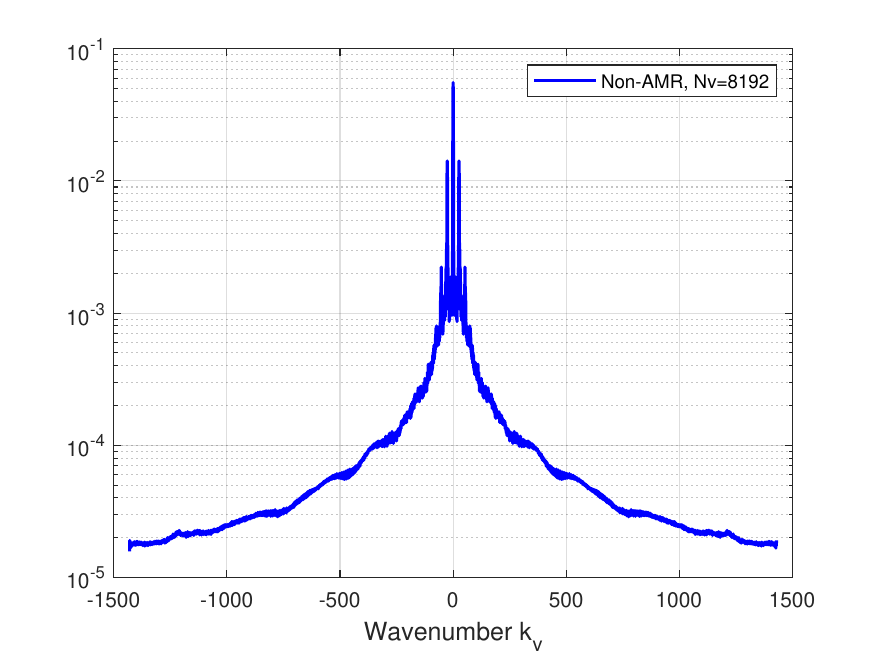}
		\subcaption{$t=50$} 
	\end{subfigure}		

    \begin{subfigure}[b]{0.24\linewidth}
		\includegraphics[width=1\linewidth]{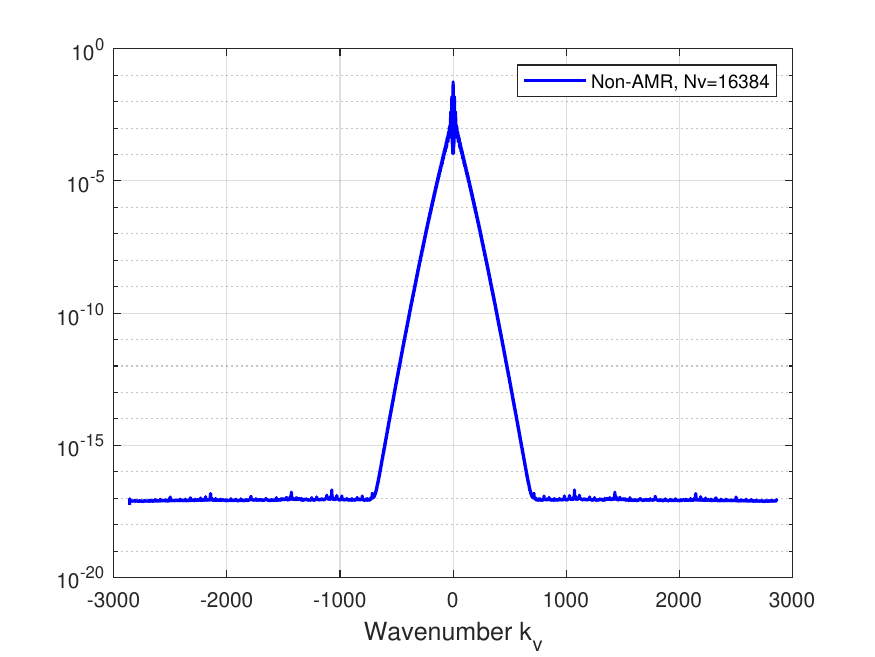}
		\subcaption{$t=20$} 
	\end{subfigure}
	\begin{subfigure}[b]{0.24\linewidth}
		\includegraphics[width=1\linewidth]{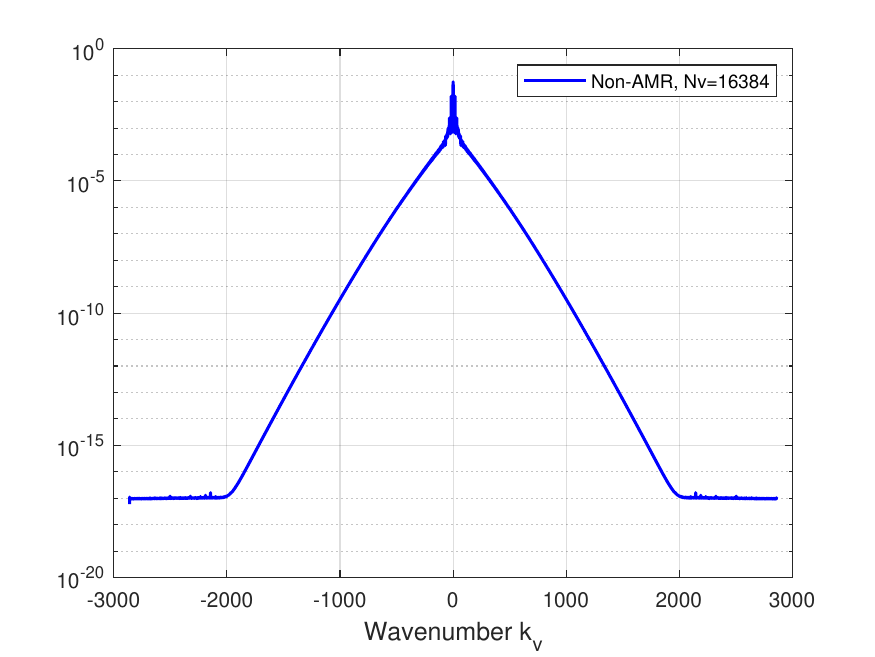}
		\subcaption{$t=30$} 
	\end{subfigure}
	\begin{subfigure}[b]{0.24\linewidth}		\includegraphics[width=1\linewidth]{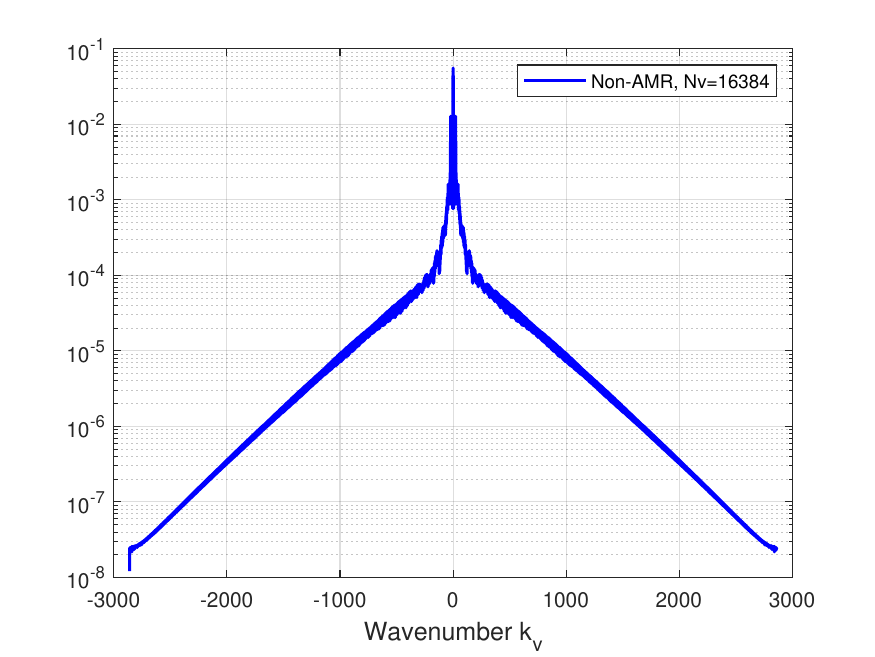}
		\subcaption{$t=40$} 
	\end{subfigure}		
	\begin{subfigure}[b]{0.24\linewidth}
		\includegraphics[width=1\linewidth]{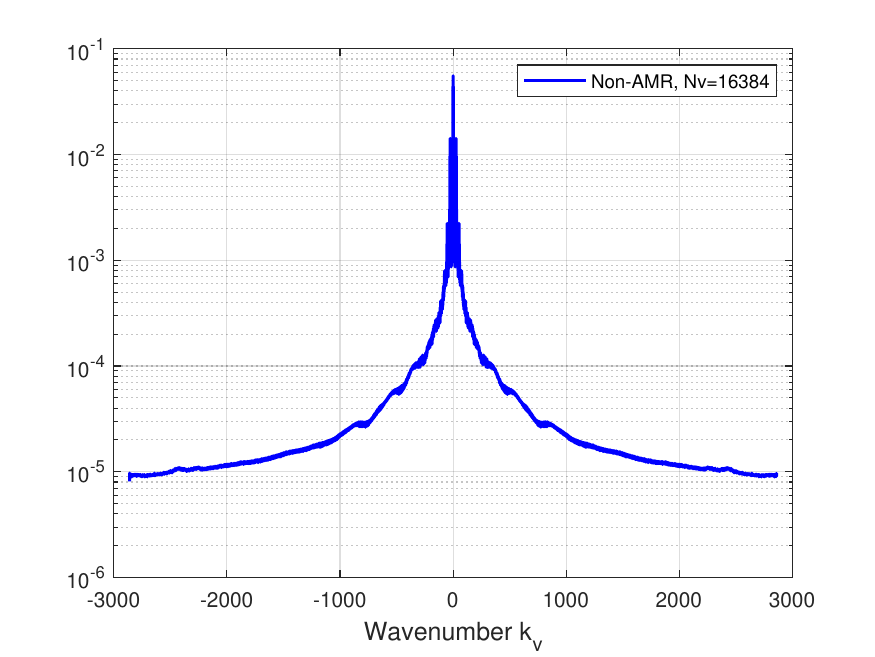}
		\subcaption{$t=50$} 
	\end{subfigure}

\begin{subfigure}[b]{0.24\linewidth}
		\includegraphics[width=1\linewidth]{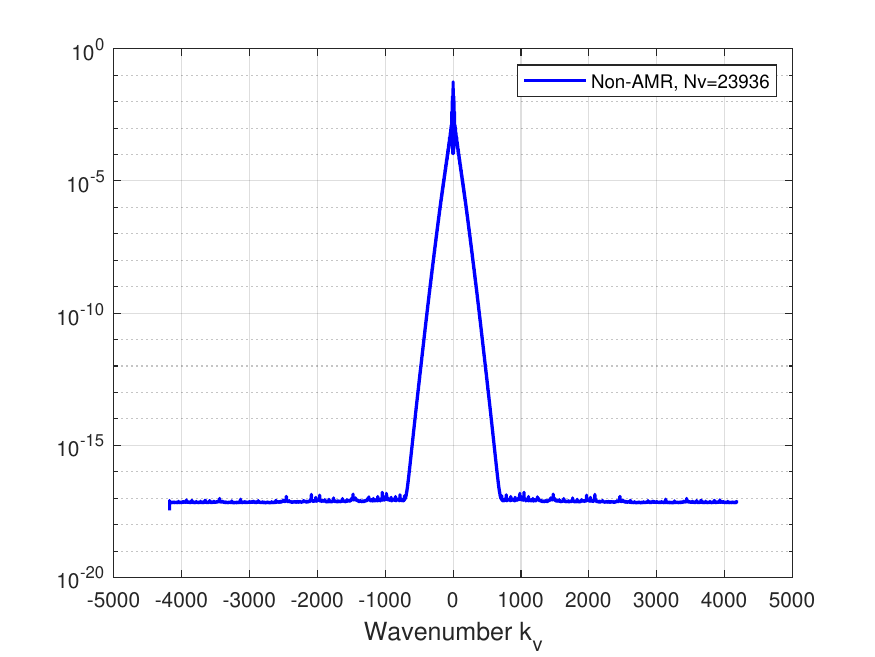}
		\subcaption{$t=20$} 
	\end{subfigure}
	\begin{subfigure}[b]{0.24\linewidth}
		\includegraphics[width=1\linewidth]{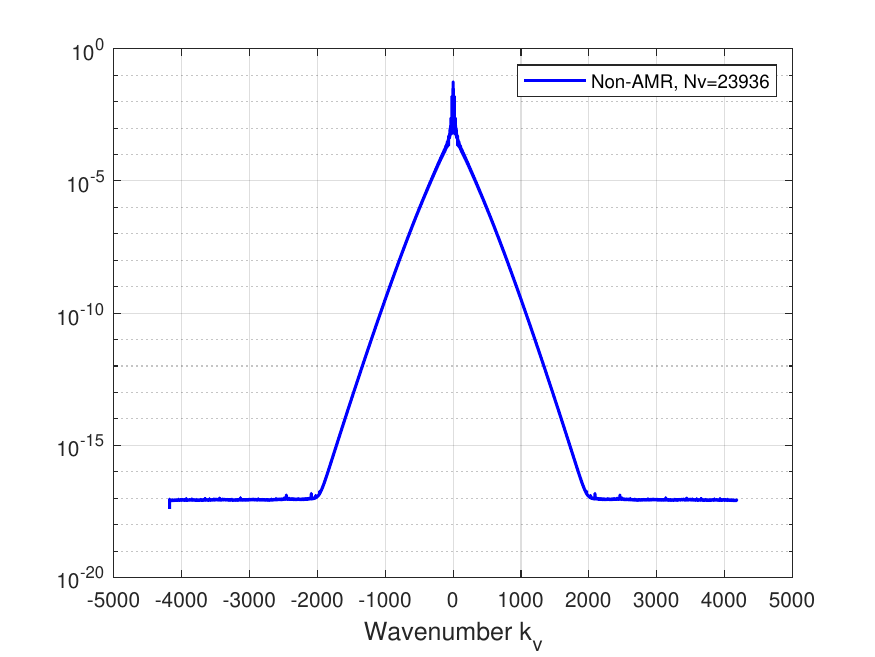}
		\subcaption{$t=30$} 
	\end{subfigure}
	\begin{subfigure}[b]{0.24\linewidth}		\includegraphics[width=1\linewidth]{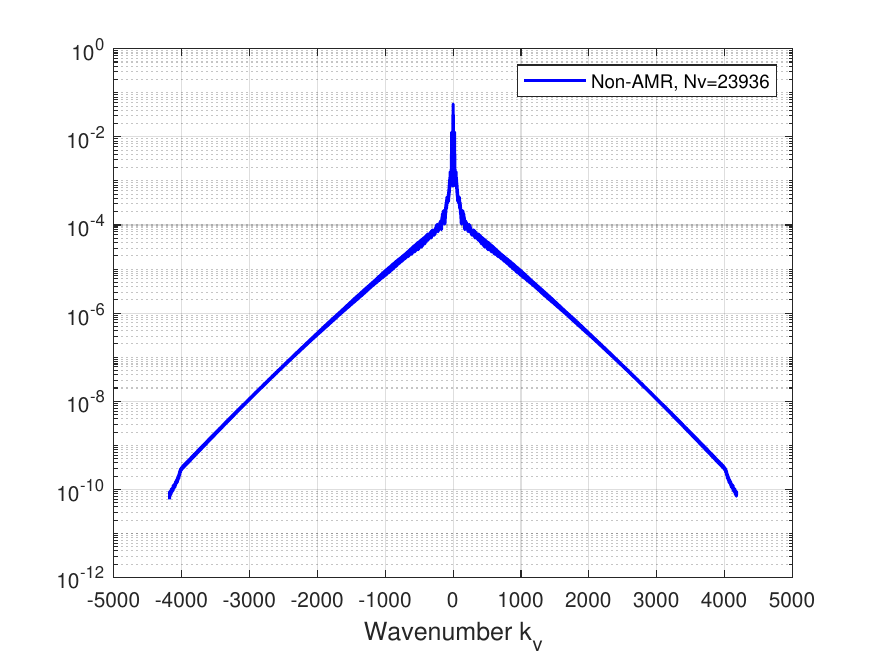}
		\subcaption{$t=40$} 
	\end{subfigure}		
	\begin{subfigure}[b]{0.24\linewidth}
		\includegraphics[width=1\linewidth]{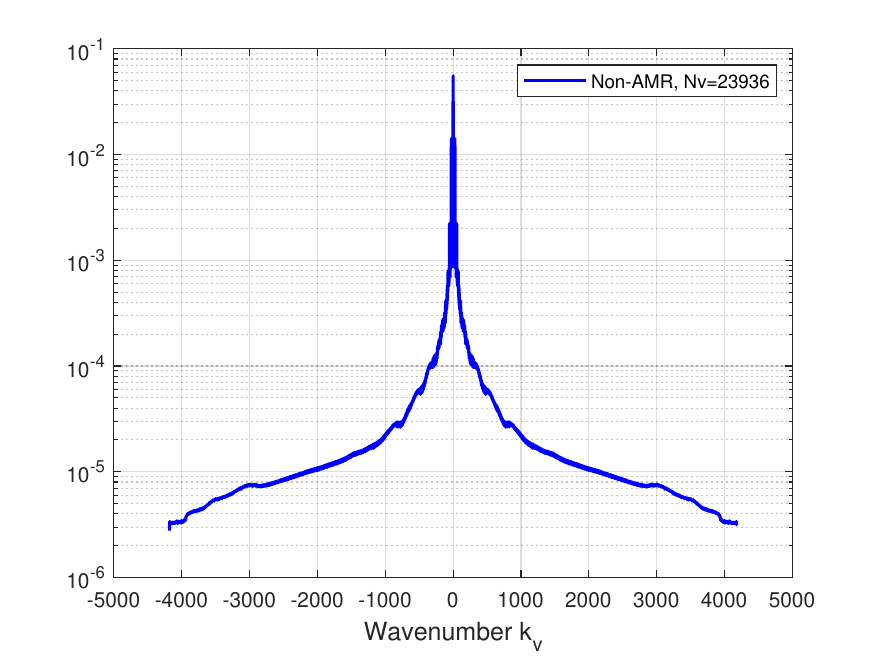}
		\subcaption{$t=50$} 
	\end{subfigure}		
     \begin{subfigure}[b]{0.24\linewidth}
		\includegraphics[width=1\linewidth]{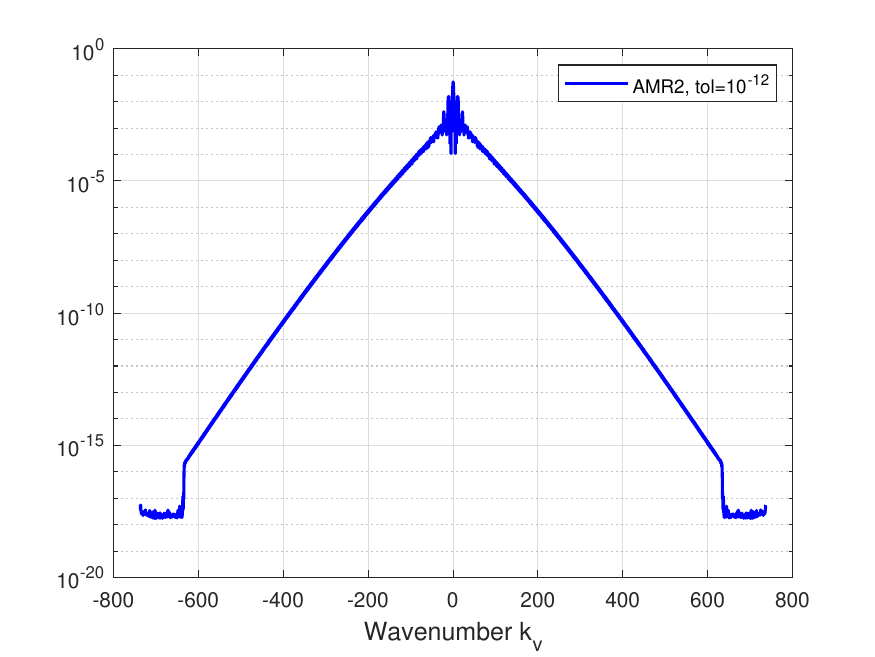}
		\subcaption{$t=20$} 
	\end{subfigure}
	\begin{subfigure}[b]{0.24\linewidth}
		\includegraphics[width=1\linewidth]{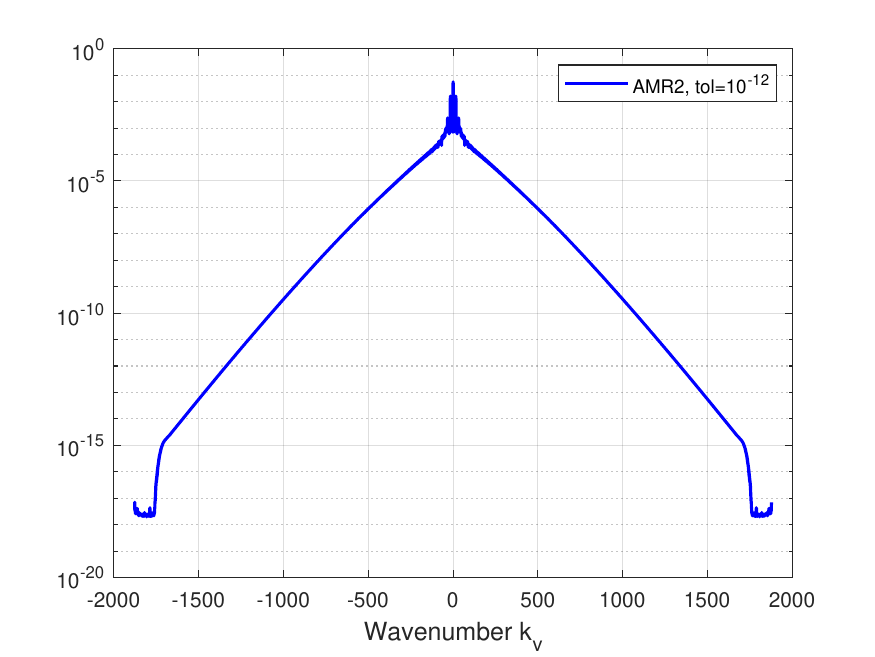}
		\subcaption{$t=30$} 
	\end{subfigure}
	\begin{subfigure}[b]{0.24\linewidth}		\includegraphics[width=1\linewidth]{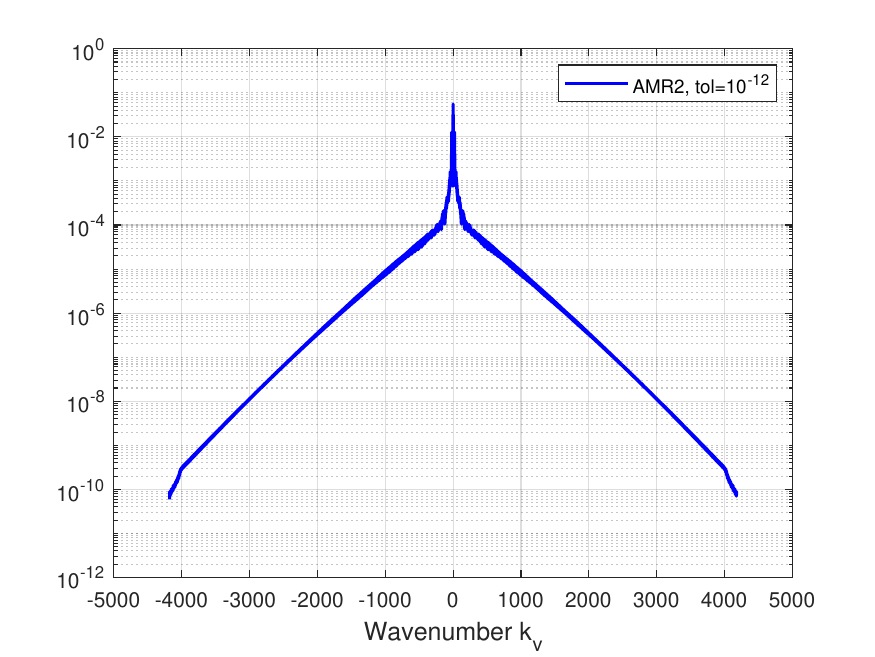}
		\subcaption{$t=40$} 
	\end{subfigure}		
	\begin{subfigure}[b]{0.24\linewidth}
		\includegraphics[width=1\linewidth]{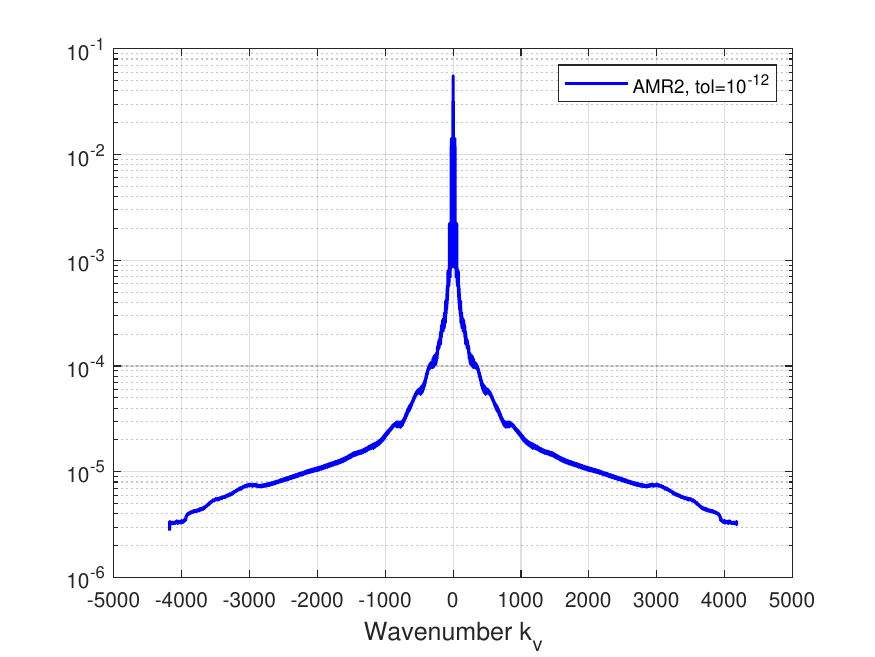}
		\subcaption{$t=50$} 
	\end{subfigure}	
    \caption{Strong Landau damping. Time evolution of the averaged Fourier mode magnitude $\frac{1}{N_x}\sum_l |\hat{f}(l,k_v)|$. Each row corresponds to Non-AMR with $N_x=2816$ and various $N_v=4096, 8192, 16384, 23936$. The last row reports the result for AMR2 with $\tau_{tol}=10^{-12}$).}\label{strong Fourier mode}
    
\end{figure}

					\begin{remark}
						It is worth noting two important implications regarding the hardware memory limit. First, hitting this adaptive limit serves as a natural diagnostic indicator; it explicitly informs us of the exact moment when the numerical solution begins to lose its reliability. Second, under strong nonlinear effects, any simulation will inevitably exhaust its memory capacity at a certain point due to the exponentially growing grid demand. Determining which numerical scheme performs better in the severely unresolved regime beyond this limit remains an open question, which falls outside the scope of this paper.
					\end{remark}

					\subsection{Two-stream Instability}
					We consider the benchmark problem of the two-stream instability. The initial distribution function is given by
					\begin{equation*}
						f_0(x, v) = \frac{2}{7\sqrt{2\pi}} (1+5v^2) \left( 1 + \alpha \left[ \frac{\cos(2kx) + \cos(3kx)}{1.2} + \cos(kx) \right] \right) \exp\left(-\frac{v^2}{2}\right),
					\end{equation*} 
					where $\alpha=0.01$ and $k=0.5$. The computational domain is $[0,4\pi]\times[-10,10]$.

					In Figure \ref{Two a}, the time evolution of the electric field in the $L^2$-norm is depicted. Driven by the multiple spatial modes ($k, 2k, 3k$), the electric field experiences exponential growth in the linear regime followed by a highly nonlinear saturation phase. Figures \ref{Two b}-\ref{Two f} illustrate the relative errors of the conserved quantities. Consistent with previous observations, the simulations dynamically resolve the complex physics until reaching the hardware memory limit at $t_m \approx 27.25$. Beyond this time, we observe an inevitable loss of conservation in most physical quantities due to unresolved high-frequency modes. In contrast, the discrete total mass remains exactly conserved up to machine precision, which confirms the structural robustness of the proposed method. 
                    The grid adaptation procedure and the associated computational time in Figure \ref{Two N} support the accuracy and efficiency of the method. The relevant phase-space dynamics with vortex formation and rapid generation of fine-scale filaments are depicted in Figure \ref{twostream phase}. To sum up, the AMR models allow us to obtain high accurate solution with significantly less average memory and computational time than a globally refined fixed grid.

						\begin{figure}[h]
						\centering
						\begin{subfigure}[b]{0.4\linewidth}
							\includegraphics[width=1\linewidth]{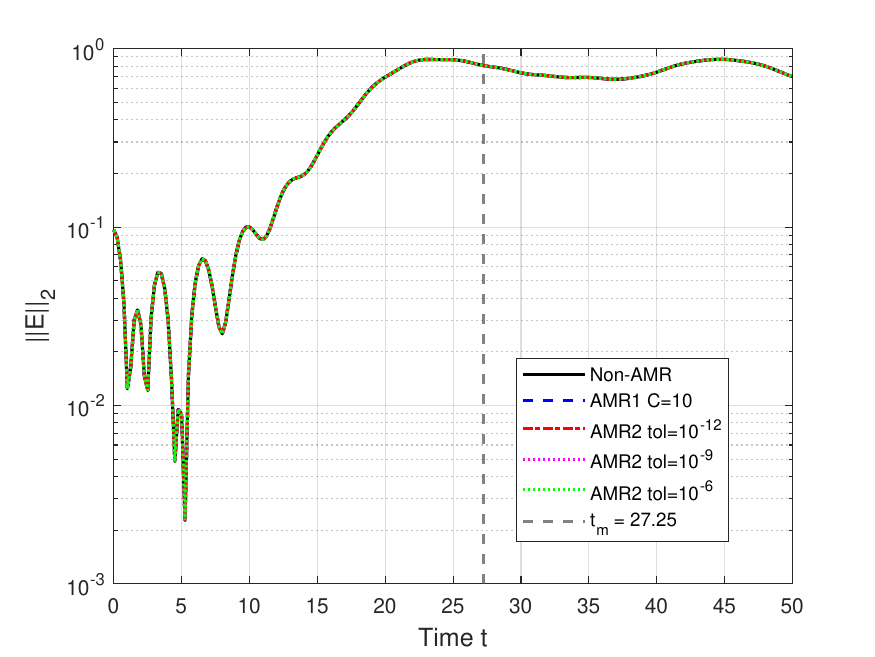}
							\subcaption{$L^2$-norm of $E$\\ \hspace{3mm}} \label{Two a}
						\end{subfigure}	
					\begin{subfigure}[b]{0.4\linewidth}
						\includegraphics[width=1\linewidth]{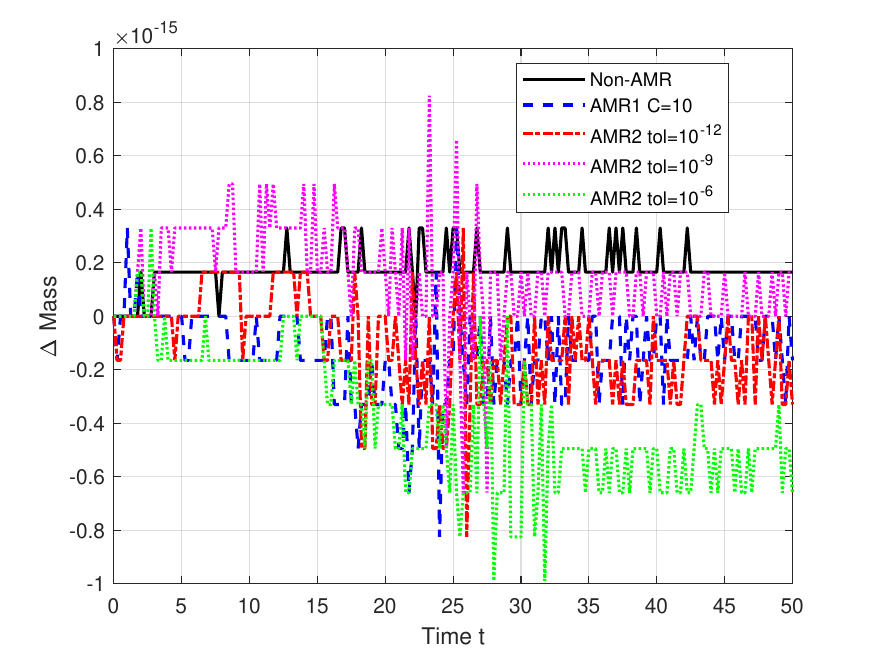}
						\subcaption{$\displaystyle \frac{\int f dvdx - \int f_0 dvdx}{\int f_0 dvdx}$}\label{Two b}
					\end{subfigure}	
					\begin{subfigure}[b]{0.4\linewidth}
					\includegraphics[width=1\linewidth]{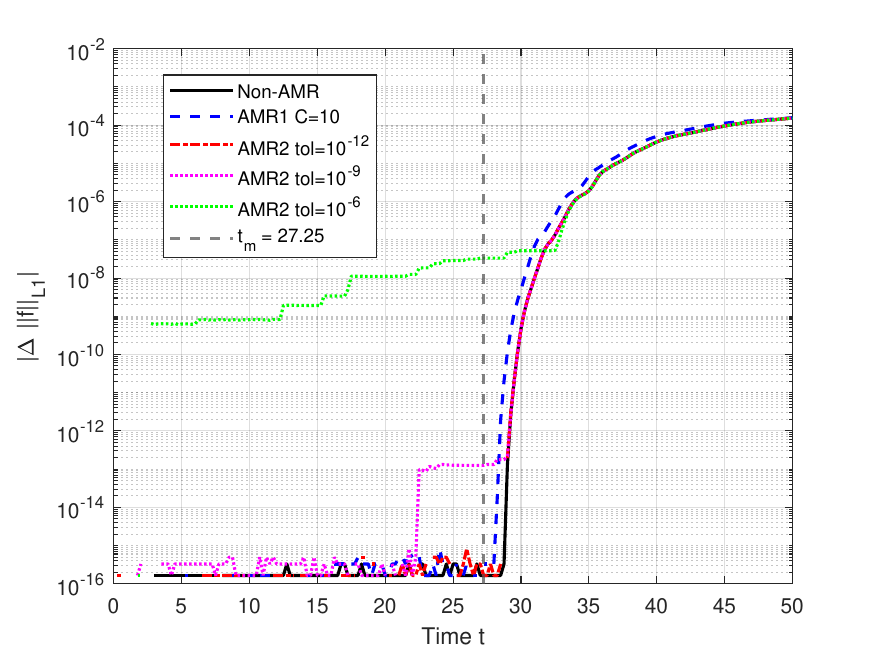}
					\subcaption{ $\displaystyle\frac{\|f^n\|_1-\|f^0\|_1}{\|f^0\|_1}$}\label{Two c}
				\end{subfigure}
						\begin{subfigure}[b]{0.4\linewidth}
							\includegraphics[width=1\linewidth]{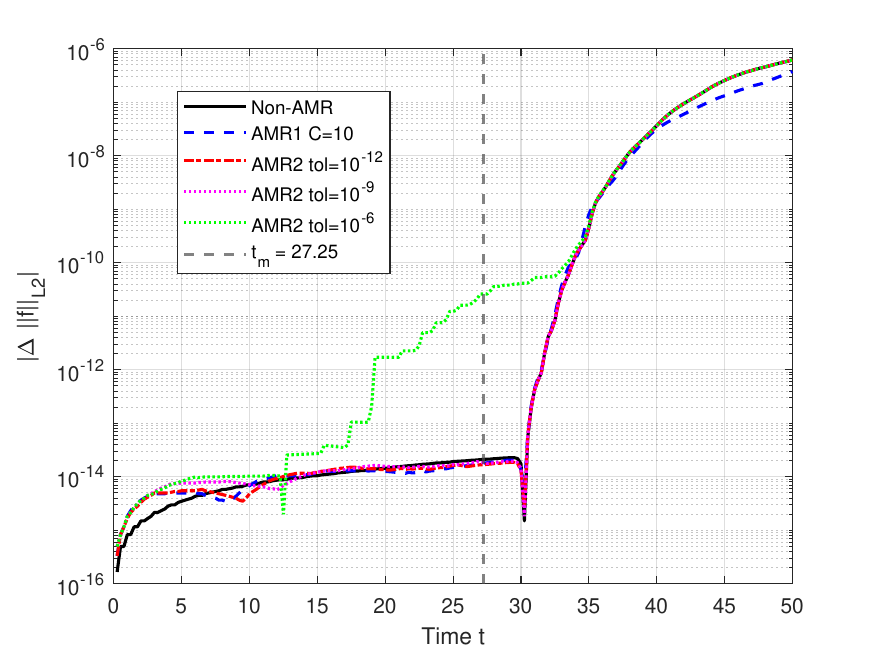}
							\subcaption{$\displaystyle\frac{\|f^n\|_2-\|f^0\|_2}{\|f^0\|_2}$}\label{Two d}
						\end{subfigure}	
						\begin{subfigure}[b]{0.4\linewidth}
						\includegraphics[width=1\linewidth]{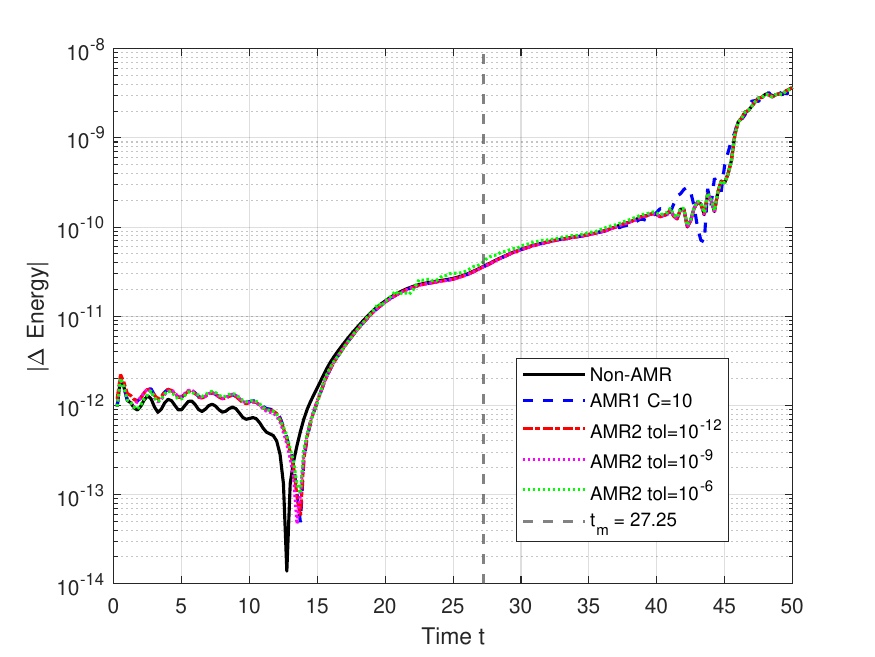}
						\subcaption{$\displaystyle \frac{Energy(t)-Energy(0)}{Energy(0)}$}\label{Two e}
					\end{subfigure}	
						\begin{subfigure}[b]{0.4\linewidth}
							\includegraphics[width=1\linewidth]{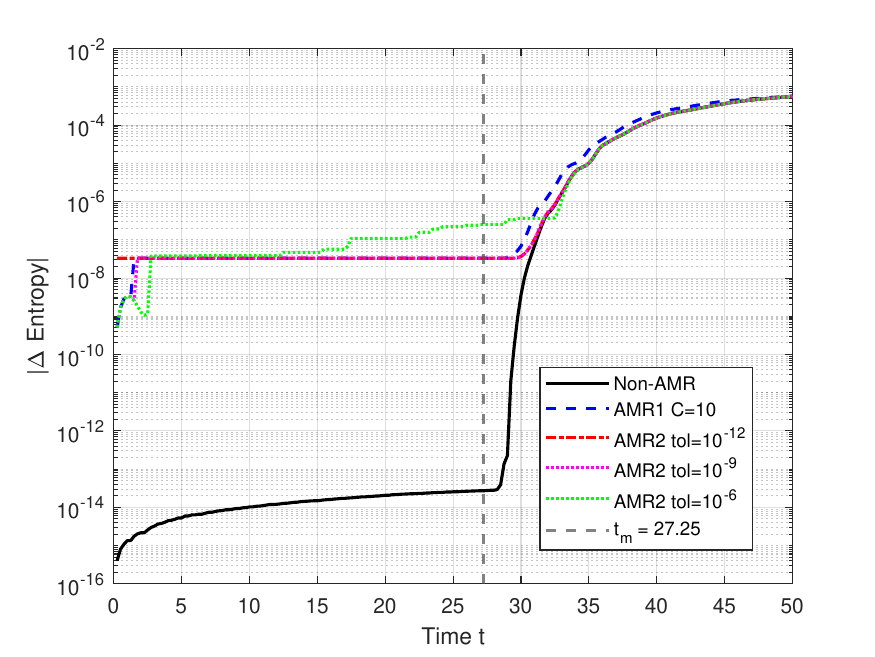}
							\subcaption{$\displaystyle \frac{Entropy(t)-Entropy(0)}{Entropy(0)}$}\label{Two f}
						\end{subfigure}			
						\caption{Two-stream Instability.}\label{Fig twostream 1}
					\end{figure}
					
					\begin{figure}[h]
						\centering
						\begin{subfigure}[b]{0.4\linewidth}
							\includegraphics[width=1\linewidth]{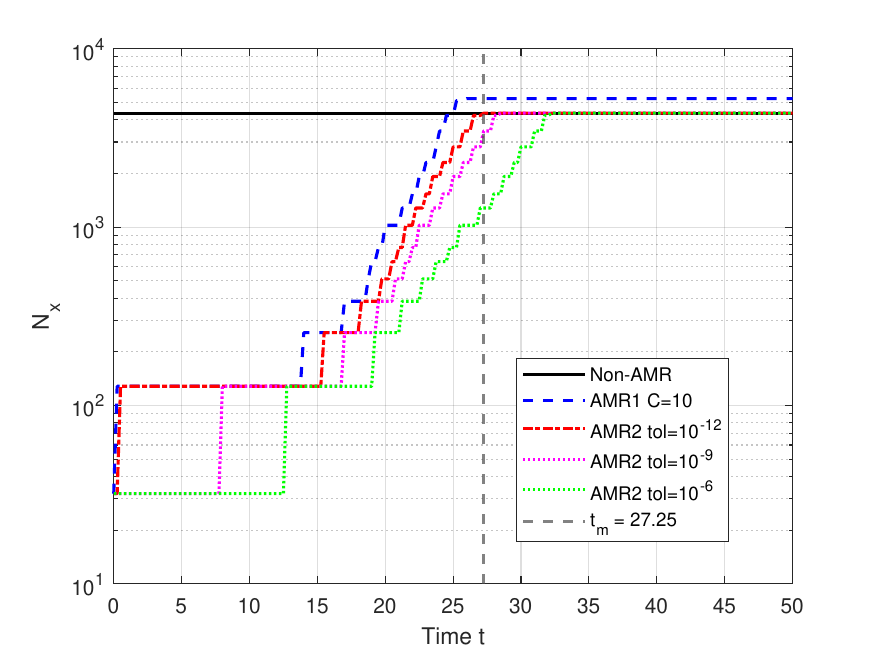}
							\subcaption{$N_x$} 
						\end{subfigure}	
						\begin{subfigure}[b]{0.4\linewidth}
							\includegraphics[width=1\linewidth]{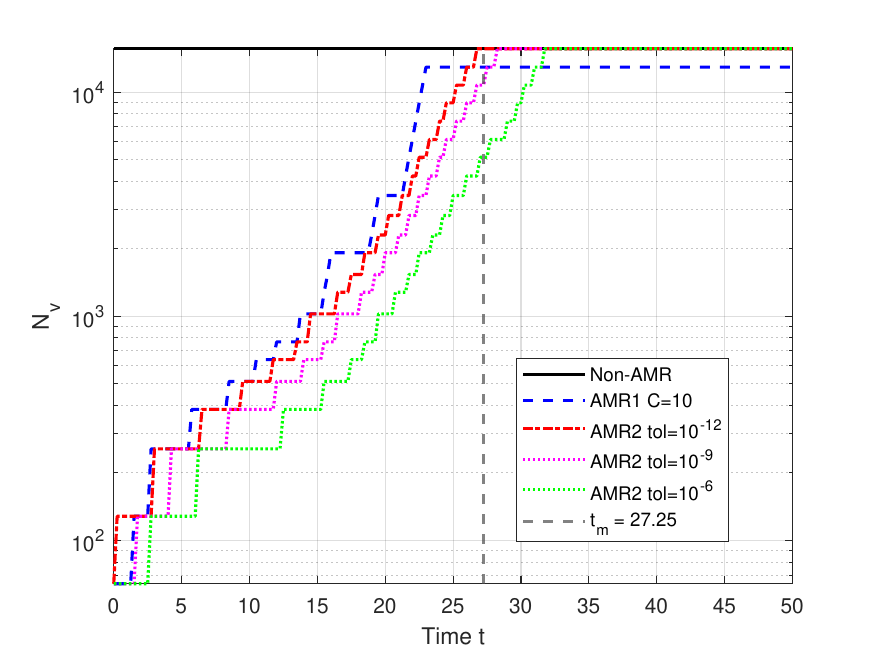}
							\subcaption{$N_v$} 
						\end{subfigure}	
						\begin{subfigure}[b]{0.4\linewidth}
							\includegraphics[width=1\linewidth]{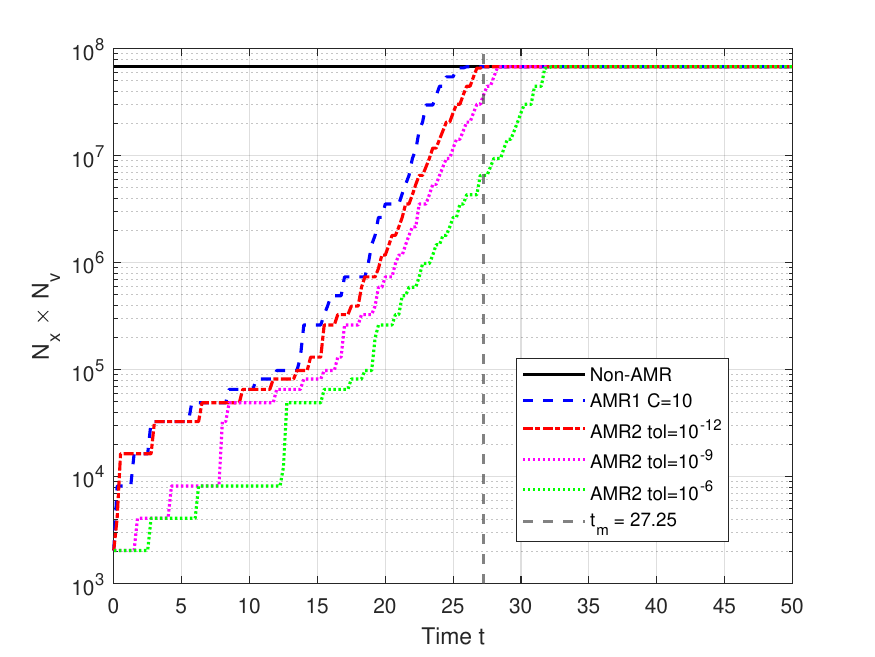}
							\subcaption{$N_x \times N_v$} 
						\end{subfigure}	
                        \begin{subfigure}[b]{0.4\linewidth}
					\includegraphics[width=1\linewidth]{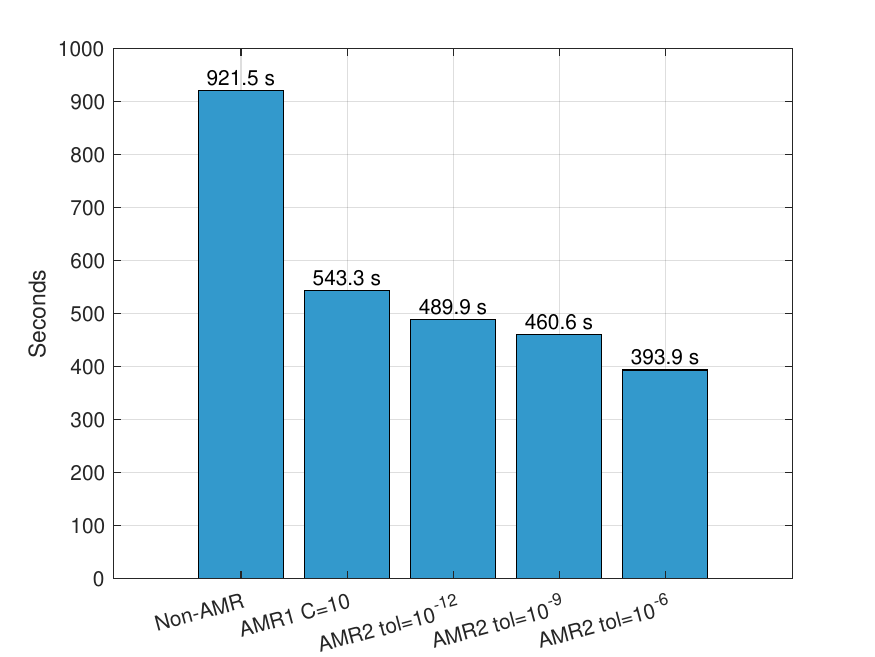}
					\subcaption{GPU time}\label{Two cpu} 
				\end{subfigure}			
						\caption{Two-stream Instability. Time evolution of grid numbers and GPU time.}\label{Two N}
					\end{figure}

			\begin{figure}[h]
				\centering
				\begin{subfigure}[b]{0.32\linewidth}
					\includegraphics[width=1\linewidth]{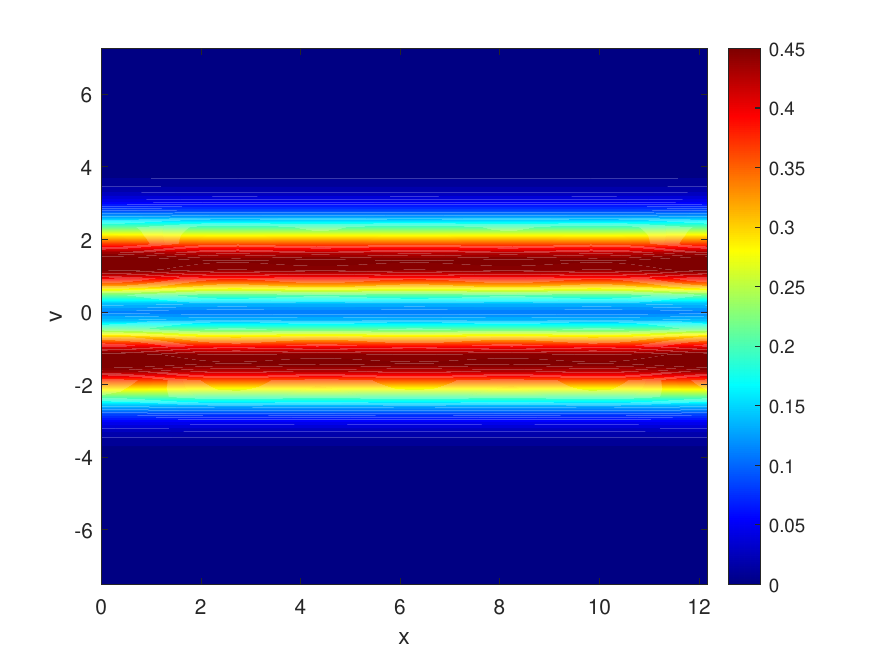}
					\subcaption{$t=0$} 
				\end{subfigure}
				\begin{subfigure}[b]{0.32\linewidth}
					\includegraphics[width=1\linewidth]{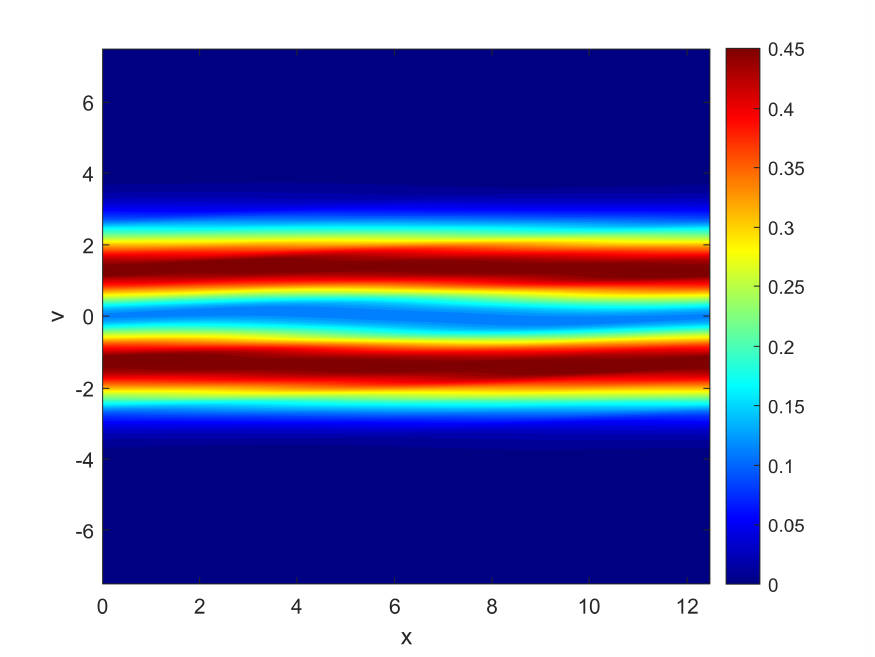}
					\subcaption{$t=10$} 
				\end{subfigure}
				\begin{subfigure}[b]{0.32\linewidth}
					\includegraphics[width=1\linewidth]{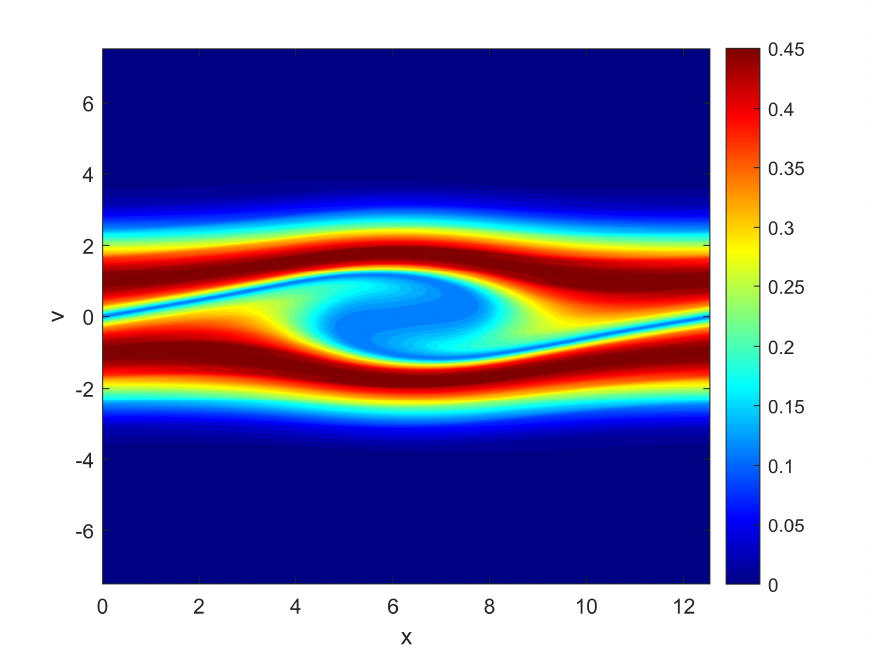}
					\subcaption{$t=20$} 
				\end{subfigure}
				\begin{subfigure}[b]{0.32\linewidth}
					\includegraphics[width=1\linewidth]{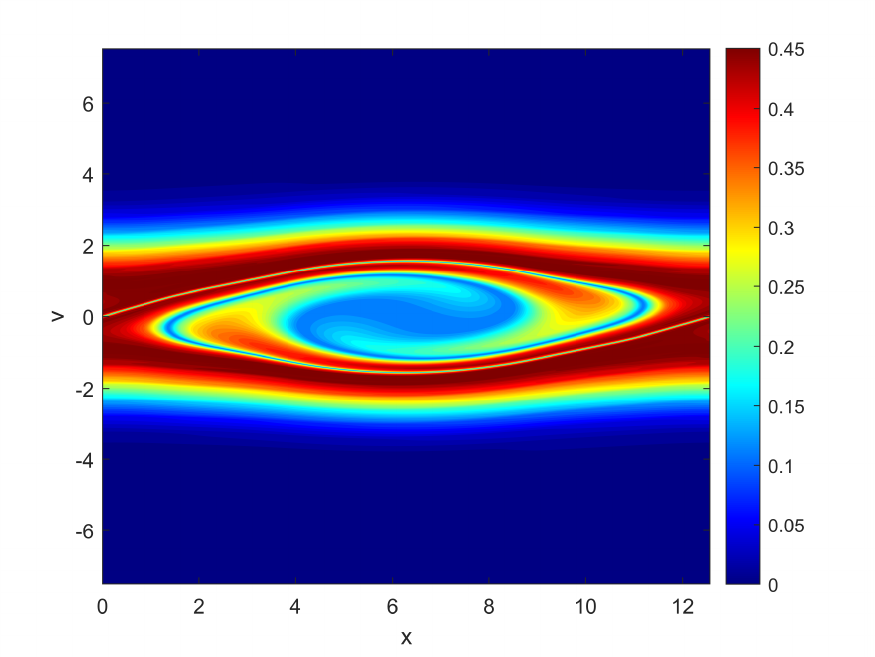}
					\subcaption{$t=30$} 
				\end{subfigure}
				\begin{subfigure}[b]{0.32\linewidth}		\includegraphics[width=1\linewidth]{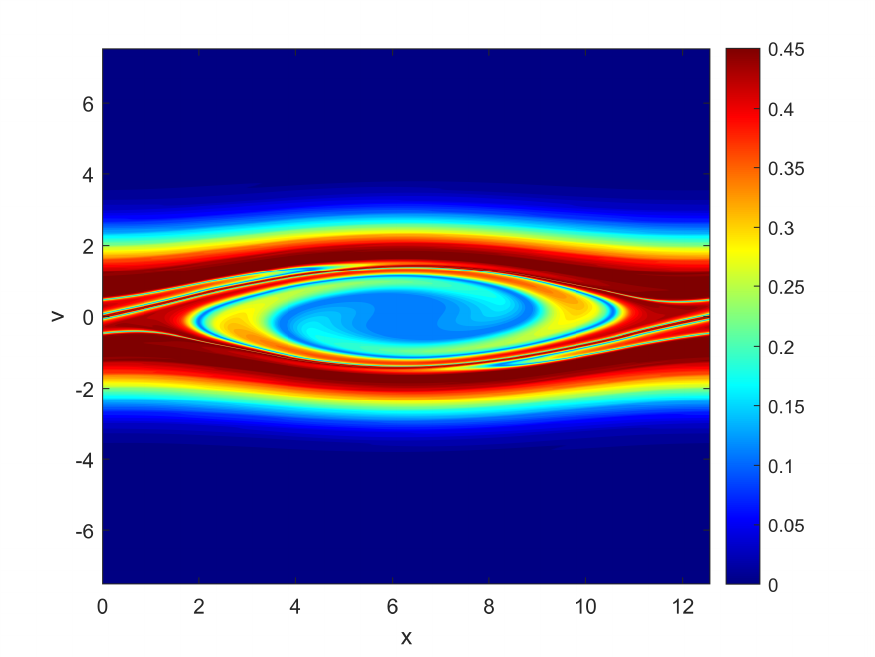}
					\subcaption{$t=40$} 
				\end{subfigure}		
				\begin{subfigure}[b]{0.32\linewidth}
					\includegraphics[width=1\linewidth]{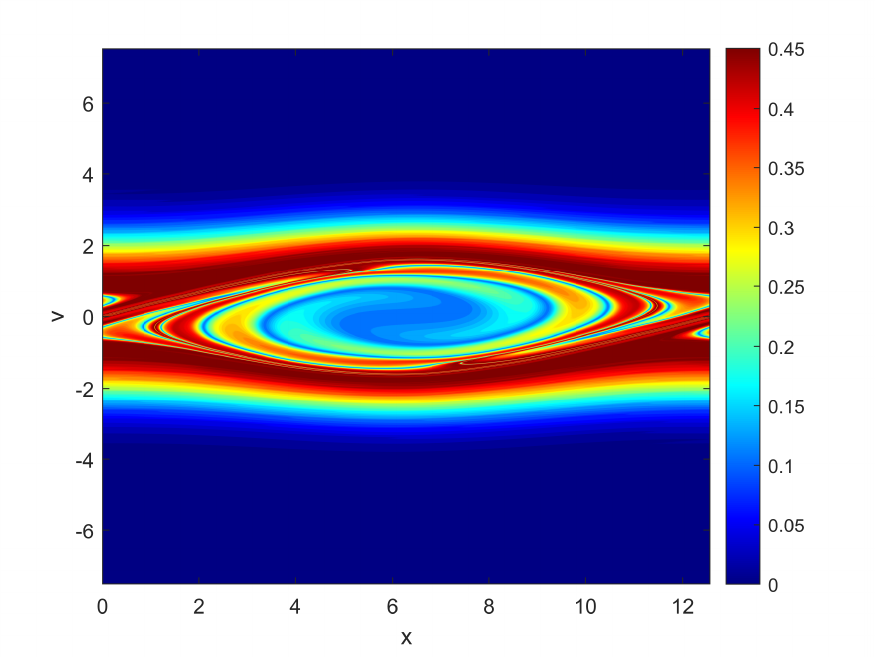}
					\subcaption{$t=50$} 
				\end{subfigure}				
				\caption{Two-stream Instability. Phase space profile of $f$.}\label{twostream phase}
			\end{figure}

                    \subsection{Symmetric Two-stream Instability}
We next consider the symmetric two-stream instability problem, where the initial distribution function is given by a perturbation of two equilibria:
\begin{equation}
	f_0(x, v) = \frac{1}{\sqrt{8\pi}v_{th}} \left( f_M(v-u; v_{th}) + f_M(v+u; v_{th}) \right) (1 + \alpha \cos(kx)),
\end{equation}
where $\alpha=0.05$, $k=\frac{2}{13}$, $v_{th}=0.3$, and $u=0.99$. To ensure spatial periodicity, the spatial domain is set to $[0,13\pi]$, and the velocity domain is truncated to $[-9,9]$.					
					
The time evolution of the $L^2$-norm of the electric field is depicted in Figure \ref{STwo a}. The system initially undergoes exponential growth driven by the linear instability. Figures \ref{STwo b}-\ref{STwo f} reveal that the macroscopic invariants, including the $L^2$-norm and total energy, are well preserved up to $t_m \approx 30$. Beyond this critical threshold, however, the grid resolution becomes insufficient to capture the increasingly fine structures of the filaments, causing most of these quantities to experience a noticeable loss of accuracy. In contrast, the discrete total mass remains invariant up to machine precision throughout the entire simulation.
					
Nevertheless, within the dynamically resolved window ($t<t_m$), Figure \ref{STwo N} shows that employing the AMR methods with an appropriate tolerance allows us to resolve the complex phase-space dynamics at only a fraction of the computational cost required by a globally refined fixed grid.
					
						\begin{figure}[h]
						\centering
						\begin{subfigure}[b]{0.4\linewidth}
							\includegraphics[width=1\linewidth]{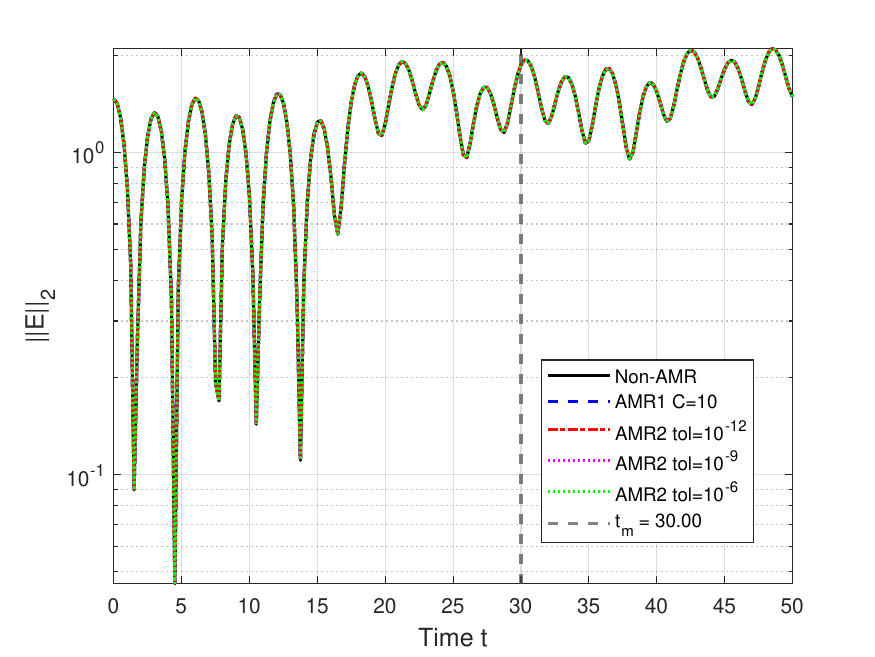}
							\subcaption{$L^2$-norm of $E$\\ \hspace{3mm}} \label{STwo a}
						\end{subfigure}	
						\begin{subfigure}[b]{0.4\linewidth}
							\includegraphics[width=1\linewidth]{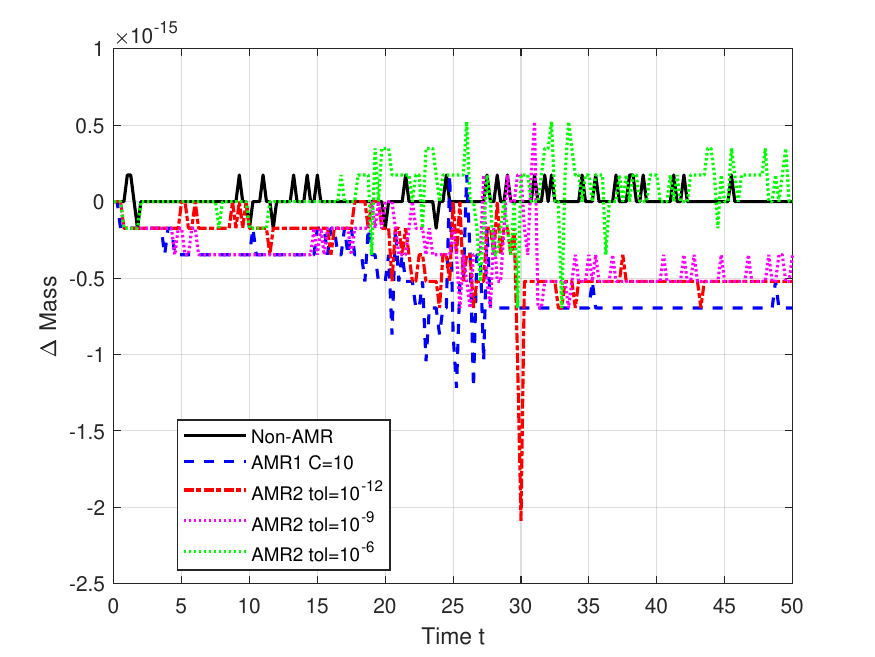}
							\subcaption{$\displaystyle \frac{\int f dvdx - \int f_0 dvdx}{\int f_0 dvdx}$}\label{STwo b}
						\end{subfigure}					
						\begin{subfigure}[b]{0.4\linewidth}
							\includegraphics[width=1\linewidth]{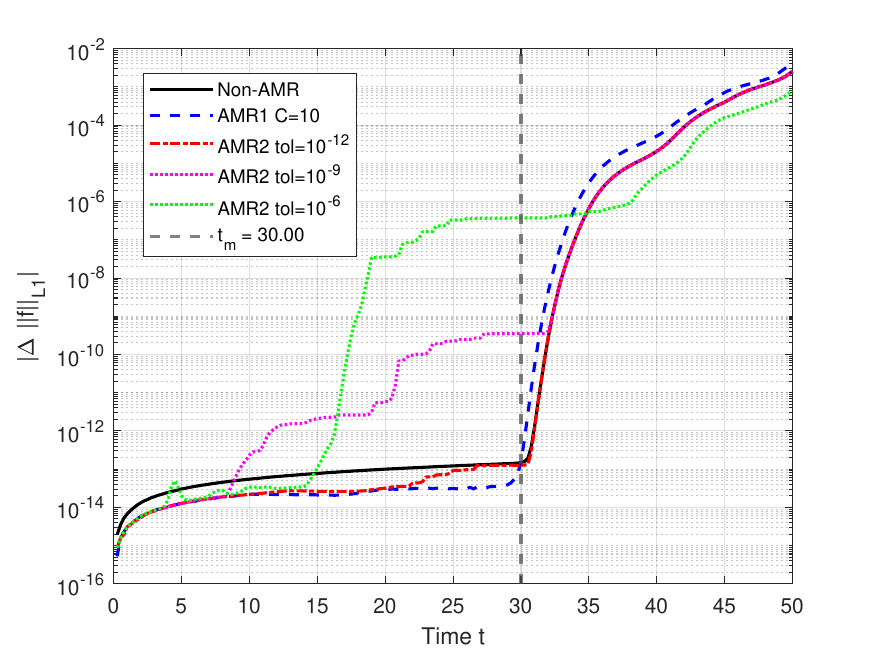}
							\subcaption{ $\displaystyle\frac{\|f^n\|_1-\|f^0\|_1}{\|f^0\|_1}$}\label{STwo c}
						\end{subfigure}	
						\begin{subfigure}[b]{0.4\linewidth}
							\includegraphics[width=1\linewidth]{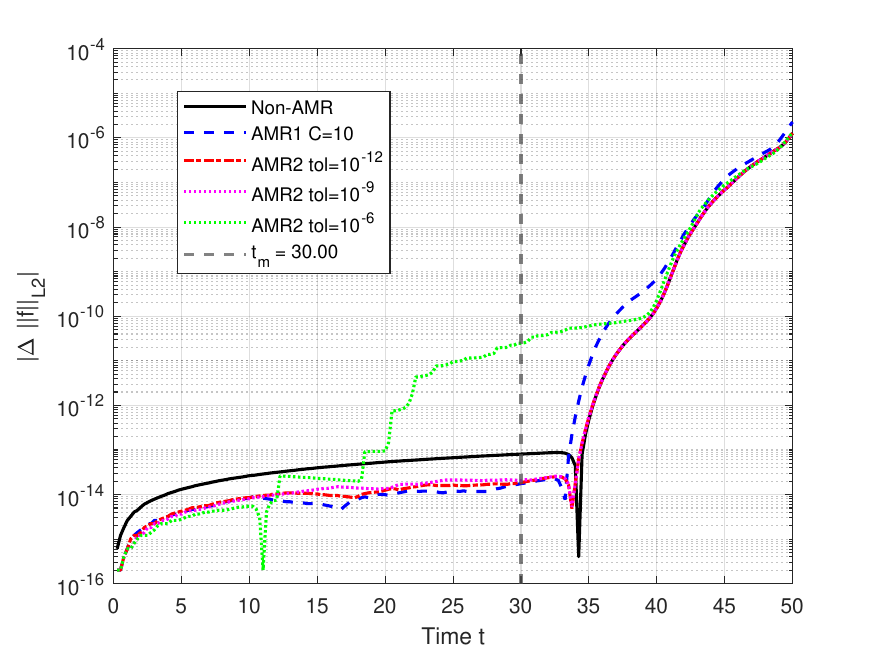}
							\subcaption{$\displaystyle\frac{\|f^n\|_2-\|f^0\|_2}{\|f^0\|_2}$}\label{STwo d}
						\end{subfigure}	
						
						\begin{subfigure}[b]{0.4\linewidth}
							\includegraphics[width=1\linewidth]{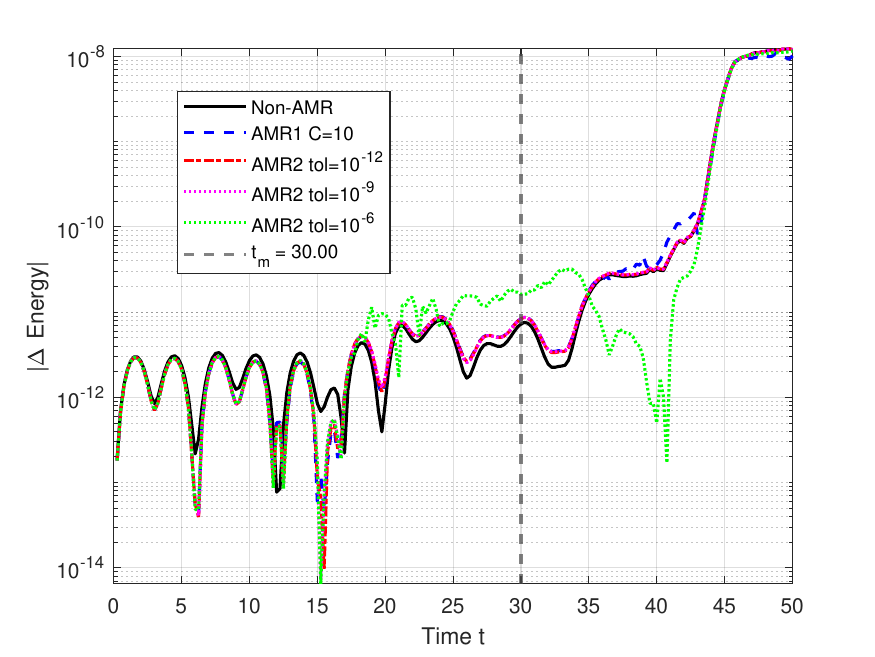}
							\subcaption{$\displaystyle \frac{Energy(t)-Energy(0)}{Energy(0)}$}\label{STwo e}
						\end{subfigure}	
						\begin{subfigure}[b]{0.4\linewidth}
							\includegraphics[width=1\linewidth]{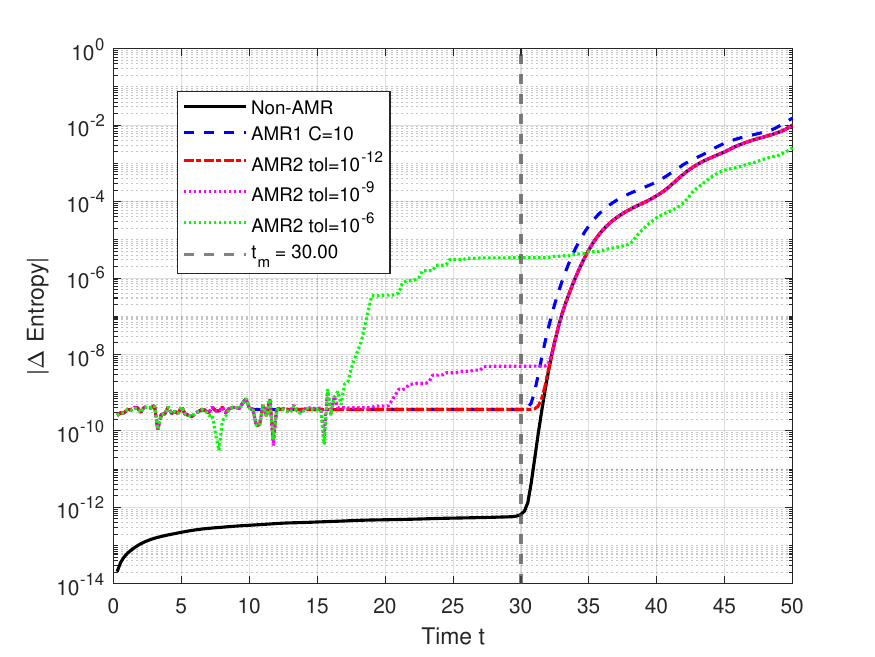}
							\subcaption{$\displaystyle \frac{Entropy(t)-Entropy(0)}{Entropy(0)}$}\label{STwo f}
						\end{subfigure}		
						\caption{Symmetric two-stream Instability. }\label{Fig symtwostream 1}
					\end{figure}
					
					\begin{figure}[h]
						\centering
						\begin{subfigure}[b]{0.4\linewidth}
							\includegraphics[width=1\linewidth]{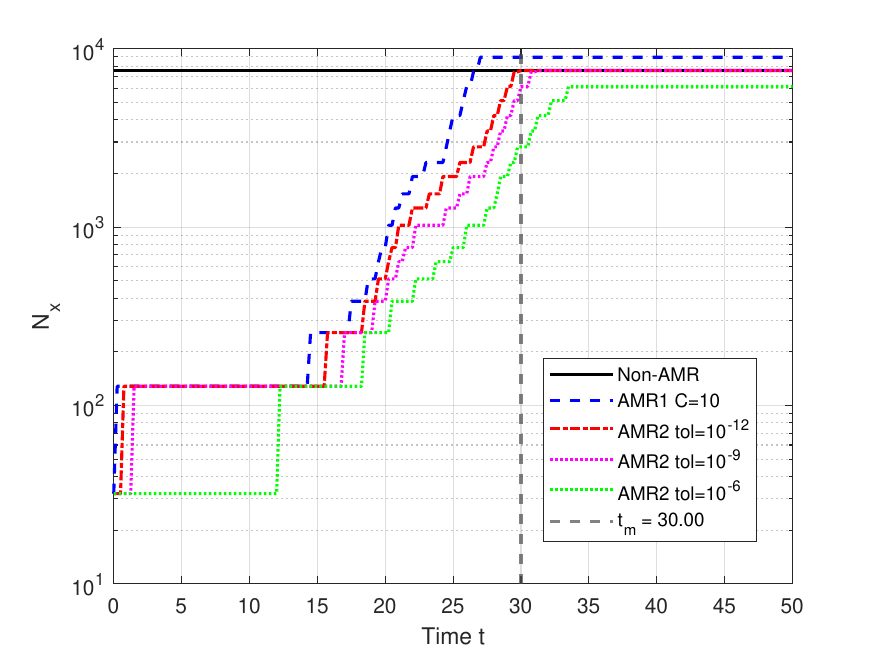}
							\subcaption{$N_x$} 
						\end{subfigure}	
						\begin{subfigure}[b]{0.4\linewidth}
							\includegraphics[width=1\linewidth]{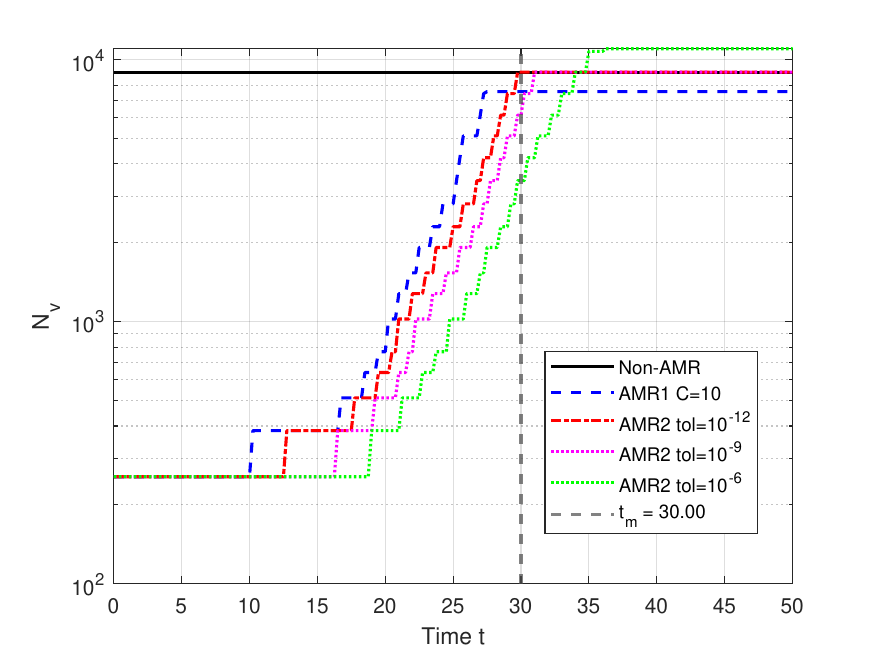}
							\subcaption{$N_v$} 
						\end{subfigure}	
						\begin{subfigure}[b]{0.4\linewidth}
							\includegraphics[width=1\linewidth]{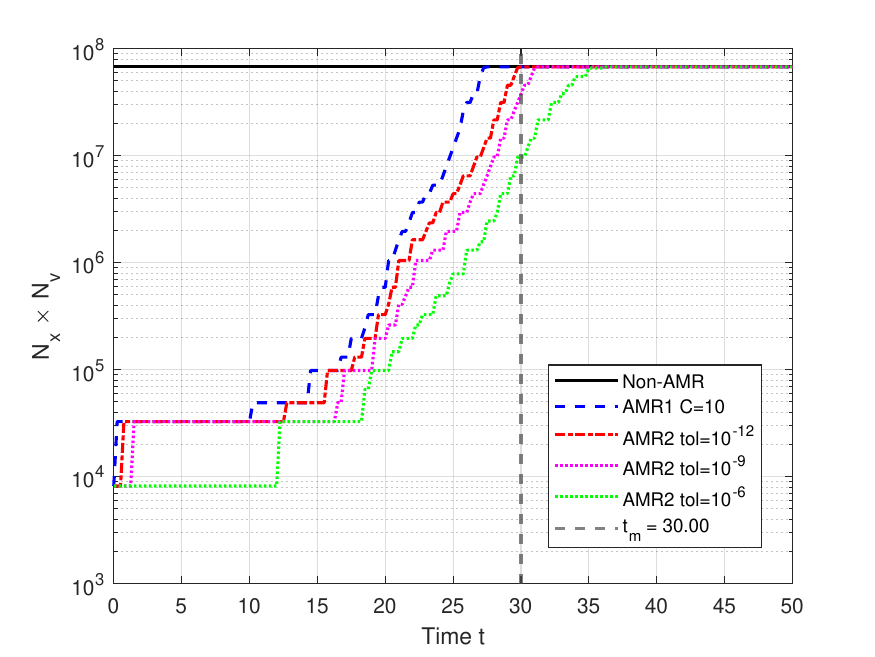}
							\subcaption{$N_x \times N_v$} 
						\end{subfigure}	
                        \begin{subfigure}[b]{0.4\linewidth}
							\includegraphics[width=1\linewidth]{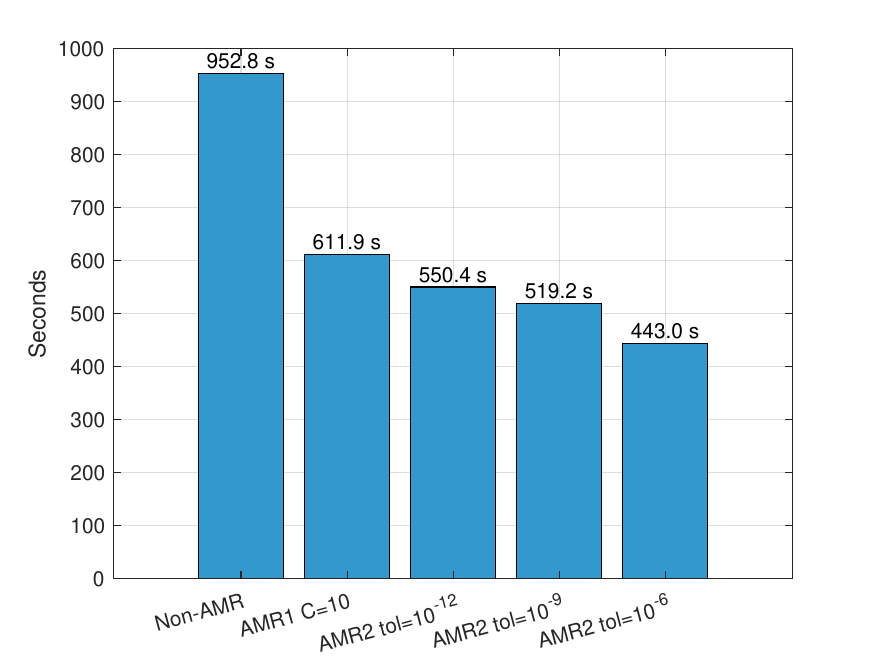}				\subcaption{GPU time} 
						\end{subfigure}			
						\caption{Symmetric two-stream Instability. Time evolution of grid numbers and GPU time.}\label{STwo N}
					\end{figure}

					Figure \ref{symtwostream phase} displays the corresponding time evolution of the phase space profiles, clearly illustrating the formation and evolution of the highly symmetric vortex structures.

						\begin{figure}[h]
							\centering
							\begin{subfigure}[b]{0.32\linewidth}
								\includegraphics[width=1\linewidth]{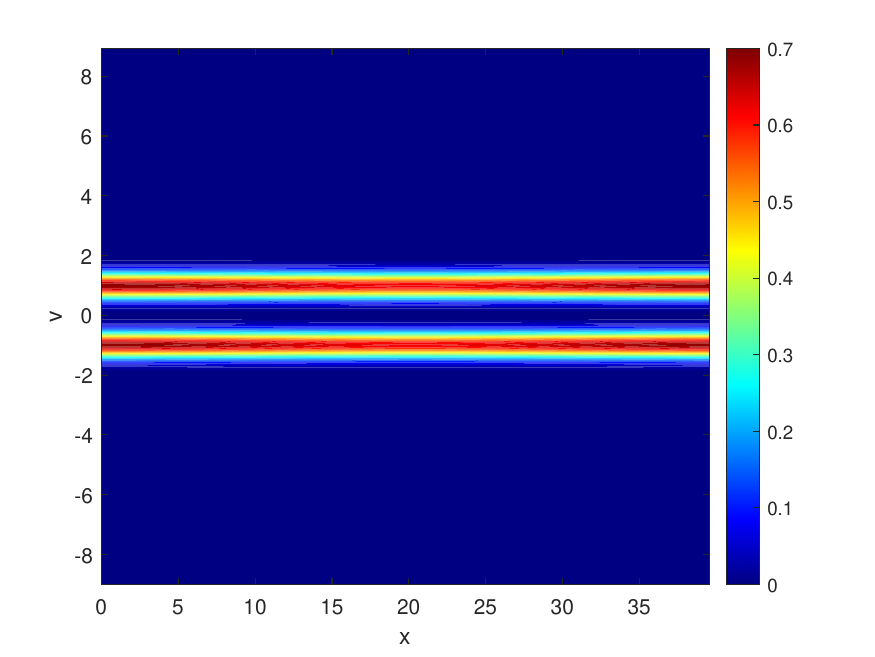}
								\subcaption{$t=0$} 
							\end{subfigure}
							\begin{subfigure}[b]{0.32\linewidth}
								\includegraphics[width=1\linewidth]{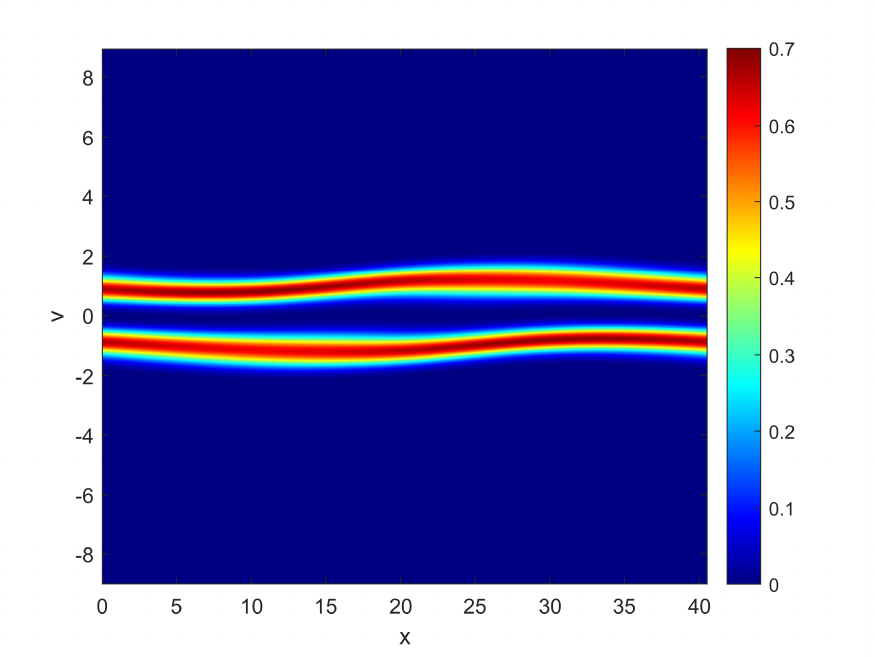}
								\subcaption{$t=10$} 
							\end{subfigure}
							\begin{subfigure}[b]{0.32\linewidth}
								\includegraphics[width=1\linewidth]{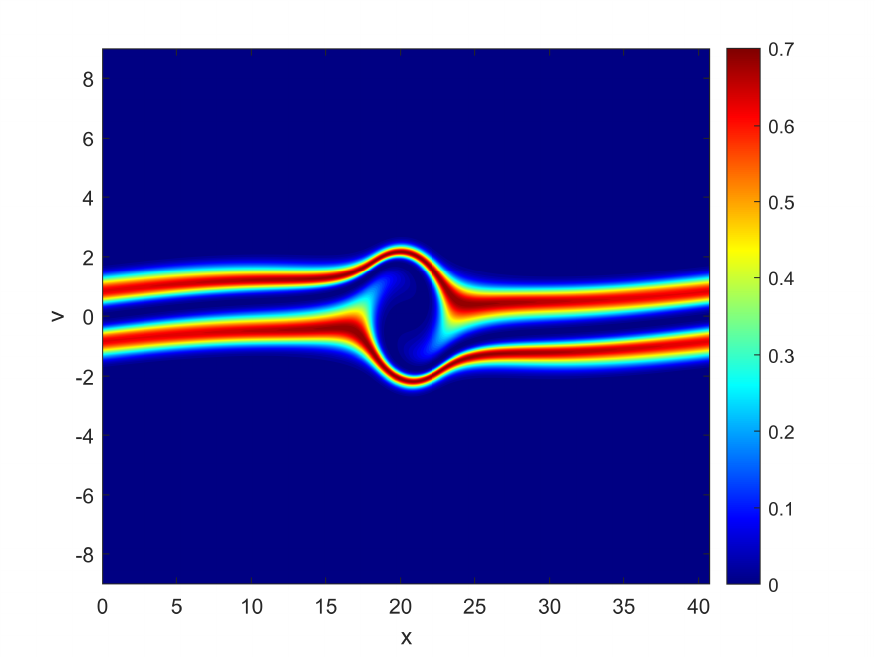}
								\subcaption{$t=20$} 
							\end{subfigure}
							\begin{subfigure}[b]{0.32\linewidth}
								\includegraphics[width=1\linewidth]{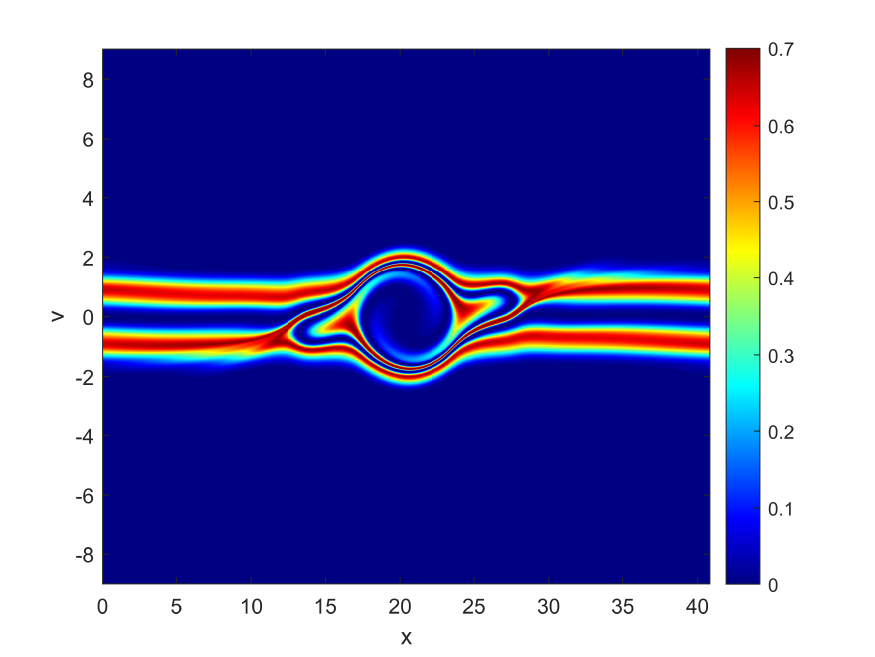}
								\subcaption{$t=30$} 
							\end{subfigure}
							\begin{subfigure}[b]{0.32\linewidth}		\includegraphics[width=1\linewidth]{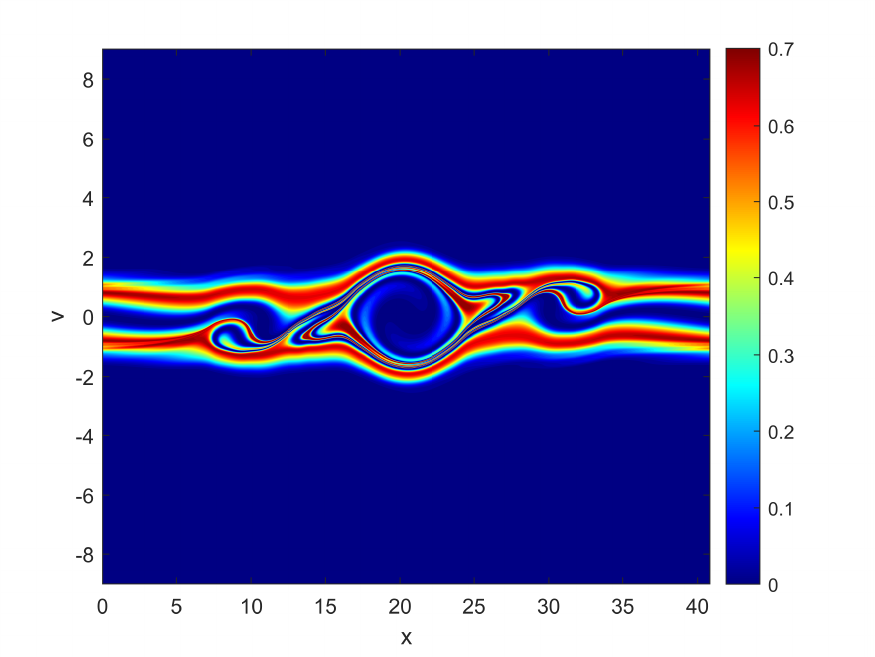}
								\subcaption{$t=40$} 
							\end{subfigure}		
							\begin{subfigure}[b]{0.32\linewidth}
								\includegraphics[width=1\linewidth]{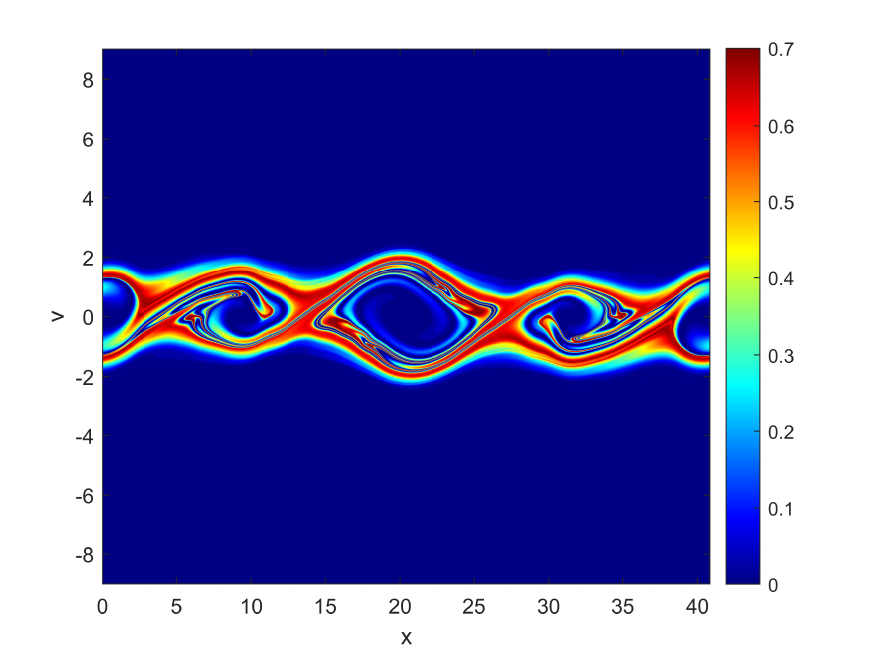}
								\subcaption{$t=50$} 
							\end{subfigure}				
							\caption{Symmetric two-stream Instability. Phase space profile of $f$.}\label{symtwostream phase}
						\end{figure}

						\subsection{Bump-on-tail Instability}				
                        The last example is the unstable bump-on-tail problem. The initial distribution function is given by 
\begin{equation*}
	f_0(x, v) = \left( 0.9 f_M(v; 1) + 0.2 f_M(v-u; v_{th}) \right) (1 + \alpha \cos(kx)),
\end{equation*}
where $\alpha=0.04$, $k=0.3$, $v_{th}=0.5$, and $u=4.5$. The spatial domain is $[0,20\pi/3]$, and we truncate the velocity domain to $[-13,13]$. This test is particularly challenging for standard spectral methods because it requires a significantly large velocity domain to encompass the fast-moving beam. In a fixed-grid approach, this results in a massive waste of grid points in empty regions where the distribution function is nearly zero.                        

						In Figures \ref{Fig bump 1}-\ref{bot cpu}, it is observable that the proposed adaptive method shows good performance in the preservation of conservative variables. Consistent with the prior instability tests, the simulations reach the hardware memory limit in a finite time around $t_m \approx 25$. Beyond this critical point, restricted grid expansion leads to the expected gradual decay in the $L^2$-norm. Once again, the discrete total mass remains strictly invariant up to machine precision throughout the entire simulation, unaffected by grid resolution. The numerical results in Figures \ref{bot N} support the efficiency of the proposed method. Finally, the corresponding phase space profiles can be found in Figure \ref{bump phase}.

							\begin{figure}[h]
							\centering
							\begin{subfigure}[b]{0.4\linewidth}
								\includegraphics[width=1\linewidth]{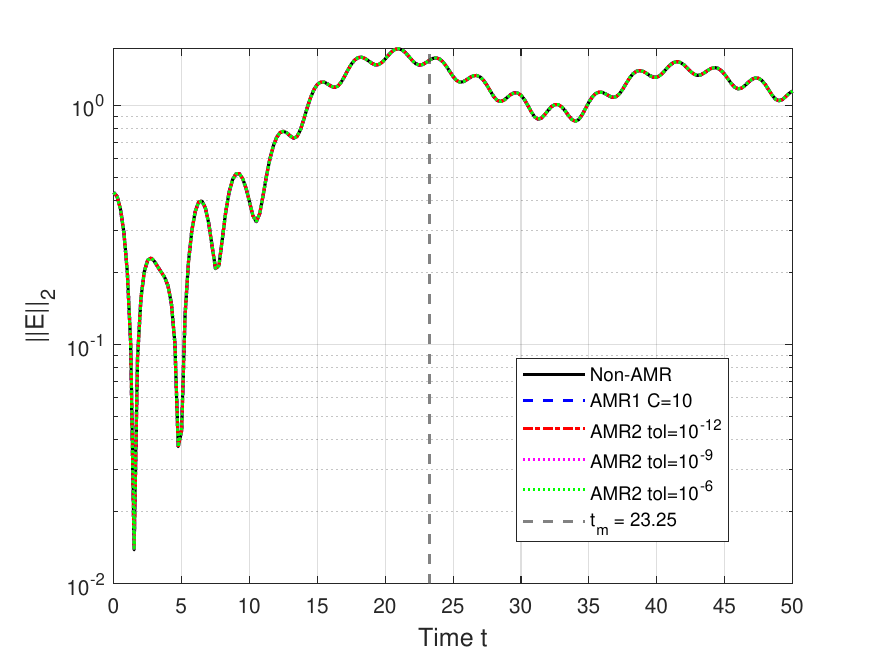}
								\subcaption{$L^2$-norm of $E$\\ \hspace{3mm}} \label{bot a}
							\end{subfigure}	
						\begin{subfigure}[b]{0.4\linewidth}
							\includegraphics[width=1\linewidth]{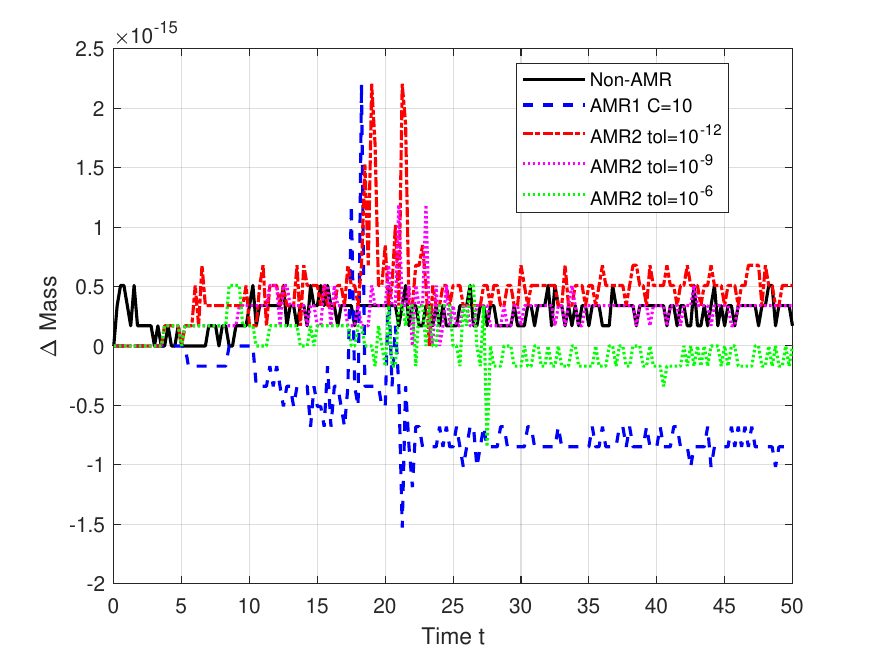}
							\subcaption{$\displaystyle \frac{\int f dvdx - \int f_0 dvdx}{\int f_0 dvdx}$}\label{bot b}
						\end{subfigure}	
						\begin{subfigure}[b]{0.4\linewidth}
						\includegraphics[width=1\linewidth]{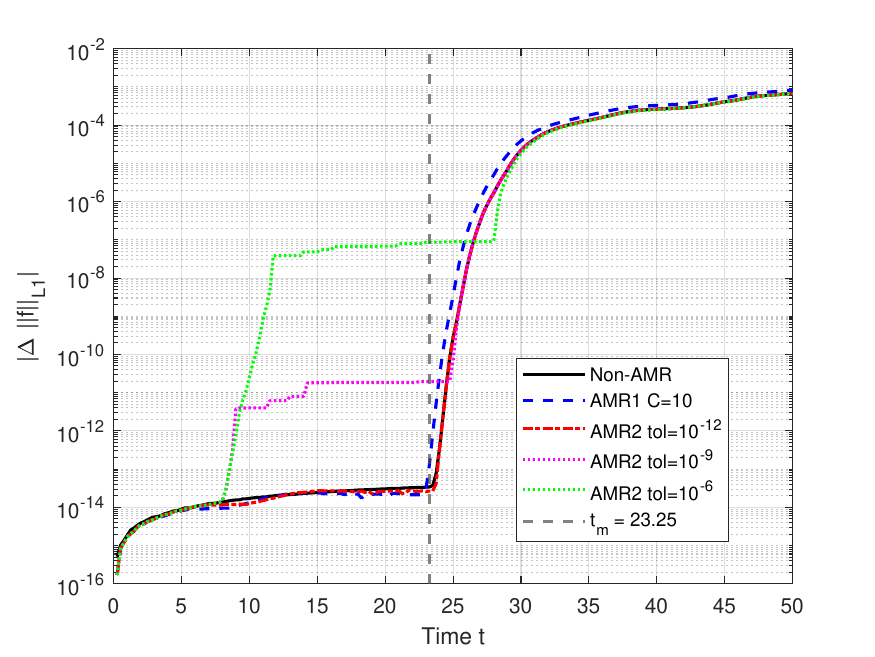}
						\subcaption{ $\displaystyle\frac{\|f^n\|_1-\|f^0\|_1}{\|f^0\|_1}$}\label{bot c}
					\end{subfigure}	
							\begin{subfigure}[b]{0.4\linewidth}
								\includegraphics[width=1\linewidth]{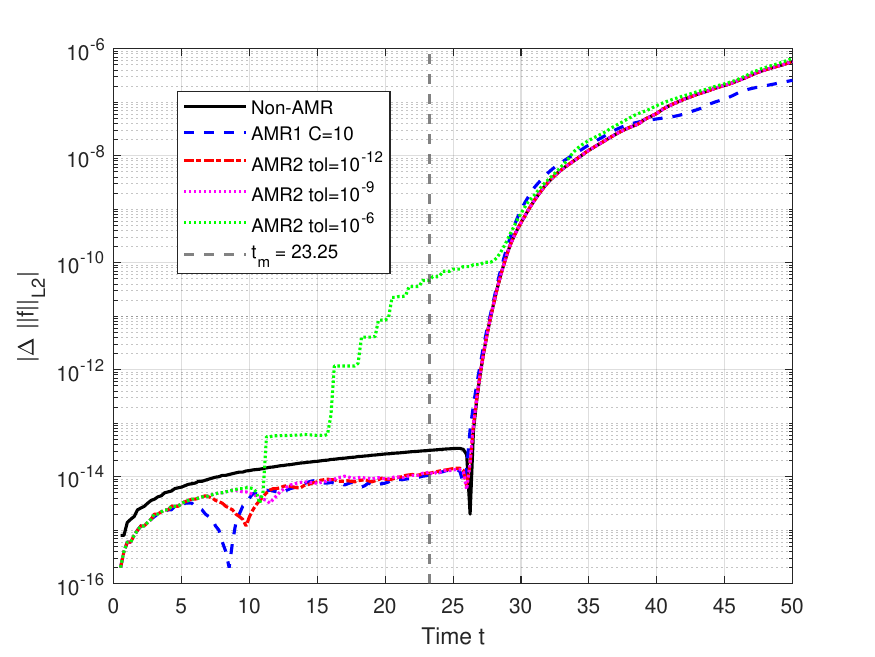}
								\subcaption{$\displaystyle\frac{\|f^n\|_2-\|f^0\|_2}{\|f^0\|_2}$}\label{bot d}
							\end{subfigure}	
							\begin{subfigure}[b]{0.4\linewidth}
							\includegraphics[width=1\linewidth]{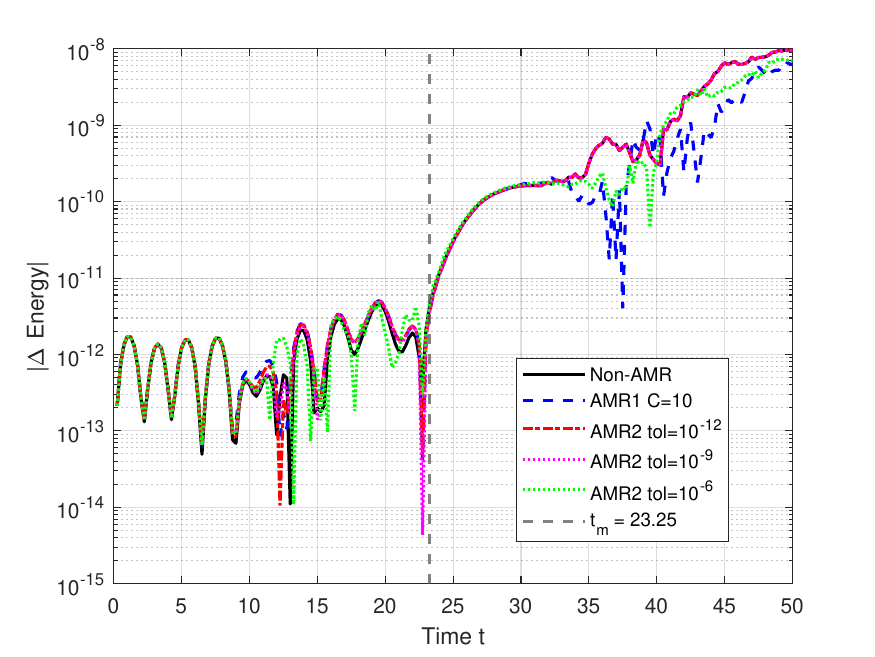}
							\subcaption{$\displaystyle \frac{Energy(t)-Energy(0)}{Energy(0)}$}\label{bot e}
						\end{subfigure}	
							\begin{subfigure}[b]{0.4\linewidth}
								\includegraphics[width=1\linewidth]{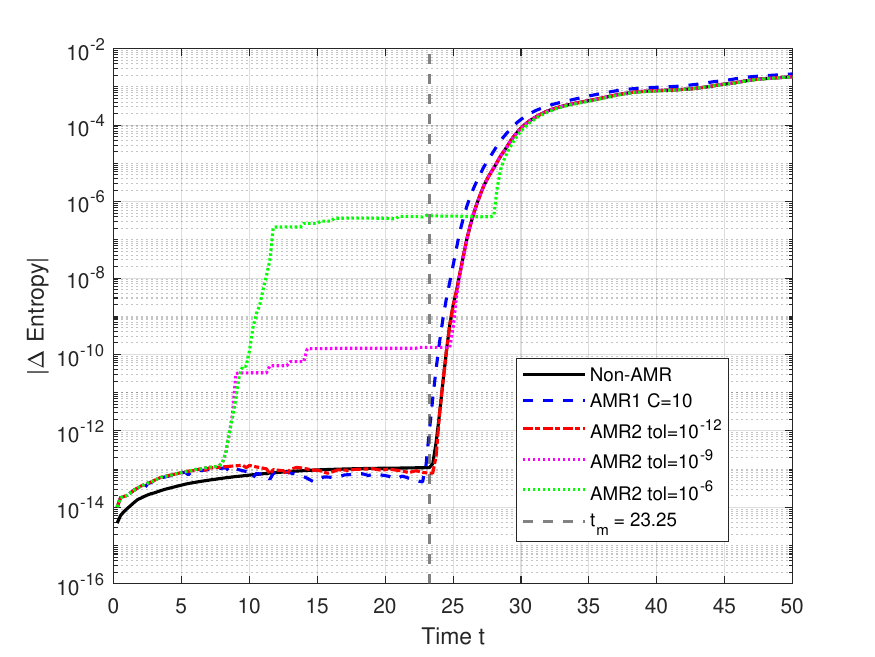}
								\subcaption{$\displaystyle \frac{Entropy(t)-Entropy(0)}{Entropy(0)}$}\label{bot f}
							\end{subfigure}			
							\caption{Bump-on-tail instability.}\label{Fig bump 1}
						\end{figure}
						
						\begin{figure}[h]
							\centering
							\begin{subfigure}[b]{0.4\linewidth}
								\includegraphics[width=1\linewidth]{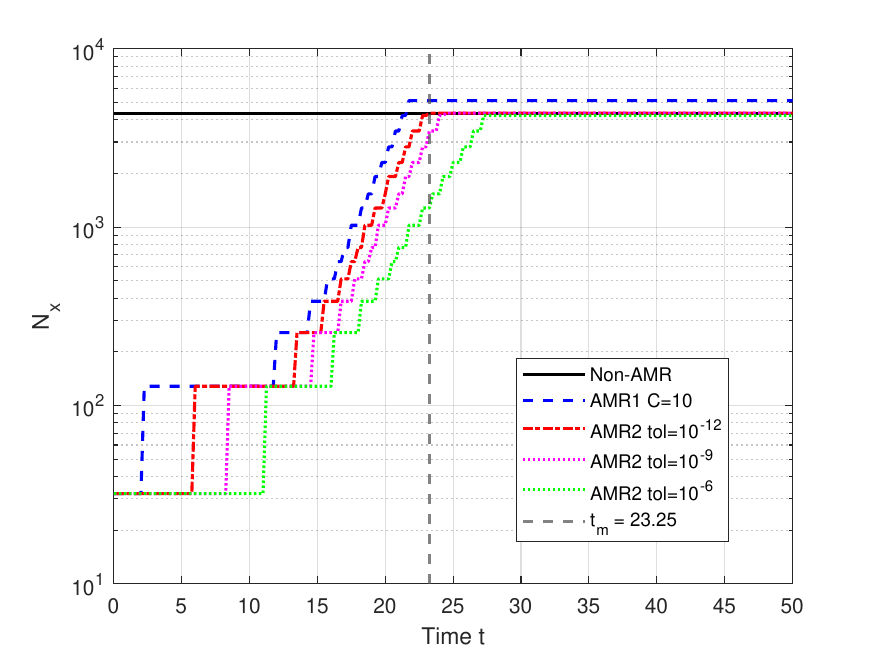}
								\subcaption{$N_x$} 
							\end{subfigure}	
							\begin{subfigure}[b]{0.4\linewidth}
								\includegraphics[width=1\linewidth]{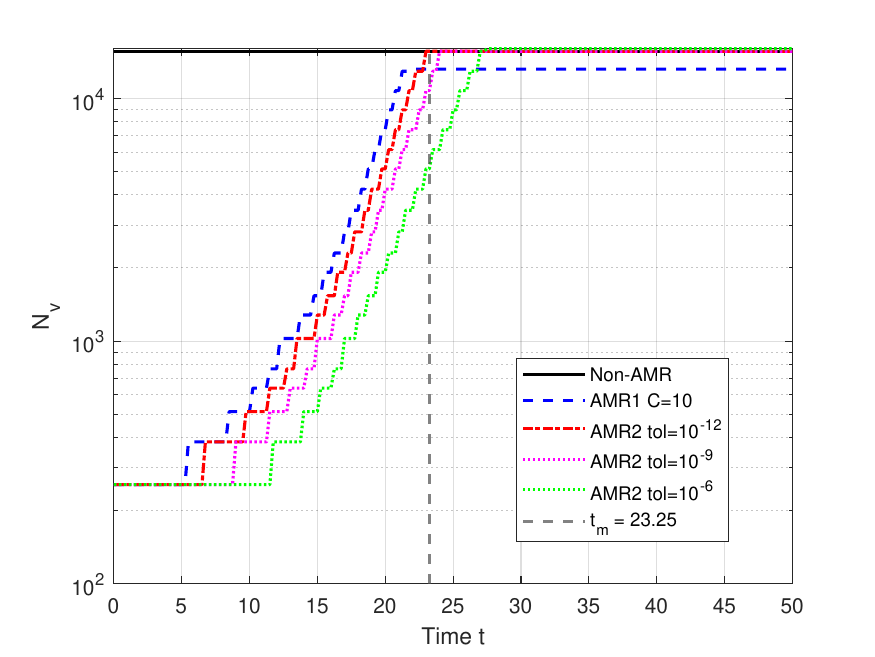}
								\subcaption{$N_v$} 
							\end{subfigure}	
							\begin{subfigure}[b]{0.4\linewidth}
								\includegraphics[width=1\linewidth]{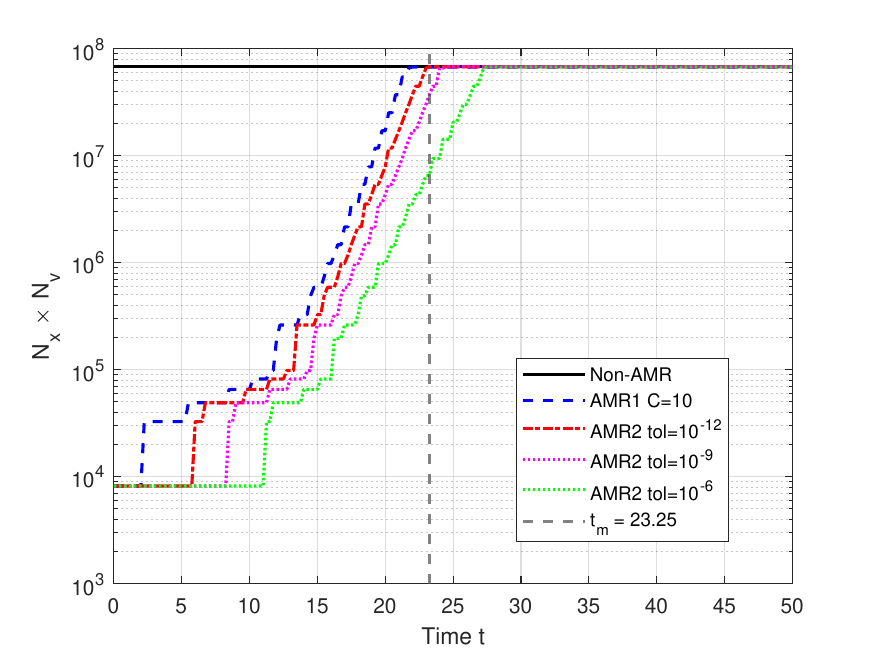}
								\subcaption{$N_x \times N_v$} 
							\end{subfigure}
                            \begin{subfigure}[b]{0.4\linewidth}
								\includegraphics[width=1\linewidth]{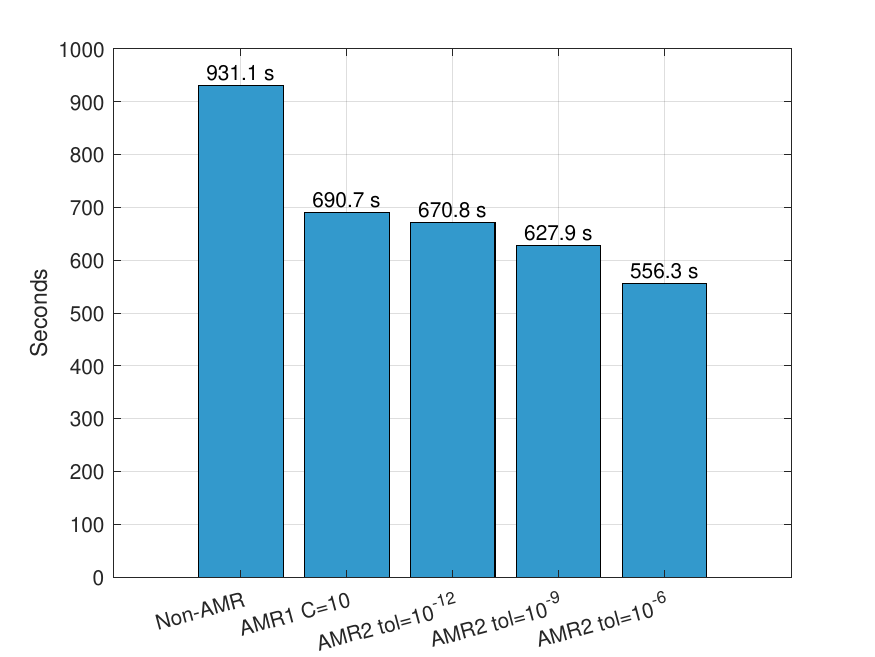}
								\subcaption{GPU time}\label{bot cpu}
							\end{subfigure}			
							\caption{Bump-on-tail instability. Time evolution of grid numbers and GPU time.}\label{bot N}
						\end{figure}

					\begin{figure}[h]
						\centering
						\begin{subfigure}[b]{0.32\linewidth}
							\includegraphics[width=1\linewidth]{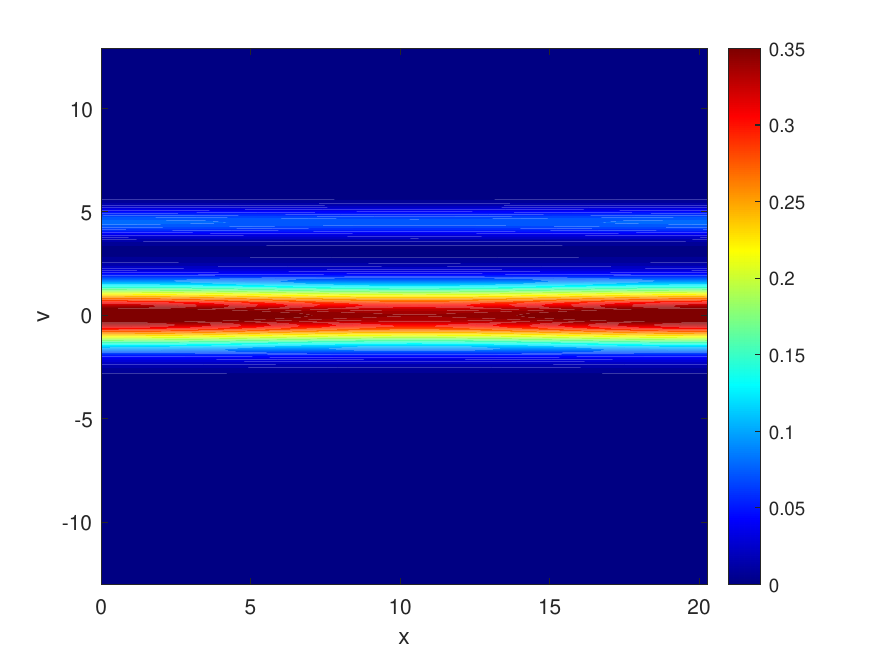}
							\subcaption{$t=0$} 
						\end{subfigure}
						\begin{subfigure}[b]{0.32\linewidth}
							\includegraphics[width=1\linewidth]{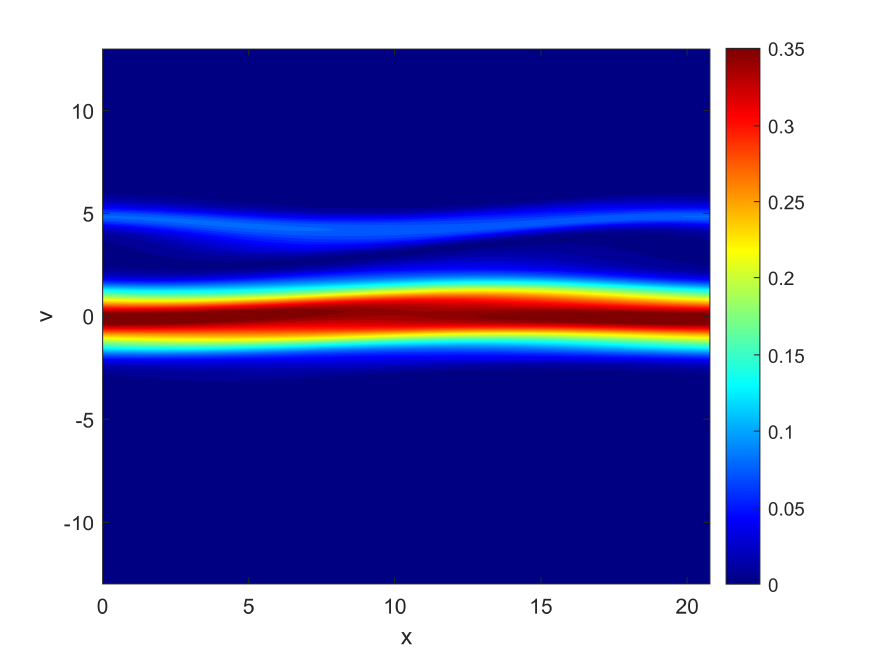}
							\subcaption{$t=10$} 
						\end{subfigure}
						\begin{subfigure}[b]{0.32\linewidth}
							\includegraphics[width=1\linewidth]{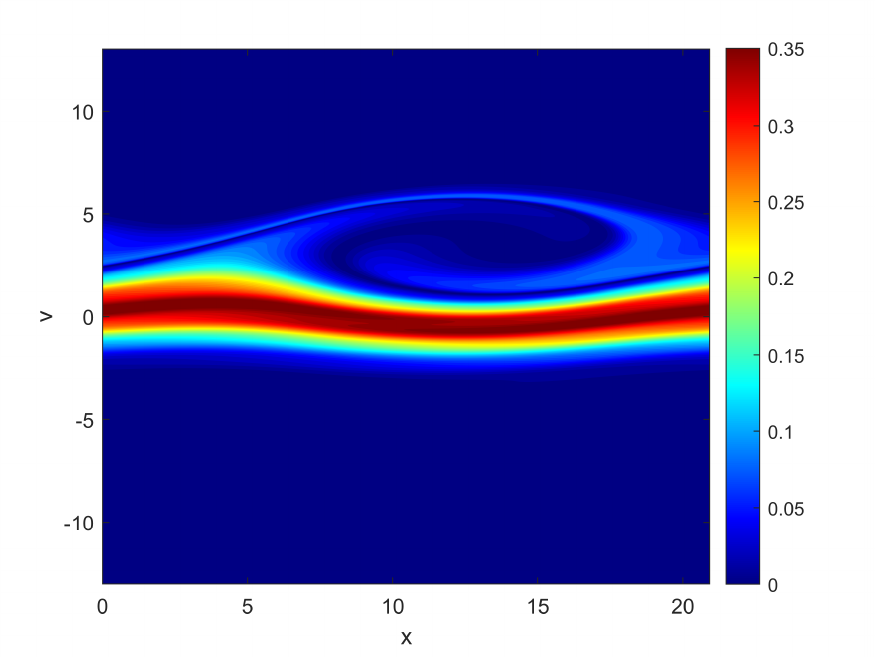}
							\subcaption{$t=20$} 
						\end{subfigure}
						\begin{subfigure}[b]{0.32\linewidth}
							\includegraphics[width=1\linewidth]{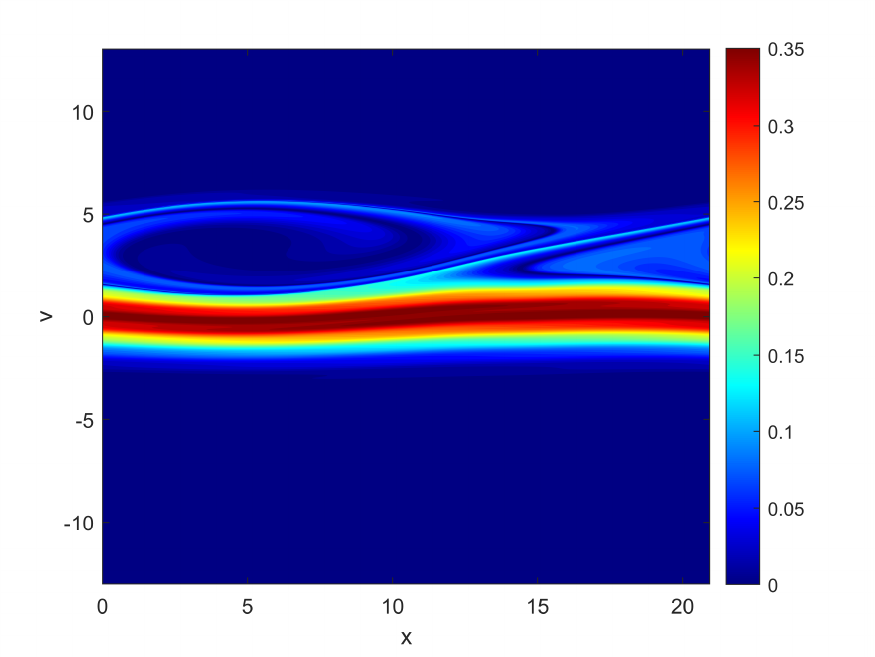}
							\subcaption{$t=30$} 
						\end{subfigure}
						\begin{subfigure}[b]{0.32\linewidth}	\includegraphics[width=1\linewidth]{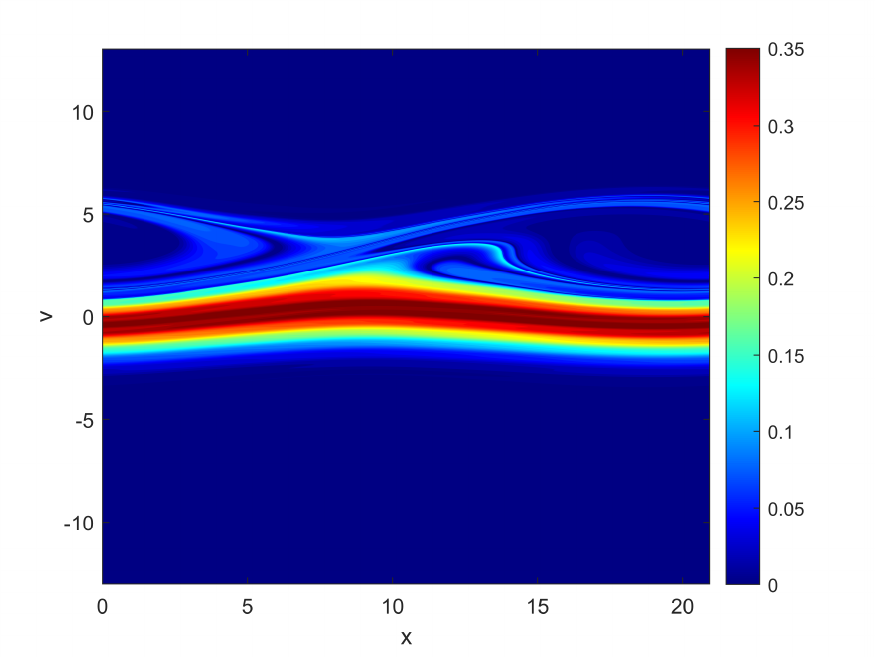}
							\subcaption{$t=40$} 
						\end{subfigure}		
						\begin{subfigure}[b]{0.32\linewidth}
							\includegraphics[width=1\linewidth]{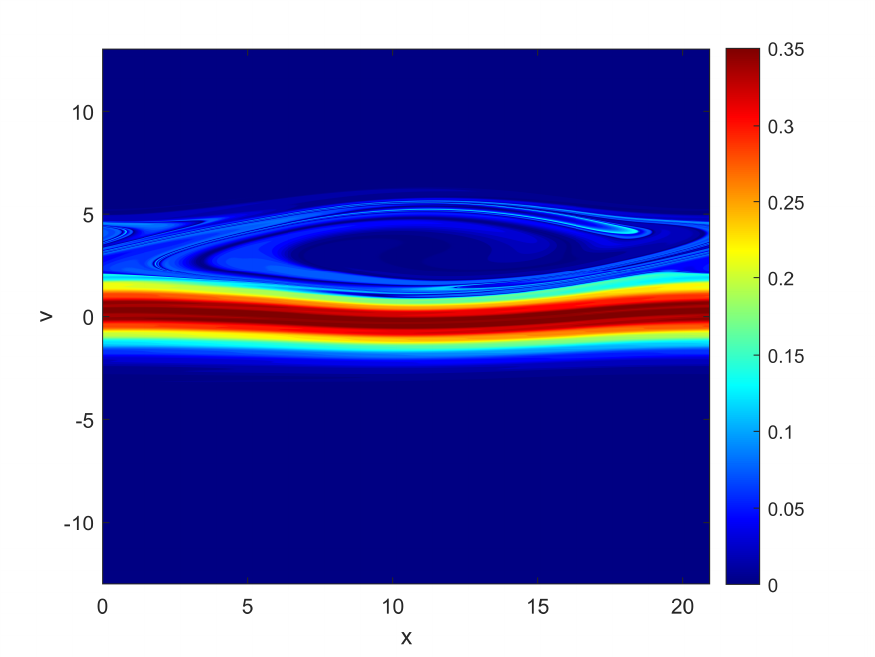}
							\subcaption{$t=50$} 
						\end{subfigure}				
						\caption{Bump-on-tail Instability. Phase space profile of $f$.}\label{bump phase}
					\end{figure}

			\section{Conclusion}
		In this paper, we proposed and analyzed an adaptive Fourier spectral method coupled with a semi-Lagrangian splitting scheme for solving the Vlasov–Poisson system. The core of our approach is an adaptive mesh refinement (AMR) technique based on dynamic zero-padding in the spectral domain, which allocates computational resources exclusively to the actively developing phase-space filamentations. By effectively avoiding the massive waste of grid points typical in standard fixed-grid methods, this approach significantly enhances overall computational efficiency. Furthermore, we mathematically established that the discrete total mass and the $L^2$-norm are exactly conserved during both the adaptive grid expansion and the phase-shift-based advection steps. Several numerical experiments, including the strong Landau damping and the bump-on-tail instability, verified these theoretical properties and demonstrated the extreme fidelity of the scheme. In these simulations, the severe filamentation inherent to the system inevitably pushes the computations to a hardware memory limit at a critical time $t_m$, preventing any further grid expansion. During the fully resolved regime ($t \le t_m$), the proposed scheme preserves all fundamental macroscopic invariants—namely, discrete total mass, $L^1$ and $L^2$ norms, total energy, and entropy—to machine precision. Once this resolution limit is exceeded ($t > t_m$), most of these physical quantities inevitably begin to deviate from their initial states due to the unresolved high-frequency modes. Remarkably, however, the discrete total mass remains a strict exception; it is mathematically conserved to machine precision throughout the entire simulation, completely unaffected by the extreme grid resolution constraints.
			

\section*{Acknowledgment}
This work has been partially supported by the project PIA.CE.RI 2024/2026 ``Control of complex systems from theory to applications”, (CoCoS), from University of Catania.  G. Russo would like to thank the Italian Ministry of University and Research (MUR) for the support of this research with funds coming from PRIN Project 2022 (N.\ 2022KA3JBA entitled  ``Advanced numerical methods for time dependent parametric partial differential equations with applications''). G. Russo is a member of the INdAM Research group GNCS. S. Y. Cho was supported by the National Research Foundation of Korea(NRF) grant funded by the Korea government(MSIT) (RS-2026-25493252).

\bibliographystyle{amsplain}

\begin{thebibliography}{10}


\bibitem{A} Adams, M. F. (2026). Semi-Lagrangian Discontinuous Galerkin Method with Adaptive Mesh Refinement for the Vlasov--Poisson System in 1X+ 3V. arXiv preprint arXiv:2603.19959.
	
\bibitem{AV} Arber, T. D., and Vann, R. G. L. (2002). A critical comparison of Eulerian-grid-based Vlasov solvers. Journal of computational physics, 180(1), 339-357.



\bibitem{BH} Banks, J. W.,  Hittinger, J. A. F. (2010). A new class of nonlinear finite-volume methods for Vlasov simulation. IEEE Transactions on Plasma Science, 38(9), 2198-2207.


\bibitem{BM} Besse, N., Mehrenberger, M. (2008). Convergence of classes of high-order semi-Lagrangian schemes for the Vlasov–Poisson system. Mathematics of computation, 77(261), 93-123.

\bibitem{BL} Birdsall, C. K., Langdon, A. B. (2018). Plasma physics via computer simulation. CRC press.


\bibitem{CGQ} Cai, X., Guo, W.,  Qiu, J. M. (2018). A high order semi-Lagrangian discontinuous Galerkin method for Vlasov–Poisson simulations without operator splitting. Journal of Computational Physics, 354, 529-551.


\bibitem{CCFM} Casas, F., Crouseilles, N., Faou, E.,  Mehrenberger, M. (2017). High-order Hamiltonian splitting for the Vlasov–Poisson equations. Numerische Mathematik, 135(3), 769-801.

\bibitem{CK} Cheng, C. Z., Knorr, G. (1976). The integration of the Vlasov equation in configuration space. Journal of Computational Physics, 22(3), 330-351.

\bibitem{CBRY} Cho, S. Y., Boscarino, S., Russo, G., Yun, S. B. (2021). Conservative semi-Lagrangian schemes for kinetic equations Part II: Applications. Journal of Computational Physics, 436, 110281.

\bibitem{CL} Crouseilles, N.,  Lemou, M. (2011). An asymptotic preserving scheme based on a micro-macro decomposition for collisional Vlasov equations: diffusion and high-field scaling limits. Kinetic and related models, 4(2), 441-477.

\bibitem{CRS} Crouseilles, N., Respaud, T.,  Sonnendr{\"u}cker, E. (2009). A forward semi-Lagrangian method for the numerical solution of the Vlasov equation. Computer Physics Communications, 180(10), 1730-1745.



\bibitem{DDSV} Degond, P., Deluzet, F., Navoret, L., Sun, A. B., Vignal, M. H. (2010). Asymptotic-preserving particle-in-cell method for the Vlasov–Poisson system near quasineutrality. Journal of Computational Physics, 229(16), 5630-5652.







\bibitem{EL} Einkemmer, L., Lubich, C. (2018). A low-rank projector-splitting integrator for the Vlasov--Poisson equation. SIAM Journal on Scientific Computing, 40(5), B1330-B1360.


\bibitem{FMP} Filbet, F., Mouhot, C.,  Pareschi, L. (2006). Solving the Boltzmann equation in N log2 N. SIAM Journal on Scientific Computing, 28(3), 1029-1053.

\bibitem{FS} Filbet, F., Sonnendrücker, E. (2003). Comparison of Eulerian vlasov solvers. Computer Physics Communications, 150(3), 247-266.


\bibitem{FSB} Filbet, F., Sonnendrücker, E., Bertrand, P. (2001). Conservative numerical schemes for the Vlasov equation. Journal of Computational Physics, 172(1), 166-187.

\bibitem{FR} Filbet, F.,  Russo, G. (2003). High order numerical methods for the space non-homogeneous Boltzmann equation. Journal of Computational Physics, 186(2), 457-480.

\bibitem{FX} Filbet, F.,  Xiong, T. (2022). Conservative discontinuous Galerkin/Hermite spectral method for the Vlasov–Poisson system. Communications on Applied Mathematics and Computation, 4(1), 34-59.

\bibitem{GV} Ganguly, K., Victory, Jr, H. D. (1989). On the convergence of particle methods for multidimensional Vlasov–Poisson systems. SIAM journal on numerical analysis, 26(2), 249-288.


\bibitem{GQ} Guo, W., Qiu, J. M. (2024). A conservative low rank tensor method for the Vlasov dynamics. SIAM Journal on Scientific Computing, 46(1), A232-A263.

\bibitem{HE}
Hockney, R. W., Eastwood, J. W., Computer simulation using particles. Taylor \& Francis, 1988.



\bibitem{KF} Klimas, A. J., Farrell, W. M. (1994). A splitting algorithm for Vlasov simulation with filamentation filtration. Journal of Computational physics, 110(1), 150-163.


\bibitem{K} Kormann, K. (2015). A semi-Lagrangian Vlasov solver in tensor train format. SIAM Journal on Scientific Computing, 37(4), B613-B632.


\bibitem{MP} Mouhot, C.,  Pareschi, L. (2006). Fast algorithms for computing the Boltzmann collision operator. Mathematics of computation, 75(256), 1833-1852.


\bibitem{PP} Pareschi, L., Perthame, B. (1996). A Fourier spectral method for homogeneous Boltzmann equations. Transport Theory and Statistical Physics, 25(3-5), 369-382.

\bibitem{PR} Pareschi, L.,  Russo, G. (2000). Numerical solution of the Boltzmann equation I: Spectrally accurate approximation of the collision operator. SIAM journal on numerical analysis, 37(4), 1217-1245.

\bibitem{PRT} Pareschi, L., Russo, G.,  Toscani, G. (2000). Fast spectral methods for the Fokker–Planck–Landau collision operator. Journal of Computational Physics, 165(1), 216-236.



\bibitem{QR} Qiu, J. M., Russo, G. (2017). A high order multi-dimensional characteristic tracing strategy for the Vlasov–Poisson system. Journal of Scientific Computing, 71(1), 414-434.


\bibitem{QS} Qiu, J. M., Shu, C. W. (2011). Conservative semi-Lagrangian finite difference WENO formulations with applications to the Vlasov equation. Communications in Computational Physics, 10(4), 979-1000.





\bibitem{RS} Rossmanith, J. A.,  Seal, D. C. (2011). A positivity-preserving high-order semi-Lagrangian discontinuous Galerkin scheme for the Vlasov–Poisson equations. Journal of Computational Physics, 230(16), 6203-6232.

\bibitem{SKT} Sandberg, R. T., Krasny, R., Thomas, A. G. (2025). The FARSIGHT Vlasov-Poisson code. Journal of Computational Physics, 523, 113664.


\bibitem{SWW} Shao, S., Wang, Y., Wu, J. (2026). Solving Vlasov-Poisson system with an adaptive Hermite spectral method. arXiv preprint arXiv:2605.17820.


\bibitem{SRBG} Sonnendrücker, E., Roche, J., Bertrand, P.,  Ghizzo, A. (1999). The semi-Lagrangian method for the numerical resolution of the Vlasov equation. Journal of Computational physics, 149(2), 201-220.


\bibitem{ZHCQ} Zheng, N., Hayes, D., Christlieb, A., and Qiu, J. M. (2025). A semi-Lagrangian adaptive-rank (SLAR) method for linear advection and nonlinear Vlasov-Poisson system. Journal of Computational Physics, 532, 113970.

\end{thebibliography}
			
\end{document}